%% file: main.tex
\documentclass[reqno,11pt]{amsart}

\input{preamble.tex}

\title[Large sieve for exceptional Maass forms and the greatest prime factor of $n^2+1$]{Large sieve inequalities for exceptional Maass forms and the greatest prime factor of $n^2+1$}
\author[Alexandru Pascadi]{Alexandru Pascadi}
\address{Mathematisches Institut, Endenicher Allee 60, 53115 Bonn, Germany}
\email{alexpascadi@gmail.com}

\begin{document}

\begin{abstract}
    We prove new large sieve inequalities for the Fourier coefficients $\rho_{j\mathfrak{a}}(n)$ of exceptional Maass forms of a given level, weighted by sequences $(a_n)$ with sparse Fourier transforms---including two key types of sequences that arise in the dispersion method. These give the first savings in the exceptional spectrum for the critical case of sequences as long as the level, and lead to improved bounds for various multilinear forms of Kloosterman sums.

    As an application, we show that the greatest prime factor of $n^2+1$ is infinitely often greater than $n^{1.3}$, improving Merikoski's previous threshold of $n^{1.279}$. We also announce applications to the exponents of distribution of primes and smooth numbers in arithmetic progressions.
\end{abstract}

\maketitle

\vspace{-0.7cm}
{
\setlength{\parskip}{0em}
\setcounter{tocdepth}{1}
\tableofcontents
}
\vspace{-1cm}

\section{Introduction} \label{sec:intro}

Let $m, n, c \in \Z$ with $c \ge 1$, and consider the classical Kloosterman sums
\begin{equation} \label{eq:kloosterman}
    S(m, n; c) := \sum_{x \in (\Z/c\Z)^\times} e\left(\frac{mx + n\bar{x}}{c}\right),
\end{equation}
where $e(\alpha) := \exp(2\pi i \alpha)$ and $x\bar{x} \equiv 1 \pmod{c}$. A great number of results in analytic number theory, particularly on the distribution of primes \cite{bombieri1986primes,maynard2025primes,maynard2025primes2,maynard2025primes3,deshouillers1982greatest,de2020niveau,merikoski2023largest,lichtman2023primes} and properties of Dirichlet $L$-functions \cite{deshouillers1982power,deshouillers1984power,watt1995kloosterman,young2011fourth,drappeau2023one,topacogullari2018shifted}, rely on bounding exponential sums of the form
\begin{equation} \label{eq:sum-of-kloosterman}
    \sum_{m \sim M} a_m \sum_{n \sim N} b_n
    \sum_{(c, r) = 1} g\left(\frac{c}{C}\right) S(m\bar{r}, \pm n; sc),
\end{equation}
where $(a_m)$ and $(b_n)$ are rough sequences of complex numbers, $g$ is a compactly-supported smooth function, and $r, s$ are coprime positive integers. One can often (but not always \cite{merikoski2023largest,de2020niveau,maynard2025primes}) leverage some additional averaging over $r$ and $s$, if one of the sequences $(a_m), (b_n)$ is independent of $r, s$.

Estimates for sums like \cref{eq:sum-of-kloosterman} are typically obtained via the spectral theory of automorphic forms \cite{iwaniec2021spectral,iwaniec1997topics}, following Deshouillers--Iwaniec \cite{deshouillers1982kloosterman}; this allows one to bound \cref{eq:sum-of-kloosterman} by certain averages of the sequences $(a_m)$, $(b_n)$ with the Fourier coefficients of automorphic forms for $\Gamma_0(rs)$. Often in applications, the limitation in these bounds comes from our inability to rule out the existence of \emph{exceptional Maass cusp forms}, corresponding to exceptional eigenvalues $\lambda \in (0, \tfrac{1}{4})$ of the hyperbolic Laplacian. This is measured by a parameter $\theta = \max_\lambda \max(0, \tfrac{1}{4} - \lambda)^{1/2}$; under \emph{Selberg's eigenvalue conjecture} there would be no exceptional eigenvalues \cite{selberg1965estimation}, so one could take $\theta = 0$. But unconditionally, the record is Kim--Sarnak's bound $\theta \le \tfrac{7}{64}$, based on the automorphy of symmetric fourth power $L$-functions \cite[Appendix 2]{kim2003functoriality}. 

This creates a power-saving gap between the best conditional and unconditional results in various arithmetic problems, for example, on the prime factors of quadratic polynomials \cite{de2020niveau,merikoski2023largest}, the exponents of distribution of primes \cite{lichtman2023primes} and smooth numbers \cite{pascadi2025smooth} in arithmetic progressions, and low-lying zeros of Dirichlet $L$-functions \cite{drappeau2023one}. 
Improvements to the dependency on $\theta$, which help narrow this gap, come from large sieve inequalities for the Fourier coefficients of exceptional Maass cusp forms (see \cite[Theorems 5, 6, 7]{deshouillers1982kloosterman} and their optimizations in \cite{drappeau2017sums,assing2021uniform,lichtman2023primes,pascadi2025smooth}), which function as weak on-average substitutes for Selberg's eigenvalue conjecture. However, in the key setting of fixed $r, s$ and sequences $(a_n)$ of length $N \approx rs$, no such savings were previously available.

Luckily, for many of the most important applications, we don't need to handle \cref{eq:sum-of-kloosterman} for completely arbitrary sequences, but only for those arising from variations of Linnik's dispersion method \cite{linnik1963dispersion,friedlander1985incomplete,bombieri1986primes,bombieri1987primes2,bombieri1989primes3}; these often have the rough form
\begin{equation} \label{eq:dispersion-coeffs}
    a_m = e(m\alpha) \qquad\qquad \text{and} \qquad\qquad b_n = \sum_{\substack{h_1, h_2 \sim H \\ h_1\ell_1 - h_2\ell_2 = n}} 1,
\end{equation}
for $\alpha \in \R/\Z$ and $\ell_1 \asymp \ell_2 \gg H$ with $(\ell_1, \ell_2) = 1$. Our main results in this paper are new large sieve inequalities for such sequences, with Fourier transforms that obey strong concentration conditions. These are obtained by combining the framework of Deshouillers--Iwaniec with combinatorial ideas---specifically, with new estimates for bilinear sums of Kloosterman sums, stemming from a counting argument of Cilleruelo--Garaev \cite{cilleruelo2011concentration}. 
The resulting improved bounds for \cref{eq:sum-of-kloosterman} can then feed through to the strongest results on several well-studied arithmetic problems.

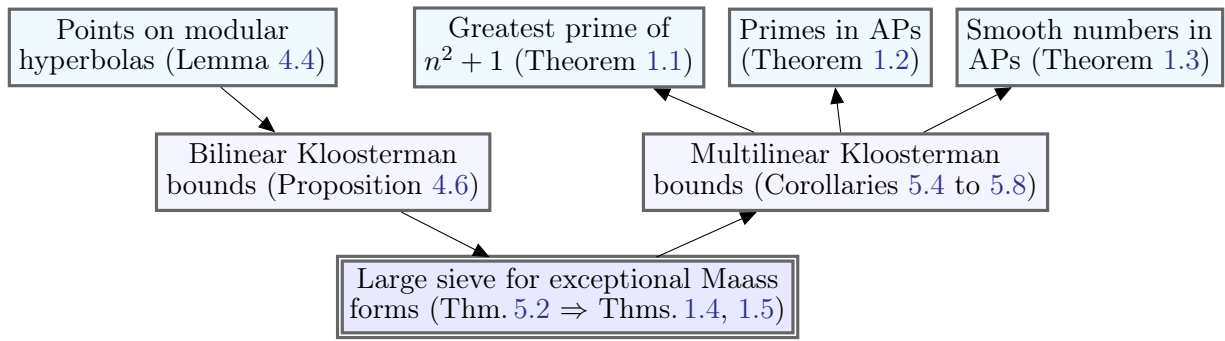
\begin{figure}[ht]
\centering
\begin{tikzpicture}[
square/.style={rectangle, draw=black!60, fill=white, very thick, minimum size=5mm},
important/.style={rectangle, double, draw=black!60, fill=white, very thick, minimum size=5mm},
]
%NODES
\node[square, fill=cyan!5] (concentration)
    {\twoline{Points on modular}{hyperbolas (\cref{lem:cilleruelo-garaev})}};
\node[square, fill=blue!4] (bilinear) [xshift=1.75cm][below = 0.4cm of concentration]
    {\twoline{Bilinear Kloosterman}{bounds (\cref{prop:bilinear-freq-concentration})}};
\node[important, fill=blue!9] (large-sieve-exc) [xshift=3cm] [below = 0.4cm of bilinear]
    {\twoline{Large sieve for exceptional Maass}{forms (Thm.\,\ref{thm:large-sieve-freq-concentration} $\Rightarrow$ Thms.\,\ref{thm:large-sieve-expo-phases}, \ref{thm:large-sieve-dispersion-coeffs})}};
\node[square, fill=blue!4] (multilinear) [right = 1.5cm of bilinear]
    {\twoline{Multilinear Kloosterman bounds}{(Corollaries \ref{cor:kloosterman-averaging-nc}, \ref{cor:kloosterman-averaging-mnc}, \ref{cor:kloosterman-averaging-qmnc}, \ref{cor:kloosterman-incomplete})}};
\node[square, fill=cyan!5] (appl-greatest-prime) [right = 6cm of concentration]
    {\twoline{Applications}{(including \cref{thm:appl-greatest-prime})}};
%EDGES
\draw[-triangle 45] (concentration) -- (bilinear);
\draw[-triangle 45] (bilinear) -- (large-sieve-exc);
\draw[-triangle 45] (large-sieve-exc) -- (multilinear);
\draw[-triangle 45] (multilinear) -- (appl-greatest-prime);
\end{tikzpicture}
\caption{Structure of paper (\emph{arrows signify logical implications}).}
\label{fig:outline}
\end{figure}
\FloatBarrier

\Cref{fig:outline} summarizes the results outlined above, which go from \emph{counting problems} (on the top row), to \emph{exponential sums} (middle row), to \emph{automorphic forms} (bottom row), and then backwards. The transition between the first two rows is mostly elementary (using successive applications of Poisson summation, Cauchy--Schwarz, combinatorial decompositions, and/or sieve methods), while the transition between the last two rows uses the Kuznetsov trace formula \cite{kuznetsov1980petersson,deshouillers1982kloosterman}.

Before we dive into the large sieve inequalities, let us motivate our discussion with applications.

\begin{theorem}\label{thm:appl-greatest-prime}
    For infinitely many $n \in \Z_+$, the greatest prime factor of $n^2+1$ is larger than $n^{1.3}$.
\end{theorem}

This result makes progress on a longstanding problem, approximating the famous conjecture that there exist infinitely many primes of the form $n^2+1$. Back in 1967, Hooley \cite{hooley1967greatest} proved the same result with an exponent of $1.1001$, using the Weil bound for Kloosterman sums. In 1982, Deshouillers--Iwaniec \cite{deshouillers1982greatest} used their bounds on multilinear forms of Kloosterman sums \cite{deshouillers1982kloosterman} to improve this substantially, up to an exponent of $1.2024$. More recently, using Kim--Sarnak's bound $\theta \le \tfrac{7}{64}$ \cite[Appendix 2]{kim2003functoriality}, de la Bret\`eche and Drappeau \cite{de2020niveau} optimized the exponent to $1.2182$. Finally, Merikoski \cite{merikoski2023largest} proved a new bilinear estimate (still relying on the bounds of Deshouillers--Iwaniec \cite{deshouillers1982kloosterman}), and used Harman's sieve to reach the exponent $1.279$; assuming Selberg's eigenvalue conjecture, Merikoski also reached the conditional exponent $1.312$. With our new large sieve inequalities (\cref{thm:large-sieve-expo-phases,thm:large-sieve-dispersion-coeffs}), we can improve the arithmetic information due to both Merikoski \cite{merikoski2023largest} and de la Bret\`eche--Drappeau \cite{de2020niveau}, leading to the unconditional result in \cref{thm:appl-greatest-prime}. As in \cite{merikoski2023largest,de2020niveau}, by
adapting our proof, it should be possible to obtain similar results for other irreducible quadratic polynomials.

Additionally, we announce applications to the distribution of primes and smooth numbers in arithmetic progressions to large moduli. In \cite{pascadi2025exponents}, the author will show that the primes have \emph{exponent of distribution} $5/8-\eps$ using ``triply-well-factorable'' weights $(\lambda_q)$ \cite{maynard2025primes2}, in the sense that 
\[
    \sum_{\substack{q \le x^{5/8-\eps} \\ (q, a) = 1}} 
    \lambda_q \left(\pi(x; q, a) - \frac{\pi(x)}{\varphi(q)}\right) \ll_{\eps,A,a}
    \frac{x}{(\log x)^A},
\]
where $\pi(x; q, a)$ denotes the number of primes up to $x$ which are congruent to $a$ mod $q$. A similar result, with the same exponent of $5/8-\eps$, will be established for smooth numbers, using arbitrary $1$-bounded weights $(\lambda_q)$. These will improve results of Maynard \cite{maynard2025primes2} and Lichtman \cite{lichtman2023primes}, respectively Drappeau \cite{drappeau2015theoremes} and the author \cite{pascadi2025smooth}. Notably, our large sieve inequalities will suffice to completely eliminate the dependency on Selberg's eigenvalue conjecture in these cases.

We also note that an extension of our large sieve inequalities to Maass forms with a general nebentypus should have consequences to counting smooth values of irreducible quadratic polynomials \cite{de2020niveau,harman2008values,harmantwo2024} (by improving de la Bret\`eche--Drappeau's \cite[Th\'eor\`eme 5.2]{de2020niveau}), and to enlarging the Fourier support in one-level density estimates for Dirichlet $L$-functions \cite{drappeau2023one}.

%\begin{remark}
%The first version of this paper's preprint (with identifier \text{arXiv:2404.04239v1}) included the results mentioned above; we have now left these results to a separate paper to improve presentation.
%\end{remark}

\subsection{The large sieve inequalities} \label{subsec:intro-large-sieve}

We now turn to our main technical results. The sums of Kloosterman sums from \cref{eq:sum-of-kloosterman} are related to the Fourier coefficients of $\GL_2$ automorphic forms of \emph{level} $q = rs$ by the Kuznetsov trace formula \cite{kuznetsov1980petersson,deshouillers1982kloosterman} for the congruence group $\Gamma_0(q)$.

More precisely, the spectral side of the Kuznetsov formula contains three terms, corresponding to the contribution of holomorphic forms, Maass forms, and Eisenstein series. The \emph{exceptional} Maass forms are eigenfunctions of the hyperbolic Laplacian on $L^2(\Gamma_0(q) \backslash \H)$ with eigenvalues $0 < \lambda < 1/4$; this (conjecturally empty) exceptional spectrum typically produces losses of the form $X^{\theta(q)}$, where $X$ is a large parameter and $\theta(q) := \max_\lambda \sqrt{\max(0, \tfrac{1}{4} - \lambda)}$. The aforementioned large sieve inequalities for exceptional Maass forms can help alleviate this loss, by incorporating factors of $X^{\theta}$.
Below we state a known result for general sequences $(a_n)$ (the values $X \in \{1, q/N\}$ corresponding to \cite[Theorems 2 and 5]{deshouillers1982kloosterman}), which we aim to improve; we detail our notation in \cref{sec:notation}.

\begin{knowntheorem}[Large sieve with general sequences \cite{deshouillers1982kloosterman}] \label{thm:large-sieve-general}
Let $\eps > 0$, $X > 0$, $N \ge 1/2$, and $(a_n)_{n \sim N}$ be a complex sequence. Let $q \in \Z_+$, $\ma$ be a cusp of $\Gamma_0(q)$ with $\mu(\ma) = q^{-1}$, and $\sigma_{\ma} \in \PSL_2(\R)$ be a scaling matrix for $\ma$. Consider an orthonormal basis of Maass cusp forms for $\Gamma_0(q)$, with eigenvalues $\lambda_j$ and Fourier coefficients $\rho_{j\ma}(n)$ around the cusp $\ma$ (via $\sigma_{\ma}$). Then with $\theta_j := \sqrt{\tfrac{1}{4} - \lambda_j}$, one has 
\begin{equation} \label{eq:large-sieve-general}
    \sum_{\lambda_j < 1/4}
    X^{2\theta_j} 
    \left\vert 
    \sum_{n \sim N} a_n\, \rho_{j\ma}(n)
    \right\vert^2 
    \ll_\eps 
    (qN)^\eps
    \left(1 + \frac{N}{q}\right) \|a_n\|_2^2,
\end{equation}
for any
\begin{equation} \label{eq:basic-X-range}
    X \ll \max\left(1, \frac{q}{N}, \frac{q^2}{N^3}\right).
\end{equation}
\end{knowntheorem}

\begin{remark}
As in \cite{maynard2025primes,pascadi2025smooth,lichtman2023primes}, we use Deshouillers--Iwaniec's normalization \cite{deshouillers1982kloosterman} for the Fourier coefficients $\rho_{j\ma}(n)$ of Maass forms. In various other works \cite{topacogullari2018shifted,drappeau2017sums,de2020niveau,merikoski2023largest}, $\rho_{j\ma}(n)$ are rescaled by $n^{-1/2}$.
\end{remark}

\begin{remark}
An equivalent (and more common \cite{deshouillers1982kloosterman,drappeau2017sums}) way to phrase results like \cref{thm:large-sieve-general} is that
\[
    \sum_{\lambda_j < 1/4}
    X^{2\theta_j} 
    \left\vert 
    \sum_{n \sim N} a_n\, \rho_{j\ma}(n)
    \right\vert^2 
    \ll_\eps 
    (qN)^\eps 
    \left(1 + \frac{X}{X_0}\right)^{2\theta(q)}
    \left(1 + \frac{N}{q}\right) \|a_n\|_2^2,
\]
for \emph{any} $X > 0$, and $X_0 = X_0(N, q)$ given by the right-hand side of \cref{eq:basic-X-range}. We prefer to state our large sieve inequalities in terms of the maximal value of $X$ which does not produce any losses in the right-hand side, compared to the regular spectrum (i.e., $X \ll X_0$). We note that in applications, one usually has $\sqrt{q} \ll N \ll q$, and the best choice in \cref{eq:basic-X-range} for this range is $X \asymp q/N$. But in the critical range $N \asymp q$, \cref{thm:large-sieve-general} is as good as the large sieve inequalities for the full spectrum \cite[Theorem 2]{deshouillers1982kloosterman}, since the limitation $X \ll 1$ forestalls any savings in the $\theta$-aspect.

When some averaging over levels $q \sim Q$ is available, $\ma = \infty$, and $(a_n)$, $\sigma_\infty$ are independent of $q$, Deshouillers--Iwaniec \cite[Theorem 6]{deshouillers1982kloosterman} improved the admissible range to $X \ll \max(1, (Q/N)^2)$; Lichtman \cite{lichtman2023primes} recently refined this to $X \ll \max(1, \min((Q/N)^{32/7}, Q^2/N))$, by making $\theta$-dependencies explicit in \cite[\S 8.2]{deshouillers1982kloosterman}. We note that these results are still limited at $X \ll 1$ when $N \asymp Q$.
\end{remark}

Although it seems difficult to improve \cref{thm:large-sieve-general} in general (see \cref{subsec:overview-general-sequences}), one can hope to do better for special sequences $(a_n)$; for instance, the last term in \cref{eq:basic-X-range} can be improved if the sequence $(a_n)$ is sparse. In this paper, we consider the ``dual'' setting when $(a_n)$ is sparse in frequency space, i.e., when the Fourier transform $\hat{a}(\xi) := \sum_n a_n\, e(-n\xi)$ is concentrated on a subset of $\R/\Z$. We give a general result of this sort in \cref{thm:large-sieve-freq-concentration}, which also depends on rational approximations to the support of $\hat{a}$. Below we state the two main cases of interest, corresponding to the sequences in \cref{eq:dispersion-coeffs} (we also incorporate a scalar $a$ in the Fourier coefficients, but on a first read one should take $a = 1$).

\begin{theorem}[Large sieve with exponential phases] \label{thm:large-sieve-expo-phases}
Let $\eps, X > 0$, $N \ge 1/2$, $\alpha \in \R/\Z$, and $q, a \in \Z_+$. Then with the notation of \cref{thm:large-sieve-general} and the choice of scaling matrix in \cref{eq:scaling-choices}, the bound
\begin{equation} \label{eq:large-sieve-expo}
    \sum_{\lambda_j < 1/4}
    X^{2\theta_j} 
    \left\vert 
    \sum_{n \sim N} e(n\alpha)\, \rho_{j\ma}(an)
    \right\vert^2 
    \ll_\eps 
    (qaN)^\eps
    \left(1 + \frac{aN}{q}\right) N
\end{equation}
holds for all
\begin{equation} \label{eq:expo-X-range}
        X \ll \frac{\max\left(N, \frac{q}{a}\right)}{\min_{t \in \Z_+} \left(t + N\|t\alpha\| \right)}.
\end{equation}
In particular, this implies the range $X \ll \max(\sqrt{N}, \tfrac{q}{a\sqrt{N}})$, uniformly in $\alpha$ and $\sigma_{\ma}$.
The same result holds if $e(n\alpha)$ is multiplied by $\Phi(n/N)$, for any smooth function  $\Phi : (0, 4) \to \C$ with $\Phi^{(j)} \ll_j 1$.
\end{theorem}

Here, $\|\alpha\|$ denotes the distance from $\alpha$ to $0$ inside $\R/\Z$; the fact that the worst (``minor-arc'') range covered by \cref{eq:expo-X-range} is $X \ll \max(\sqrt{N}, \tfrac{q}{a\sqrt{N}})$ follows from a pigeonhole argument. The best range, $X \ll \max(N, \tfrac{q}{a})$, is achieved when $\alpha$ is $O(N^{-1})$ away from a rational number with bounded denominator.
In particular, \cref{thm:large-sieve-expo-phases} obtains significant savings in the $\theta$-aspect in the critical case $N \asymp q$, for an individual level $q$, which was previously impossible to the best of our knowledge.

\begin{remark}
As detailed in \cref{subsec:automorphic}, altering the scaling matrix $\sigma_{\ma}$ in bounds like \cref{eq:large-sieve-expo} is equivalent to altering the phase $\alpha$; the canonical choice in \cref{eq:scaling-choices} leads to several simplifications in practice.
\end{remark}

When $a = 1$, $\ma = \infty$, and $\alpha$ is independent of $q$, Deshouillers--Iwaniec \cite[Theorem 7]{deshouillers1982kloosterman} showed that the bound in \cref{eq:large-sieve-expo} holds on average over levels $q \sim Q$ in the larger range $X \ll \max(N, Q^2/N)$. In this on-average setting, we also mention the large sieve inequality of Watt \cite[Theorem 2]{watt1995kloosterman}, which saves roughly $X = Q^2/N^{3/2}$ when $a_n$ is a smoothed divisor-type function.

For the second sequences mentioned in \cref{eq:dispersion-coeffs}, we state a bound which also incorporates exponential phases $e(h_i\alpha_i)$. The reader should keep in mind the case of parameter sizes $N \asymp HL$, $H \asymp L$, and $\alpha_i = 0$, when the $X$-factor saved below can be as large as $\max(\sqrt{N}, \tfrac{q}{a\sqrt{N}})$.

\begin{theorem}[Large sieve with dispersion coefficients] \label{thm:large-sieve-dispersion-coeffs}
Let $\eps, X > 0$, $N \ge 1/2$, $L, H \gg 1$, $\alpha_1, \alpha_2 \in \R/\Z$, and $q, a, \ell_1, \ell_2 \in \Z_+$ satisfy $\ell_1, \ell_2 \asymp L$, $(\ell_1, \ell_2) = 1$. Consider the sequence $(a_n)_{n \sim N}$ given by
\[
    a_n := \sum_{\substack{h_1, h_2 \in \Z \\ h_1 \ell_1 - h_2 \ell_2 = n}} \Phi_1\left(\frac{h_1}{H}\right) \Phi_2\left(\frac{h_2}{H}\right) e(h_1 \alpha_1 + h_2 \alpha_2),
\]
where $\Phi_i : (-\infty, \infty) \to \C$ are smooth functions supported in $(-O(1),O(1))$, with $\Phi_i^{(j)} \ll_j 1$, $\forall j \ge 0$. Then with the notation of \cref{thm:large-sieve-general} and the choice of scaling matrix in \cref{eq:scaling-choices}, if $q \gg L^2$, one has
\begin{equation} \label{eq:large-sieve-disp}
    \sum_{\lambda_j < 1/4}
    X^{2\theta_j} 
    \left\vert 
    \sum_{n \sim N} a_n\, \rho_{j\ma}(an)
    \right\vert^2 
    \ll_\eps 
    (qaH)^{\eps}
    \left(1 + \frac{aN}{q}\right)
    \left(\|a_n\|_2^2 + \gcd(a, q) N \left(\frac{H}{L} + \frac{H^2}{L^2}\right)\right),
\end{equation}
whenever 
\begin{equation} \label{eq:disp-X-range}
    X \ll \max\left(1, \frac{q}{aN}\right) \max\left(1, \frac{NH}{(H+L)LM}\right),
    \qquad\quad 
    M := \min_{\substack{t \in \Z_+ \\ i \in \{1, 2\}}} \left(t + \frac{N}{L} \|t\alpha_i\|\right).
\end{equation}
\end{theorem}

\begin{remark}
In \cref{thm:large-sieve-dispersion-coeffs}, when $N \asymp HL$ and $\alpha_i = 0$, the norm $\|a_n\|_2^2$ is on the order of $N(\frac{H}{L} + \frac{H^2}{L^2})$. So in this setting, which is the limiting case for our applications, the right-hand side of \cref{eq:large-sieve-disp} produces no important losses over the regular-spectrum bound of $(qN)^\eps (1 + \tfrac{aN}{q})\, \|a_n\|_2^2$.
\end{remark}

\begin{remark}
Some instances of the dispersion method \cite{drappeau2023one,drappeau2017sums,assing2021uniform} use coefficients roughly of the shape
\begin{equation} \label{eq:third-type-disp}
    b_n = \sum_{\substack{h \sim H \\ h(\ell_1-\ell_2) = n}} 1,
\end{equation}
where $\ell_1 \asymp \ell_2 \gg H$, $\ell_1 \neq \ell_2$, and the level is $q = \ell_1 \ell_2$. Although these resemble the second sequence from \cref{eq:dispersion-coeffs} (treated by \cref{thm:large-sieve-dispersion-coeffs}), one should actually handle this case using \cref{thm:large-sieve-expo-phases}, with $\alpha = 0$, $N = H$, and $a = |\ell_1 - \ell_2|$. In particular, for these ranges we have $aN = |\ell_1-\ell_2|H \ll \ell_1\ell_2 = q$, so the $1$-term in the right-hand side of \cref{eq:large-sieve-expo} is dominant, and the range in \cref{eq:expo-X-range} becomes $X \ll \ell_1\ell_2/|\ell_1 - \ell_2|$.
\end{remark}

\begin{remark}
For simplicity, we state and prove our results in the setting of arbitrary bases of classical Maass forms, following the original work of Deshouillers--Iwaniec \cite{deshouillers1982kloosterman}. However, our work should admit two independent extensions, which are relevant for some applications. The first is handling Maass forms with a nontrivial nebentypus, following Drappeau \cite{drappeau2017sums}; this leads to bounds for sums like \cref{eq:sum-of-kloosterman} with $c$ restricted to an arithmetic progression. The second is considering Hecke--Maass forms which are exceptional with respect to the Ramanujan--Petersson conjecture at finite places, the non-Archimedean analogue of Selberg's conjecture; this should improve the dependency on the scalar $a$ when $aN > q$. One could either follow Assing--Blomer--Li \cite{assing2021uniform} to `factor out' $a$ from $\rho_{j\ma}(an)$ (and apply Kim--Sarnak's bound \cite{kim2003functoriality} at places dividing $a$ before using our large sieve inequalities), or treat the exceptional forms at places dividing $a$ similarly to the Archimedean case, to match the regular-spectrum bound whenever $aX$ is at most a function of $q$ and $N$ (this option should work better when $a$ is well-factorable).
\end{remark}

\subsection{Acknowledgements}

The author is grateful to his PhD advisor, James Maynard, for his kind guidance, to Sary Drappeau for many thoughtful discussions, and to Jori Merikoski, Lasse Grimmelt, Jared Duker Lichtman, and the referees for helpful comments and suggestions. For the duration of this project, the author was sponsored by the EPSRC Scholarship at University of Oxford.

\section{Informal overview} \label{sec:overview}

Let us summarize the key ideas behind our work, ignoring a handful of technical details such as smooth weights, GCD constraints, or keeping track of $x^{o(1)}$ factors. %We will use the notation `$\lesssim$' to emphasize that the bounds in this section are not precise.

\subsection{Large sieve with general sequences} \label{subsec:overview-general-sequences} 

Let $q \in \Z_+$ and consider the simplified version
\begin{equation} \label{eq:large-sieve-informal}
    \sum_{\lambda_j < 1/4} X^{2\theta_j} \left\vert \sum_{n \sim N} a_n\, \rho_{j\infty}(n) \right\vert^2 \lesssim \left(1 + \frac{N}{q}\right) \|a_n\|_2^2
\end{equation}
of the large sieve inequality from \cref{thm:large-sieve-general}, for $\ma = \infty$, ignoring $(qN)^{o(1)}$ factors. Here $(a_n)$ are arbitrary complex coefficients, and the reader may pretend that $|a_n| \approx 1$ for each $n$, so that $\|a_n\|_2^2 \approx N$. Such an inequality follows from \cite[Theorem 2]{deshouillers1982kloosterman} when $X = 1$, but we need larger values of $X$ to temper the contribution of exceptional eigenvalues. The Kuznetsov trace formula \cite{kuznetsov1980petersson} in \cref{prop:kuznetsov}, combined with large sieve inequalities for the regular spectrum \cite[Theorem 2]{deshouillers1982kloosterman}, essentially reduces the problem to bounding (a smoothed variant of) the sum
\begin{equation} \label{eq:kloosterman-informal}
    \sum_{\substack{c \sim NX \\ c \equiv 0 \pmod{q}}} \frac{1}{c} \sum_{m \sim N} \bar{a_m} \sum_{n \sim N} a_n\, S(m, n; c)
\end{equation}
by the same amount as in the right-hand side of \cref{eq:large-sieve-informal}---see \cref{cor:exc-bound} for a formal statement in this direction. The left-hand side vanishes for $X < q/(2N)$, so we immediately obtain \cref{eq:large-sieve-informal} for $X \ll q/N$, which is the content of \cite[Theorem 5]{deshouillers1982kloosterman}. Alternatively, we can plug in the pointwise Weil bound for $S(m, n; c)$ and apply Cauchy--Schwarz, to obtain an upper bound of roughly
\begin{equation} \label{eq:sqrt-growth-weil}
    \frac{NX}{q} \frac{1}{NX} N \|a_n\|_2^2 \sqrt{NX} =
    \frac{N^{3/2} X^{1/2}}{q} \|a_n\|_2^2.
\end{equation}
This is acceptable in \cref{eq:large-sieve-informal} provided that $X \le q^2/N^3$, which completes the range from \cref{thm:large-sieve-general}. 

Improving the range $X \le \max(1, q/N, q^2/N^3)$ turns out to be quite difficult. Indeed, it is not clear how to exploit the averaging over $c$ without the Kuznetsov formula, so any savings are more likely to come from bounding bilinear forms of Kloosterman sums $\sum_{m \sim N} a_m \sum_{n \sim N} b_n\, S(m, n; c)$; this is a notoriously hard problem for general sequences $(a_m), (b_n)$ \cite{kowalski2017bilinear,kowalski2020stratification,kerr2023bounds,xi2018ternary}. For example, an extension of the work of Kowalski--Michel--Sawin \cite{kowalski2017bilinear} to general moduli should improve \cref{thm:large-sieve-general} in the critical range $q \approx N^2$, but even then the final numerical savings would be relatively small.

The other critical case encountered in applications is $q \approx N$, where \cref{thm:large-sieve-general} gives no non-trivial savings in the $\theta$-aspect (i.e., $X \ll 1$), and where such savings should in fact be impossible for general sequences $(a_n)$. Indeed, we expect $|\rho_j(n)|$ to typically be of size $\approx q^{-1/2}$, so by picking $a_n = q\, \bar{\rho_1(n)}$, the left-hand side of \cref{eq:large-sieve-informal} is at least $X^{2\theta(q)} N^2$, while the right-hand side is $(1 + \frac{N}{q}) qN$; this limits the most optimistic savings for general sequences at $X = (1 + \frac{q}{N})^{1/(2\theta(q))}$.

The key idea in our work is to make use of the special structure of the sequences $(a_n)$ which show up in variations of the dispersion method \cite{linnik1963dispersion}. Often, such sequences have sparse Fourier transforms, and using Fourier analysis on the corresponding exponential sums leads to a combinatorial problem.

\subsection{Exponential phases and a counting problem} 
\label{subsec:overview-counting}
Let us focus on the case $a_n = e(n\alpha)$, for some $\alpha \in [0, 1)$. Expanding the Kloosterman sums from \cref{eq:kloosterman-informal} and Fourier-completing in $m, n$ leads to a variant of the identity
\begin{equation} \label{eq:Kloosterman-expo-informal}
    \sum_{m \sim N} e(-m\alpha) \sum_{n \sim N} e(n\alpha)\, S(m, n; c)
    \approx 
    N^2 \sum_{\substack{|x - c\alpha| \le c/N \\ |y + c\alpha| \le c/N}} e\left(\frac{N(x+y)}{c} \right) \one_{xy \equiv 1 \pmod{c}}.
\end{equation}
Taking absolute values and ignoring the outer averaging over $c$, we are left with the task of bounding
\begin{equation} \label{eq:combinatorial-counts}
    \sum_{\substack{|x - c\alpha| \le X \\ |y + c\alpha| \le X}} \one_{xy \equiv 1 \pmod{c}},
\end{equation}
for $c \sim NX$,
which is just a count of points on a modular hyperbola in short intervals (as considered in \cite{cilleruelo2011concentration}). When $\alpha = 0$, one can directly use the divisor bound to write
\[
    \sum_{\substack{|x|, |y| \le X}} \one_{xy \equiv 1 \pmod{c}}
    =
    \sum_{|z| \le \frac{X^2}{c}}\, \sum_{|x|, |y| \le X} \one_{xy = cz+1}
    \lesssim 
    \frac{X^2}{c} + 1,
\]
up to a factor of $X^{o(1)}$, which leads to a variant of
\[
    \sum_{m,n \sim N} S(m, n; c) \lesssim c + N^2 = NX + N^2.
\]
(This type of bound was also observed by Shparlinski and Zhang \cite{shparlinski2016cancellations}.)
Overall, we roughly obtain
\begin{equation} \label{eq:final-Kloosterman-informal}
    \sum_{\substack{c \sim NX \\ c \equiv 0 \pmod{q}}} \frac{1}{c} \sum_{m, n \sim N} S(m, n; c)
    \lesssim
    \frac{NX + N^2}{q},
\end{equation}
which is at most $(1 + \frac{N}{q}) N$, as required in \cref{eq:large-sieve-informal}, provided that
\[
    X \le \max(N, q).
\]
This gives the best-case range from \cref{eq:expo-X-range} (when $a = 1$). The analogue of this argument for other values of $\alpha \in \R/\Z$ depends on the quality of the best rational approximations to $\alpha$, due to a rescaling trick of Cilleruelo--Garaev \cite{cilleruelo2011concentration}. For an arbitrary value of $\alpha$, a pigeonhole argument (Dirichlet approximation) leads to a bound of the shape
\begin{equation} \label{eq:final-Kloosterman-informal-2}
    \sum_{\substack{c \sim NX \\ c \equiv 0 \pmod{q}}} \frac{1}{c} \sum_{m \sim N} e(-m\alpha) \sum_{n \sim N} e(n\alpha)\, S(m, n; c)
    \lesssim
    \frac{N^{3/2} X + N^2}{q},
\end{equation}
and ultimately to the range $X \le \max(\sqrt{N}, q/\sqrt{N})$, which is the worst (and average) case in \cref{eq:expo-X-range} when $a = 1$. Incorporating a scalar $a$ inside $\rho_{j\infty}(an)$ is not too difficult, since a similar argument handles the analogous bilinear sums of $S(am, an; c)$, up to a loss of $\gcd(a, c)$. 

\begin{remark}
A consequence of not leveraging the exponential phases in the right-hand side of \cref{eq:Kloosterman-expo-informal} is that the same argument extends to sums over $|m|, |n| \le N$. In particular, the term $m = n = 0$ already gives a contribution of about $c \asymp NX$, which produces a term of $NX/q$ in \cref{eq:final-Kloosterman-informal} with a linear growth in $X$ (as opposed to the square-root growth from \cref{eq:sqrt-growth-weil}, coming from the Weil bound).
\end{remark}

\subsection{Sequences with frequency concentration} \label{subsec:overview-freq-conc}

It will probably not come as a surprise that one can extend the preceeding discussion by Fourier-expanding other sequences $(a_n)$, given a strong-enough concentration condition for their Fourier transforms, but there are some subtleties in how to do this optimally. If $a_n = \check{\mu}(n) = \int_{\R/\Z} e(n\alpha)\, d\mu(\alpha)$ for all $n \sim N$ and some bounded-variation complex measure $\mu$, then there are at least two ways to proceed---depending on whether the integral over $\alpha$ is kept inside or outside of the square. 

Indeed, by applying Cauchy--Schwarz in $\alpha$ and our \cref{thm:large-sieve-expo-phases} for exponential phases as a black-box, one can directly obtain a bound like
\begin{equation} \label{eq:large-sieve-informal-fourier}
    \sum_{\lambda_j < 1/4} X^{2\theta_j} \left\vert \sum_{n \sim N} a_n\, \rho_{j\ma}(n) \right\vert^2 \lesssim \left(1 + \frac{N}{q}\right) N |\mu|(\R/\Z)^2,
\end{equation}
for all $X \le \max(\sqrt{N}, q/\sqrt{N})$ (and this range can be slightly improved given more information about the support of $\mu$ near rational numbers of small denominators). Unfortunately, this replaces the norm $\|a_n\|_2$ from \cref{thm:large-sieve-general} with $\sqrt{N} |\mu|(\R/\Z)$, which produces a significant loss unless $\mu$ is very highly concentrated---and it is difficult to make up for this loss through gains of $X^{2\theta}$. %The same happens if one Fourier-expands the sequence $(a_n)$ at the very start, when dealing with multilinear forms of Kloosterman sums, before the Kuznetsov and Cauchy--Schwarz steps.

The alternative approach is to expand the square in the left-hand side of \cref{eq:large-sieve-informal-fourier}, pass to a sum of Kloosterman sums as in \cref{eq:kloosterman-informal} by Kuznetsov, and only then Fourier-expand (two instances of) the sequence $(a_n)$. Using similar combinatorial ideas as for \cref{eq:final-Kloosterman-informal-2}, we can then essentially bound
\begin{equation} \label{eq:final-Kloosterman-informal-3}
    \sum_{\substack{c \sim NX \\ c \equiv 0 \pmod{q}}} \frac{1}{c} \sum_{m \sim N} e(m\alpha) \sum_{n \sim N} e(n\beta)\, S(m, n; c)
    \lesssim
    \frac{N^{5/3} X + N^2}{q},
\end{equation}
for arbitrary values of $\alpha, \beta \in \R/\Z$. With no further information about the support of $\mu$, this ultimately gives a bound like
\[
    \sum_{\lambda_j < 1/4} X^{2\theta_j} \left\vert \sum_{n \sim N} a_n\, \rho_{j\ma}(n) \right\vert^2 \lesssim \left(1 + \frac{N}{q}\right) \|a_n\|_2^2 
    +
    \frac{N^{5/3}X + N^2}{q} |\mu|(\R/\Z)^2,
\]
which is acceptable in \cref{eq:large-sieve-informal}, in particular, whenever $X < N^{1/3}$ and $\sqrt{N} |\mu|(\R/\Z) \le \sqrt{q/N} \|a_n\|_2$.
Compared to the first approach, this generally gains less in the $X$-aspect, but it relaxes the concentration condition on $\mu$ if $N < q$. This second approach turns out to be better for our applications; the resulting large sieve inequality is \cref{thm:large-sieve-freq-concentration}, which particularizes to \cref{thm:large-sieve-expo-phases,thm:large-sieve-dispersion-coeffs}.

What is perhaps more surprising, though, is that strong-enough frequency concentration (i.e., $\sqrt{N} |\mu|(\R/\Z)^2 \le \sqrt{q/N} \|a_n\|_2$) arises in applications, beyond the case of exponential sequences. A key observation is that the aforementioned dispersion coefficients
\begin{equation} \label{eq:disp-coeff-informal}
    a_n = \sum_{\substack{h_1, h_2 \sim H \\ h_1 \ell_1 - h_2 \ell_2 = n}} 1,
\end{equation}
with $\ell_1 \asymp \ell_2 \asymp L$,
come from a convolution of two ``arithmetic progressions'' $\one_{n \equiv 0 \pmod{\ell_i}} \one_{n \sim H\ell_i}$. The Fourier transform of each of these two sequences has $\ell_i$ periodic peaks of height $H$ and width $(H \ell_i)^{-1}$, supported around multiples of $1/\ell_i$. When $(\ell_1, \ell_2) = 1$, multiplying these two Fourier transforms results in cancellation everywhere away from a small number ($\le 1 + \frac{L}{H}$) of rational points (and thus, in frequency concentration on a set of size $\frac{1}{HL} + \frac{1}{H^2}$); see \cref{lem:fourier-dispersion}. %In particular, this is the first instance of the dispersion method that we know of, which leverages the smoothness of the variables $h_1, h_2$ coming from Poisson summation. %Nah, Watt does this too! Even though it's a slightly different context

\subsection{Multilinear forms of Kloosterman sums} \label{subsec:overview-kloosterman}
Consider once again the sums \cref{eq:sum-of-kloosterman}, in the ranges
\[
    M, N \le rs, \qquad\qquad 
    X := \frac{s\sqrt{r} C}{\sqrt{MN}} \ge 1,
\] 
which are relevant for most applications.
An additional use of the Kuznetsov formula, for the level $q = rs$ and the cusps $\infty, 1/s$ (with suitable scaling matrices), gives a variant of the bound
\[
\begin{aligned}
    &\sum_{m \sim M} a_m \sum_{n \sim N} b_n \sum_{(c, r) = 1} g\left(\frac{c}{C}\right) S(m\bar{r}, \pm n; sc)
    \\ 
    &\lesssim 
    s\sqrt{r}C \sum_{\lambda_j < 1/4} X^{2\theta_j} \left\vert \sum_{m \sim M} a_m\, \rho_{j\infty}(m) \right\vert \left\vert \sum_{n \sim N} b_n\, \rho_{j\, 1/s}(n)\right\vert \ + \ \ldots
\end{aligned}
\]
Here we omitted the contribution of the regular Maass forms, Eisenstein series and holomorphic forms (which will not be dominant).
A priori, this arrangement introduces a factor of $X^{2\theta(q)}$ in our bounds, recalling that $\theta(q) = \max_{\lambda_j(q) < 1/4} \theta_j(q)$ (if the maximum is nonempty, and $\theta(q) = 0$ otherwise). However, the value of $X$ in this loss can be decreased through the large sieve inequalities for exceptional Maass forms. Indeed, after splitting $X = X_0 \sqrt{X_1 X_2}$, taking out a factor of only $(1 + X_0)^{2\theta(q)}$, and applying Cauchy--Schwarz, we reach
\[
    s\sqrt{r}C\, (1 + X_0)^{2\theta(q)} 
    \left(\sum_{\lambda_j < 1/4} X_1^{2\theta_j} \left\vert \sum_{m \sim M} a_m\, \rho_{j\infty}(m) \right\vert^2 \right)^{1/2}
    \left(\sum_{\lambda_j < 1/4} X_2^{2\theta_j} \left\vert \sum_{n \sim N} b_n\, \rho_{j\, 1/s}(n) \right\vert^2 \right)^{1/2}.
\]
Above, we can choose $X_1$ and $X_2$ as the maximal values that can be fully incorporated in large sieve inequalities like \cref{eq:large-sieve-informal} without producing losses in the right-hand side, for the specific sequences $(a_m)$ and $(b_n)$. In this case, we roughly obtain a final bound of
\[
    s\sqrt{r} C \left(1 + \frac{s\sqrt{r}C}{\sqrt{MN X_1 X_2}}\right)^{2\theta(q)} 
    \|a_m\|_2\, \|b_n\|_2.
\]
For example, if $a_m = e(m\alpha_{r,s})$ for some $\alpha_{r,s} \in \R/\Z$, then we may take $X_1 = \max(\sqrt{N}, q/\sqrt{N})$ by \cref{thm:large-sieve-expo-phases}, which ultimately saves a factor of $N^{\theta/2}$. Similarly, if $(b_n)$ are of the form in \cref{eq:disp-coeff-informal}, where $H \asymp L \asymp \sqrt{N}$, then by \cref{thm:large-sieve-dispersion-coeffs} we may also take $X_2 = \max(\sqrt{N}, q/\sqrt{N})$. 

If some averaging over $r \sim R, s \sim S$ is available and the sequence $(a_m)$ does not depend on $r, s$, then larger values of $X_1$ are available due to Deshouillers--Iwaniec \cite[Theorems 6, 7]{deshouillers1982kloosterman}. In this setting, if $a_m = e(m\omega)$ for a fixed $\omega \in \R/\Z$, one can combine the essentially-optimal value $X_1 = Q^2/N$ (see \cref{thm:large-sieve-level-avg} below) with our savings in the $X_2$-aspect. Following \cite[Theorem 12]{deshouillers1982kloosterman}, similar estimates can be deduced for multilinear forms of incomplete Kloosterman sums, simply by Fourier-completing them and appealing to the estimates for complete sums; see our \cref{cor:kloosterman-incomplete}. Such bounds feed directly into the dispersion method and its applications, as we shall see in \cref{sec:greatest-prime-factor}.

\subsection{Layout of paper} \label{subsec:layout}

In \cref{sec:notation}, we cover notation and preliminary results, including several key lemmas from the spectral theory of automorphic forms. \cref{sec:combi} only contains elementary arguments, from counting points on modular hyperbolas in \cref{lem:cilleruelo-garaev} (following Cilleruelo--Garaev \cite{cilleruelo2011concentration}), to the bilinear Kloosterman bounds in \cref{prop:bilinear-freq-concentration} (which may be of independent interest to the reader). In \cref{subsec:large-sieve}, we combine these combinatorial inputs with the Deshouillers--Iwaniec setup \cite{deshouillers1982kloosterman} to prove a general large sieve inequality in \cref{thm:large-sieve-freq-concentration}, which can be viewed as our main technical result; we then deduce \cref{thm:large-sieve-expo-phases,thm:large-sieve-dispersion-coeffs} from it. \cref{subsec:multilinear-kloosterman} contains the corollaries of these large sieve inequalities: various bounds for multilinear forms of Kloosterman sums, with improved dependencies on the $\theta$ parameter. Finally, in \cref{sec:greatest-prime-factor} we will use these bounds to prove \cref{thm:appl-greatest-prime}, building on the work of Merikoski \cite{merikoski2023largest} and de la Bret\`eche--Drappeau \cite{de2020niveau}.

\section{Notation and preliminaries} \label{sec:notation}

\subsection{Standard analytic notation} \label{subsec:standard-notation}

We write $\Z, \Q, \R, \C, \H$ for the sets of integers, rational numbers, real numbers, complex numbers, respectively complex numbers with positive imaginary part. We may scale these sets by constants, and may add the subscript $+$ to restrict to positive numbers; so for example $2\Z_+$ denotes the set of even positive integers, while $i\R$ is the imaginary line. For $\alpha \in \R$ (or $\R/\Z$), we denote $e(\alpha) := \exp(2 \pi i \alpha)$, and set $\|\alpha\| := \min_{n \in \Z} |\alpha - n|$, which induces a metric on $\R/\Z$. We write $\Z/c\Z$ for the ring of residue classes modulo a positive integer $c$, $(\Z/c\Z)^\times$ for its multiplicative group of units, and $\bar{x}$ for the inverse of $x \in (\Z/c\Z)^\times$. We may use the latter notation inside congruences, with $x \equiv y\bar{z} \pmod{c}$ meaning that $xz \equiv y \pmod{c}$ (for $\gcd(z, c) = 1$). We may also use the notation $(a, b)$ for $\gcd(a, b)$, and $[a, b]$ for $\lcm(a, b)$, when it is clear from context to not interpret these as pairs or intervals. We write $\one_S$ for the indicator function of a set $S$ (or for the truth value of a statement $S$), $n \sim N$ for the statement that $N < n \le 2N$ (so, e.g., $\one_{n \sim N} = \one_{n \le 2N} - \one_{n \le N}$), and interpret sums like $\sum_{n \sim N}$, $\sum_{n \equiv 0 \pmod{q}}$, or $\sum_{d \mid n}$ with the implied restrictions that $n, d \in \Z_+$. For $n \in \Z_+$, we define the divisor-counting function by $\tau(n) := \sum_{d \mid n} 1$, and Euler's totient function by $\varphi(n) := \sum_{m = 1}^n \one_{(m, n) = 1}$. We say that a complex sequence $(a_n)$ is \emph{divisor-bounded} iff $|a_n| \ll \tau(n)^{O(1)}$. We also write $P^+(n)$ and $P^-(n)$ for the largest and smallest prime factors of a positive integer $n$, and recall that $n$ is called $y$-smooth iff $P^+(n) \le y$.

We use the standard asymptotic notation $f \ll g$, $f \asymp g$, $f = O(g)$, $f = o_{x \to \infty}(g)$ from analytic number theory, and indicate that the implicit constants depend on some parameter $\eps$ through subscripts (e.g., $f \ll_\eps g$, $f = O_\eps(g)$). In particular, one should read bounds like $f(x) \ll x^{o(1)}$ as $\forall \eps > 0,\, f(x) \ll_\eps x^\eps$. Given $\ell \in \Z_+$, we write $f^{(\ell)}$ for the $\ell$th derivative of a function $f : \R \to \C$, and $f^{(0)} = f$. 
For $q \in [1, \infty]$, we denote by $\|f\|_{L^q}$ the $L^q$-norm of a function $f : \R \to \C$ (or $f : \R/\Z \to \C$), and by $\|a\|_q$ (or $\|a_n\|_q$) the $\ell^q$ norm of a sequence $(a_n)$.

We require multiple notations for the Fourier transforms of $L^1$ functions $f, \Phi : \R \to \C$, $\varphi : \R/\Z \to \C$, and $a : \Z \to \C$ (the latter could be, e.g., a finite sequence $(a_n)_{n \sim N}$ extended with zeroes elsewhere). These are given by
\begin{equation} \label{eq:fourier-transforms}
\begin{aligned}
    f : \R \to \C \qquad 
    \leadsto \qquad\,\,
    \hat{f} : \C \to \C, \qquad\quad 
    &\hat{f}(\xi) := \int_\R f(t)\, e(-\xi t)\ dt, 
    \\
    \Phi : \R \to \C \qquad 
    \leadsto \qquad\,
    \check{\Phi} : \C \to \C, \qquad\quad 
    &\check{\Phi}(t) := \int_\R \Phi(\xi)\, e(\xi t)\ d\xi, 
    \\ 
    a : \Z \to \C \quad\ \ \, \,
    \leadsto \ \ \, \,
    \hat{a} : \R/\Z \to \C, \qquad\quad 
    &\hat{a}(\alpha) := \sum_{n \in \Z} a_n\, e(-n\alpha), \qquad 
    \\
    \varphi : \R/\Z \to \C \qquad 
    \leadsto \qquad \
    \check{\varphi} : \Z \to \C, \qquad\quad 
    &\check{\varphi}(n) := \int_{\R/\Z} \varphi(\alpha)\, e(n\alpha)\, d\alpha.
\end{aligned}
\end{equation}
Note that the first two and the last two of these transforms are inverse operations under suitable conditions; in particular, if $\Phi$ is Schwarz, $a$ is $L^1$, and $\varphi$ is smooth (so $\check{\varphi}(n)$ decays rapidly as $|n| \to \infty$), one has
\begin{equation} \label{eq:Fourier-inverses}
    \check{\hat{\Phi}} \Big\vert_\R = \hat{\check{\Phi}} \Big\vert_\R = \Phi, \qquad\qquad
    \check{\hat{a}} = a, \qquad\qquad
    \hat{\check{\varphi}} = \varphi.
\end{equation}

We also denote the Fourier transform of a bounded-variation complex Borel measure $\mu$ on $\R/\Z$ by $\check{\mu}(n) := \int_{\R/\Z} e(n\alpha)\, d\mu(\alpha)$. For instance, one has $\check{\lambda}(n) = \one_{n = 0}$ for the Lebesgue measure $\lambda$, and $\check{\delta}_{A}(n) = \sum_{\alpha \in A} e(n\alpha)$ for the Dirac delta measure on a finite set $A \subset \R/\Z$. Moreover, if $d\mu = \varphi\, d\lambda$ for some $L^1$ function $\varphi : \R/\Z \to \C$, then $\check{\mu} = \check{\varphi}$. Finally, with our notation, the Parseval--Plancherel identity reads $\|a_n\|_2^2 = \|\hat{a}\|_{L^2}^2$ (and $\|f\|_{L^2} = \|\hat{f}\|_{L^2}$), while Poisson summation states that for any Schwarz function $f$,
\begin{equation} \label{eq:Poisson}
    \sum_{n \in \Z} f(n) = \sum_{n \in \Z} \hat{f}(n) = \sum_{n \in \Z} \check{f}(n).
\end{equation}
In practice it will be useful to truncate the Poisson summation formula; we combine this with a smooth dyadic partition of unity and a separation of variables, in the following lemma.

\begin{lemma}[Truncated Poisson with extra steps] \label{lem:truncated-poisson}
Let $x, N, Q \gg 1$ with $N, Q \ll x^{O(1)}$, $q \asymp Q$ be a positive integer, $a \in \Z$ (or $\Z/q\Z$), and $\Phi : (0, \infty) \to \C$ be a smooth function with $\Phi(t)$ supported in $t \asymp 1$ and $\Phi^{(j)} \ll_j 1$ for $j \ge 0$. Then for any $A, \delta > 0$ and $H := x^\delta N^{-1} Q$, one has
\[
\begin{aligned}
    \sum_{n \equiv a \pmod{q}} \Phi\left(\frac{n}{N}\right)
    &=
    \frac{N}{q} \hat{\Phi}(0)
    +
    O_{A,\delta}\left(x^{-A}\right)
    \\
    &+
    \frac{N}{Q} \int \sum_{\substack{H_j = 2^j \\ 1 \le H_j \le H}}
    \sum_{\frac{1}{2} H_j \le |h| \le 2H_j}
    c_{j,u}(h)\,
    \Phi\left(\frac{uq}{Q}\right) e\left(\frac{ah}{q}\right) du,
\end{aligned}
\]
where the support of the integral in $u$ is bounded, and
\begin{equation} \label{eq:h-coeffs-poisson}
    c_{j,u}(h) := \Psi_j\left(\frac{|h|}{H_j}\right) e\left(-h \frac{uN}{Q} \right),
\end{equation}
for some compactly-supported smooth functions $\Psi_j : (\frac{1}{2}, 2) \to \C$ with $\Psi_j^{(k)} \ll_k 1$ for $k \ge 0$.
\end{lemma}

\begin{proof}
The Poisson identity \cref{eq:Poisson} with a change of variables yields
\[
    \sum_{n \equiv a \pmod{q}} \Phi\left(\frac{n}{N}\right)
    =
    \frac{N}{q} \sum_{h \in \Z} \hat{\Phi}\left(\frac{hN}{q}\right) 
    e\left(\frac{ah}{q}\right).
\]
We take out the main term at $h = 0$, put $|h| \ge 1$ in dyadic ranges via a smooth partition of unity
\[
    \one_{\Z_+}(|h|) = \one_{\Z_+}(|h|) \sum_{H_j = 2^j \ge 1} \Psi_j\left(\frac{|h|}{H_j}\right),
\]
and bound the contribution of $H_j > H = x^\delta N^{-1} Q$ by $O_{A,\delta}(x^{-A})$ using the Schwarz decay of $\Phi$. In the remaining sum
\[
\begin{aligned}
    \frac{N}{q} \sum_{\substack{H_j = 2^j \\ 1 \le H_j \le H}}
    \sum_{\frac{1}{2}H_j \le |h| \le 2H_j}
    \Psi_j\left(\frac{|h|}{H_j}\right) \hat{\Phi}\left(\frac{hN}{q}\right)
    e\left(\frac{ah}{q}\right),
\end{aligned}
\]
we separate the $h, q$ variables via the Fourier integral
\[
    \hat{\Phi}\left(\frac{hN}{q}\right) = \int \Phi(t)\, e\left(-h \frac{tN}{q}\right)\, dt
    =
    \frac{q}{Q} \int \Phi\left(\frac{uq}{Q}\right)\, e\left(-h \frac{uN}{Q}\right)\, du,
\]
where we let $t = uq/Q$. Swapping the (finite) sums with the integral completes our proof.
\end{proof}

We also highlight the non-standard \cref{not:tn}, pertaining to rational approximations.
Further analytic notation specific to each section is described therein (see, e.g., \cref{not:concentration,not:set-up-info,not:set-up-harman}).
For the rest of this section, we recount the main concepts relevant to bounding sums of Kloosterman sums via the Kuznetsov trace formula, mostly to clarify our notation (in particular, to point out small changes to the notation in \cite{deshouillers1982kloosterman}), and to explicitate a few useful lemmas.

\subsection{Cusps, automorphic forms, Kloosterman sums} \label{subsec:automorphic}

Recall that $\PSL_2(\R) := \SL_2(\R)/\{\pm 1\}$ acts naturally on $\C \cup \{\infty\}$ by $\begin{psmall} a & b \\ c & d \end{psmall} z := \frac{az+b}{cz+d}$.
For $q \in \Z_+$, we denote by $\Gamma_0(q)$ the modular subgroup of level $q$, consisting of those matrices $\begin{psmall} a & b \\ c & d \end{psmall} \in \PSL_2(\Z)$ with $c \equiv 0 \pmod{q}$. A number $\ma \in \C \cup \{\infty\}$ is called a \emph{cusp} of $\Gamma_0(q)$ iff it is the unique fixed point of some $\sigma \in \Gamma_0(q)$; we write $\Gamma_{\ma} := \{ \sigma \in \Gamma_0(q) : \sigma \ma = \ma\}$ for the stabilizer of $\ma$ inside $\Gamma_0(q)$. Two cusps are \emph{equivalent} iff they lie in the same orbit of $\Gamma_0(q)$; the corresponding stabilizers are then conjugate inside $\Gamma_0(q)$. By \cite[Lemma 2.3]{deshouillers1982kloosterman}, the fractions
\begin{equation} \label{eq:unique-repr-cusps}
    \left\{\frac{u}{w}\ : \ u, w \in \Z_+,\ (u, w) = 1,\ w \mid q,\ u \le \gcd\left(w, \frac{q}{w}\right) \right\}
\end{equation}
form a maximal set of inequivalent cusps of $\Gamma_0(q)$. Following \cite[(1.1)]{deshouillers1982kloosterman}, given a cusp $\ma$ of $\Gamma_0(q)$ and its equivalent representative $u/w$ from \cref{eq:unique-repr-cusps}, we denote
\begin{equation} \label{eq:mu}
    \mu(\ma) := \frac{\gcd\left(w, \frac{q}{w}\right)}{q},
\end{equation}
(Like most of our notation involving cusps, this implicitly depends on the level $q$ as well.)
In particular, the cusp at $\infty$ of $\Gamma_0(q)$ is equivalent to the fraction $1/q$, so we have $\mu(\infty) = q^{-1}$. More generally, we have $\mu(1/s) = q^{-1}$ whenever $q = rs$ with $\gcd(r, s) = 1$, and it is these cusps which account for most applications to sums of Kloosterman sums; thus for simplicity, we restrict all of our main results to cusps with $\mu(\ma) = q^{-1}$. Following \cite[(1.2)]{deshouillers1982kloosterman}, a \emph{scaling matrix} $\sigma_{\ma}$ for a cusp $\ma$ is an element of $\PSL_2(\R)$ such that
\begin{equation} \label{eq:scaling-matrices}
    \sigma_\ma \infty = \ma 
    \qquad \text{and} \qquad 
    \sigma_{\ma}^{-1} \Gamma_\ma \sigma_\ma = \Gamma_\infty = \left\{\begin{psmall} 1 & n \\ 0 & 1 \end{psmall} : n \in \Z \right\}.
\end{equation}
Scaling matrices will allow us to expand $\Gamma_{\ma}$-invariant functions $f : \H \to \C$ as Fourier series around the cusp $\ma$, via the change of coordinates $z \gets \sigma_\ma z$ (note that if $f$ is $\Gamma_{\ma}$-invariant, then $z \mapsto f(\sigma_{\ma} z)$ is $\Gamma_\infty$-invariant). For a given cusp $\ma$, the choice of $\sigma_{\ma}$ can only vary by simple changes of coordinates 
\begin{equation} \label{eq:vary-scaling}
    \tilde{\sigma}_{\ma} = \sigma_{\ma} \begin{psmall} 1 & \alpha \\ 0 & 1 \end{psmall},
\end{equation}
for $\alpha \in \R$ (which result in multiplying the Fourier coefficients by exponential phases $e(n\alpha)$). When $\mu(\ma) = q^{-1}$, we must have $\ma = \tau (1/s)$ for some $\tau \in \Gamma_0(q)$ and $rs = q$ with $(r, s) = 1$; in this case, inspired by Watt \cite[p.\,195]{watt1995kloosterman}, we will use the canonical choice of scaling matrix
\begin{equation} \label{eq:scaling-choices}
    \sigma_{\ma} = \tau \cdot \begin{psmall} 
    \sqrt{r} & -\bar{s}/\sqrt{r} \\
    s\sqrt{r} & \bar{r}\sqrt{r}
    \end{psmall},
\end{equation}
where $\bar{r}, \bar{s}$ are integers such that $r\bar{r} + s\bar{s} = 1$ (for definiteness, let us say we pick $\bar{s} \ge 0$ to be minimal). This is different from the choice in \cite[(2.3)]{deshouillers1982kloosterman}, and leads to the simplification of certain extraneous exponential phases. For the cusp $\ma = \infty = \begin{psmall} 1 & 0 \\ -q & 1 \end{psmall} (1/q)$, \cref{eq:scaling-choices} reduces back to the identity matrix.
%\[
%    \sigma_\infty = \begin{psmall}
%    1 & 0 \\ 0 & 1
%    \end{psmall} = \Id.
%\]

We refer the reader to the aforementioned work of Deshouillers--Iwaniec \cite{deshouillers1982kloosterman} for a brief introduction to the classical spectral theory of $\GL_2$ automorphic forms, to \cite{iwaniec2021spectral,iwaniec1997topics,iwaniec2021analytic} for a deeper dive into this topic, to \cite{drappeau2017sums,topacogullari2018shifted,drappeau2023one,de2020niveau,lichtman2023primes,pascadi2025smooth} for follow-up works and optimizations, and to \cite{bump1998automorphic} for the modern viewpoint of automorphic representations. For our purposes, an \emph{automorphic form} of level $q$ and integer weight $k \ge 0$ will be a smooth function $f : \H \to \C$ satisfying the transformation law
\[
    f(\sigma z) = j(\sigma, z)^k f(z)
    \qquad 
    \forall \sigma \in \Gamma_0(q),
    \qquad 
    \text{where}
    \qquad 
    j\left(\begin{psmall} a & b \\ c & d \end{psmall}, z\right) := cz+d. 
\]
as well as moderate (at-most-polynomial) growth conditions near every cusp. We say that $f$ is \emph{square-integrable} iff $\la f, f \ra_k < \infty$, where
$\la f, g \ra_k := \iint_{\Gamma_0(q) \backslash \H} f(x+iy)\, \bar{g(x+iy)}\,  y^{k-2}\, dx\, dy$ is the Petersson inner product. We denote by $L^2(\Gamma_0(q) \backslash \H, k)$ the space of square-integrable automorphic forms of level $q$ and weight $k$; when we drop the dependency on $k$, it should be understood that $k = 0$. Finally, we call $f$ a \emph{cusp form} iff it is square-integrable and vanishes at all cusps. 

Kloosterman sums show up in the Fourier coefficients of \emph{Poincar\'e series}, which are useful in detecting the Fourier coefficients of other automorphic forms via inner products (see \cite[(1.8), (1.18)]{deshouillers1982kloosterman}). In fact, by Fourier expanding a Poincar\'e series corresponding to a cusp $\ma$ around another cusp $\mb$, one is led to a more general family of Kloosterman-type sums, depending on both $\ma$ and $\mb$. 

More specifically (following \cite[(1.3)]{deshouillers1982kloosterman}, \cite[\S 4.1.1]{drappeau2017sums}, \cite{iwaniec1997topics}), given two cusps $\ma, \mb$ of $\Gamma_0(q)$, we first let
\[
    \mC_{\ma\mb} := \left\{c \in \R_+ : \exists\, a, b, d \in \R, \begin{psmall} a & b \\ c & d \end{psmall} \in \sigma_{\ma}^{-1} \Gamma_0(q)\, \sigma_{\mb} \right\}.
\]
Here $\sigma_{\ma}$ and $\sigma_{\mb}$ are arbitrary scaling matrices for $\ma$ and $\mb$, but the set $\mC_{\ma\mb}$ actually depends only on $\ma$ and $\mb$ (since multiplication by matrices $\begin{psmall} 1 & \alpha \\ 0 & 1 \end{psmall}$ does not affect the bottom-left entry). Then we let
\[
    \mD_{\ma\mb}(c) := \left\{ \tilde{d} \in \R/c\Z : \exists\, a, b \in \R,\, d \in \tilde{d}, \begin{psmall} a & b \\ c & d \end{psmall} \in \sigma_{\ma}^{-1} \Gamma_0(q)\, \sigma_{\mb} \right\},
\]
for any $c \in \R_+$ (although this is only nonempty when $c \in \mC_{\ma \mb}$). By this definition, the set $\mD_{\ma\mb}(c)$ is finite, does not depend on $\sigma_{\ma}$, and only depends on $\sigma_{\mb}$ up to translations. It turns out that a given $\tilde{d} \in \mD_{\ma \mb}(c)$ uniquely determines the value of $\tilde{a} \in \R/c\Z$ such that $\begin{psmall} a & b \\ c & d \end{psmall} \in \sigma_{\ma}^{-1} \Gamma_0(q)\, \sigma_{\mb}$ for some $a \in \tilde{a}$, $d \in \tilde{d}$ (see \cite[p.\,239]{deshouillers1982kloosterman}). Symmetrically, this $\tilde{a}$ does not depend on $\sigma_{\mb}$, and only depends on $\sigma_{\ma}$ up to translations. Thus given $c \in \R_+$ and $m, n \in \Z$, it makes sense to define
\begin{equation} \label{eq:generalized-kloosterman}
    S_{\ma \mb}(m, n; c) := 
    \sum_{\tilde{d} \in \mD_{\ma \mb}(c)}
    e\left(\frac{m \tilde{a} + n \tilde{d}}{c}\right),
\end{equation}
where $\tilde{a}$ and $\tilde{d}$ are corresponding values mod $c$; note that this vanishes unless $c \in \mC_{\ma\mb}$. Since varying the choices of $\sigma_{\ma}$ and $\sigma_{\mb}$ has the effect of uniformly translating $\tilde{a}$, respectively $\tilde{d}$, it follows that $S_{\ma \mb}(m, n; c)$ only depends on $\sigma_{\ma}, \sigma_{\mb}$ up to multiplication by exponential phases $e(m\alpha)$, $e(n\beta)$. In fact, the same holds true when varying $\ma$ and $\mb$ in equivalence classes of cusps \cite[p.\,239]{deshouillers1982kloosterman}. We also note the symmetries
\begin{equation} \label{eq:kloosterman-symmetries}
    S_{\ma\mb}(m, -n; c) = \bar{S_{\ma\mb}(-m, n; c)},
    \qquad\qquad 
    S_{\ma\mb}(m, n; c) = \bar{S_{\mb\ma}(n, m; c)},
\end{equation}
the second one following from the fact that
\[
    \begin{psmall}
    a & b \\ 
    c & d
    \end{psmall}
    \in \sigma_{\ma}^{-1} \Gamma_0(q)\, \sigma_{\mb}
    \qquad 
    \iff 
    \qquad
    \begin{psmall}
    -d & b \\ 
    c & -a
    \end{psmall}
    =
    - 
    \begin{psmall}
    a & b \\ 
    c & d
    \end{psmall}^{-1}
    \in \sigma_{\mb}^{-1} \Gamma_0(q)\, \sigma_{\ma}.
\]
Let us now relate these sums to the classical Kloosterman sums from \cref{eq:kloosterman}.

\begin{knownlemma}[Explicit Kloosterman sums \cite{watt1995kloosterman}] \label{lem:explicit-kloosterman}
Let $q = rs$ with $r, s \in \Z_+$, $\gcd(r, s) = 1$.  Then for any $c \in \R_+$ and $m, n \in \Z$, with the choice of scaling matrices from \cref{eq:scaling-choices}, one has
\begin{equation} \label{eq:explicit-kl-diff}
    S_{\infty\,1/s}\left(m, n; s \sqrt{r} c\right) = \one_{c \in \Z_+,\, (c, r) = 1}\, S(m \bar{r}, n; sc).
\end{equation}
Moreover, let $\ma$ be any cusp of $\Gamma_0(q)$ with $\mu(\ma) = q^{-1}$, and $\sigma_{\ma}$ be as in \cref{eq:scaling-choices}. Then one has
\begin{equation} \label{eq:explicit-kl-same}
    S_{\ma \ma}(m, n; c) = \one_{c \in q\Z_+}\, S(m, n; c).
\end{equation}
Varying the choice of scaling matrix as in \cref{eq:vary-scaling} would result in an additional factor of $e((n-m)\alpha)$.
\end{knownlemma}
\begin{proof}
These identities are precisely \cite[(3.5) and (3.4)]{watt1995kloosterman}, at least when $\ma = 1/s$ for some $rs = q$, with $(r, s) = 1$. For a general cusp with $\mu(\ma) = q^{-1}$, we have $\ma = \tau(1/s)$ for some $\tau \in \Gamma_0(q)$, but the presence of $\tau$ in the scaling matrix from \cref{eq:scaling-choices} does not affect the set $\sigma_{\ma}^{-1} \Gamma_0(q)\, \sigma_{\ma}$, nor the generalized Kloosterman sum $S_{\ma\ma}(m, n; c)$. For explicit computations of this type, see \cite[\S 2]{deshouillers1982kloosterman}.
\end{proof}

\subsection{The Kuznetsov formula and exceptional eigenvalues}

We now recognize some important classes of $\GL_2$ automorphic forms of level $q$: 
\begin{itemize}
    \item[(1).] \emph{Classical modular forms}, which are holomorphic with removable singularities at all cusps, and can only have even weights $k \in 2\Z_+$ (except for the zero form). A \emph{holomorphic cusp form} $f$ additionally vanishes at all cusps; such forms have Fourier expansions
    \begin{equation} \label{eq:expansion-holomorphic}
        j(\sigma_\ma, z)^{-k} f(\sigma_{\ma} z) = \sum_{n=1}^\infty \psi_\ma(n)\, e(n z)
    \end{equation}
    around each cusp $\ma$ of $\Gamma_0(q)$ (see \cite[(1.7)]{deshouillers1982kloosterman}). We mention that the space of holomorphic cusp forms of weight $k$ is finite-dimensional, and denote its dimension by $h_k = h_k(q)$.
    \item[(2).] \emph{Maass forms} (of weight $0$), which are invariant under the action of $\Gamma_0(q)$, and are eigenfunctions of the hyperbolic Laplacian $\Delta = -y^2 \left(\partial_x^2 + \partial_y^2\right)$. These include:
    \begin{itemize}
        \item[(a).] Maass \emph{cusp} forms, which additionally vanish at all cusps and are square-integrable. These (plus the constant functions) correspond to the discrete spectrum of the hyperbolic Laplacian on $L^2\left(\Gamma_0(q) \backslash \H\right)$, consisting of eigenvalues $0 = \lambda_0 < \lambda_1 \le \lambda_2 \le \lambda_3 \le \ldots$ with no limit point. Around a given cusp $\ma$, Maass cusp forms have Fourier expansions (see \cite[(1.15)]{deshouillers1982kloosterman})
        \begin{equation} \label{eq:expansion-Maass}
            u(\sigma_{\ma} z) = y^{1/2} \sum_{n \neq 0} \rho_{\ma}(n) K_{i\kappa}(2\pi |n| y)\, e(mx),
        \end{equation}
        where $z = x + iy$ and $K$ is the Whittaker function normalized as in \cite[p.\,264]{deshouillers1982kloosterman}.
        \item[(b).] \emph{Eisenstein series}, explicitly defined by $E_\ma(z; s) := \sum_{\tau \in \Gamma_\ma \backslash \Gamma_0(q)} \Im^s\left(\sigma_{\ma}^{-1} \tau z\right)$ for $\Re\, s > 1$, and meromorphically continued to $s \in \C$. Although not square-integrable themselves, ``incomplete'' versions of Eisenstein series with $s = \frac{1}{2} + ir$ (and $r \in \R$) can be used to describe the orthogonal complement in $L^2\left(\Gamma_0(q) \backslash \H\right)$ of the space of Maass cusp forms, corresponding to the continuous spectrum of the hyperbolic Laplacian. Sharing similarities with both Maass cusp forms and Poincar\'e series, the Eisenstein series $E_\ma$ have Fourier expansions \cite[(1.17)]{deshouillers1982kloosterman} around any cusp $\mb$, involving the Whittaker function and the Kloosterman-resembling coefficients (for $n \in \Z$, $n \neq 0$)
        \begin{equation} \label{eq:expansion-Eisenstein}
            \varphi_{\ma\mb}(n; s) := \sum_{c \in \mC_{\ma\mb}} c^{-2s} \sum_{\tilde{d} \in \mD_{\ma\mb}(c)} e\left(\frac{n
            \tilde{d}}{c}\right).
        \end{equation}
    \end{itemize}
\end{itemize}

We are particularly interested in the \emph{exceptional} Maass cusp forms, which have eigenvalues $\lambda_j \in (0, 1/4)$; there can only be finitely many such forms of each level $q$, and Selberg conjectured \cite{selberg1965estimation} that there are none. With implicit dependencies on $q$, we denote 
\begin{equation} \label{eq:eigenvalue-notation}
    \kappa_j^2 := \lambda_j - \frac{1}{4} \qquad\qquad \text{and} \qquad\qquad \theta_j := i \kappa_j,
 \end{equation}
where $\kappa_j$ is chosen such that $i \kappa_j > 0$ or $\kappa_j \ge 0$; thus exceptional forms correspond to imaginary values of $\kappa_j$ and positive values of $\theta_j$. Letting
\[
    \theta(q) := \sqrt{\max\left(0, \tfrac{1}{4} - \lambda_1(q)\right)} = \begin{cases}
        \theta_1(q), & \theta_1(q) > 0 \\ 
        0, & \text{otherwise}.
    \end{cases},
    \qquad\qquad 
    \theta_{\max} := \sup_{q \ge 1}\, \theta(q),
\]
Selberg's eigenvalue conjecture asserts that $\theta_{\max} = 0$, and the best result towards it is due to Kim--Sarnak \cite[Appendix 2]{kim2003functoriality}. 
This deep unconditional result requires the theory of $\GL_n$ automorphic representations \cite{bump1998automorphic}, but it is a very useful black-box input to spectral methods, where various bounds have exponential dependencies on $\theta$.

\begin{knowntheorem}[Kim--Sarnak's eigenvalue bound \cite{kim2003functoriality}] \label{thm:kim-sarnak}
One has $\theta_{\max} \le \tfrac{7}{64}$.
\end{knowntheorem}

Based on earlier work of Kuznetsov \cite{kuznetsov1980petersson}, Deshouillers--Iwaniec \cite{deshouillers1982kloosterman} developed a trace formula relating weighted sums over $c$ of the generalized Kloosterman sums from \cref{eq:generalized-kloosterman} to (sums of products of) the Fourier coefficients of holomorphic cusp forms, Maass cusp forms, and Eisenstein series, around any two cusps $\ma, \mb$ of $\Gamma_0(q)$. Roughly speaking, this follows by summing two applications of Parseval's identity for the aforementioned Poincar\'e series: one in the space of holomorphic cusp forms (summing over all weights $k \in 2\Z_+$), and one in the space $L^2(\Gamma_0(q) \slash \H)$ of square-integrable automorphic forms of weight $0$, via the spectral decomposition of the hyperbolic Laplacian (leading to the terms from Maass cusp forms and Eisenstein series). 

One can arrange the resulting \emph{Kuznetsov trace formula} so that the Kloosterman sums in the left-hand side are weighted by an arbitrary compactly-supported smooth function $\varphi$; in the right-hand side, the Fourier coefficients of automorphic forms are consequently weighted by \emph{Bessel transforms} of $\varphi$, defined for $r \in \R \setminus \{0\}$ by
\begin{equation} \label{eq:Bessel-transform}
\begin{aligned}
    \tilde{\mB}_\varphi(r) &:= \int_0^\infty J_r(y)\, \varphi(y) \frac{dy}{y},
    \\
    \hat{\mB}_\varphi(r) &:= \frac{\pi}{\sinh(\pi r)} \int_0^\infty \frac{J_{2ir}(x) - J_{-2ir}(x)}{2i} \varphi(x) \frac{dx}{x},
    \\
    \check{\mB}_\varphi(r) &:= \frac{4}{\pi} \cosh(\pi r) \int_0^\infty K_{2ir}(x)\, \varphi(x) \frac{dx}{x},
\end{aligned}
\end{equation}
where $K_{it}$ is the aforementioned Whittaker function, and the Bessel functions $J_\ell$, $J_{it}$ are defined as in \cite[p.\,264--265]{deshouillers1982kloosterman}
(above we slightly departed from the notation in \cite{deshouillers1982kloosterman,drappeau2017sums}, to avoid confusion with Fourier transforms). All we will need to know about these transforms are the following bounds.

\begin{knownlemma}[Bessel transform bounds \cite{deshouillers1982kloosterman}] \label{lem:bessel-bounds}
Let $Y > 0$ and $\varphi : \R \to \C$ be a smooth function with compact support in $y \asymp Y$, satisfying $\varphi^{(j)}(y) \ll_j Y^{-j}$ for $j \ge 0$. Then one has
\begin{align}
    \hat{\mB}_\varphi(ir),\ \check{\mB}_{\varphi}(ir) 
    &\ll 
    \frac{1 + Y^{-2r}}{1+Y}, \qquad\qquad\ \   \text{ for }\ 0 < r < \frac{1}{2},
    \label{eq:bessel-bound-1}
    \\
    \tilde{\mB}_\varphi(r),\ \hat{\mB}_\varphi(r),\ \check{\mB}_{\varphi}(r) 
    &\ll 
    \frac{1 + |\log Y|}{1+Y}, \qquad\qquad \text{ for }\ r \in \R \setminus \{0\},
    \label{eq:bessel-bound-2}
    \\
    \tilde{\mB}_\varphi(r),\ \hat{\mB}_\varphi(r),\ \check{\mB}_{\varphi}(r) 
    &\ll 
    |r|^{-5/2} + |r|^{-3}Y, \qquad \text{ for }\ r \in \R,\ |r| \ge 1,
    \label{eq:bessel-bound-3}
\end{align}
Moreover, if $\varphi$ is nonnegative with $\int \varphi(y)\, dy \gg Y$, and $Y < c$ for some constant $c \ll 1$ (depending on the implied constants so far), then one has
\begin{align} 
    \hat{\mB}_\varphi(\kappa) &\ll (\kappa^2+1)^{-1},
    \qquad  
    \text{ for }\ \kappa \in \R \setminus \{0\},
    \label{eq:exc-bound-1}
    \\
    \hat{\mB}_\varphi(\kappa) &\asymp Y^{-2i\kappa}, 
    \qquad\qquad \, 
    \text{ for }\ 0 < i \kappa < \frac{1}{2}.
    \label{eq:exc-bound-2}
\end{align}
\end{knownlemma}
\begin{proof}
The bounds in \cref{eq:bessel-bound-1,eq:bessel-bound-2,eq:bessel-bound-3} constitute \cite[Lemma 7.1]{deshouillers1982kloosterman} (note that $\varphi$ satisfies the requirements in \cite[(1.43) and (1.44)]{deshouillers1982kloosterman} for $(Y, 1)$ in place of $(X, Y)$). Similarly, \cref{eq:exc-bound-1} and the lower bound in \cref{eq:exc-bound-2} are \cite[(8.2) and (8.3), following from (8.1)]{deshouillers1982kloosterman}, using an appropriate choice of the constants $\eta_1, \eta_2$. The upper bound in \cref{eq:exc-bound-2} also follows from \cite[(8.1)]{deshouillers1982kloosterman}, but is in fact already covered by \cref{eq:bessel-bound-1} (using $r = -i\kappa$ and the fact that $\hat{\mB}_\varphi$ is even).
\end{proof}

Finally, let us state the Kuznetsov trace formula, following the notation of Deshouillers--Iwaniec \cite{deshouillers1982kloosterman}.

\begin{knownproposition}[Kuznetsov trace formula \cite{deshouillers1982kloosterman,kuznetsov1980petersson}] \label{prop:kuznetsov}
Let $\varphi : \R \to \C$ be a compactly-supported smooth function, $q \in \Z_+$, and $\ma, \mb$ be cusps of $\Gamma_0(q)$. Then for any positive integers $m, n$ and $\sgn \in \{1, -1\}$, one has
\begin{equation} \label{eq:kuznetsov}
    \sum_{c \in \mC_{\ma\mb}} \frac{S_{\ma\mb}(m, \sgn\cdot n; c)}{c} \varphi \left(\frac{4 \pi \sqrt{mn}}{c}\right)
    =
    \begin{cases}
        \mH + \mM + \mE, & \sgn = 1, \\
        \mM' + \mE', & \sgn = -1,
    \end{cases}
\end{equation}
with the following notations. Firstly, the \emph{holomorphic} contribution is
\begin{equation} \label{eq:holomorphic-contribution}
    \mH = \frac{1}{2\pi} \sum_{k \in 2\Z_+} \tilde{\mB}_\varphi(k-1) \frac{i^k (k-1)!}{(4\pi \sqrt{mn})^{k-1}} \sum_{j = 1}^{h_k(q)}  \bar{\psi_{jk\ma}(m)}\, \psi_{jk\mb}(n),
\end{equation}
for any orthonormal bases of level-$q$ holomorphic cusp forms $(f_{jk})_j$ of weight $k \in 2\Z_+$, with Fourier coefficients $\psi_{jk\ma}(n)$ as in \cref{eq:expansion-holomorphic}. Secondly, the \emph{Maass} contributions are
\begin{equation} \label{eq:Maass-contribution}
    \mM = \sum_{j = 1}^\infty \frac{\hat{\mB}_\varphi(\kappa_j)}{\cosh(\pi \kappa_j)}\, \bar{\rho_{j\ma}(m)}\, \rho_{j\mb}(n),
    \qquad\quad
    \mM' = \sum_{j = 1}^\infty \frac{\check{\mB}_\varphi(\kappa_j)}{\cosh(\pi \kappa_j)}\, \rho_{j\ma}(m)\, \rho_{j\mb}(n),
\end{equation}
for any orthonormal basis $(u_j)_j$ of level-$q$ Maass cusp forms, with eigenvalues $\lambda_j$ (and $\kappa_j, \theta_j$ as in \cref{eq:eigenvalue-notation}), and Fourier coefficients $\rho_{j\ma}(n)$ as in \cref{eq:expansion-Maass}. Thirdly, the \emph{Eisenstein} contributions are
\begin{equation} \label{eq:Eisenstein-contribution}
\begin{aligned}
    \mE &= \frac{1}{\pi} \sum_{\mc} \int_{-\infty}^\infty \hat{\mB}_\varphi(r)\left(\frac{m}{n}\right)^{-ir} \bar{\varphi_{\mc\ma}\left(m; \frac{1}{2} + ir\right)} \varphi_{\mc\mb}\left(n; \frac{1}{2} + ir\right)\, dr,
    \\
    \mE' &= \frac{1}{\pi} \sum_{\mc} \int_{-\infty}^\infty \check{\mB}_\varphi(r)\left(mn\right)^{ir} \varphi_{\mc\ma}\left(m; \frac{1}{2} + ir\right) \varphi_{\mc\mb}\left(n; \frac{1}{2} + ir\right)\, dr,
\end{aligned}
\end{equation}
where the Fourier coefficients $\varphi_{\ma\mb}(n; s)$ are as in \cref{eq:expansion-Eisenstein}, and $\mc$ varies over the cusps of $\Gamma_0(q)$.
\end{knownproposition}

\begin{proof}
This is \cite[Theorem 2]{deshouillers1982kloosterman}.
\end{proof}

\begin{remark}
    Upon inspecting the Maass contribution \cref{eq:Maass-contribution} in light of the bounds \cref{eq:bessel-bound-1,eq:exc-bound-2}, the losses due to the exceptional spectrum are apparent. Indeed, if $\varphi(y)$ is supported in $y \asymp Y \asymp \frac{\sqrt{mn}}{C}$ for some $C > 0$ (indicating the size of $c$), then the Bessel transforms bounds for exceptional eigenvalues are (a priori) worse by a factor of 
    \[
        \max\left(1, Y^{-2\theta(q)}\right) \asymp \left(1 + \frac{C}{\sqrt{mn}}\right)^{2\theta(q)},
    \]
    compared to the regular (non-exceptional) spectrum.
\end{remark}

\subsection{Bounds for Fourier coefficients} \label{subsec:automorphic-fourier-bounds}
If one is interested in a particular holomorphic or Maass cusp form (ideally, a Hecke eigenform), then various bounds for its Fourier coefficients follow from the theory of automorphic representations and their $L$-functions \cite{deshouillers1982kloosterman,sarnak1995selberg,kim2003functoriality,bump1998automorphic,goldfeld2011automorphic1,goldfeld2011automorphic2}. Here we are interested in bounding averages over bases of automorphic forms, resembling those that show up in \cref{eq:Maass-contribution,eq:holomorphic-contribution,eq:Eisenstein-contribution}; naturally, these would be useful in combination with the Kuznetsov formula.

Remarkably, such bounds are often derived using the Kuznetsov formula once again (with different parameters, including the support range of the smooth function $\varphi$), together with various bounds for sums of Kloosterman sums, such as the Weil bound below. 

\begin{knownlemma}[Weil--Ramanujan bound] \label{lem:weil}
For any $c \in \Z_+$ and $m, n \in \Z$, one has
\[
    S(m, n; c) \ll \tau(c)\, (m, n, c)^{1/2} c^{1/2}.
\]
Also, for $m = 0$, one has $|S(0, n; c)| \le (n, c)$.
\end{knownlemma}

\begin{proof}
See \cite[Corollary 11.12]{iwaniec2021analytic} for the first bound; the second bound, concerning Ramanujan sums, is classical and follows by M\"obius inversion.
\end{proof}

The first results that we mention keep the index $n$ of the Fourier coefficients fixed, while varying the automorphic form.

\begin{knownlemma}[Fourier coefficient bounds with fixed $n$] \label{lem:coefficient-bounds}
    Let $K \gg 1$ and $\eps > 0$. With the notation of \cref{prop:kuznetsov}, each of the three expressions 
    \[
        \sum_{\substack{k \in 2\Z_+ \\ k \le K}} \frac{(k-1)!}{(4\pi n)^{k-1}} \sum_{j = 1}^{h_k(q)}  |\psi_{jk\ma}(n)|^2,
        \qquad\quad
        \sum_{|\kappa_j| \le K} \frac{|\rho_{j\ma}(n)|^2}{\ch(\pi \kappa_j)},
        \qquad\quad
        \sum_{\mc} \int_{-K}^K \left\vert \varphi_{\mc\ma}\left(n; \frac{1}{2} + ir\right) \right\vert^2\, dr
    \]
    is bounded up to a constant depending on $\eps$ by
    \[
        K^2 + (qnK)^{\eps}\, (q, n)^{1/2}\, \mu(\ma)\, n^{1/2}.
    \]
    Moreover, for the exceptional spectrum we have 
    \begin{equation} \label{eq:coefficient-bound-exceptional}
        \sum_{\lambda_j < 1/4} X^{2\theta_j}\, |\rho_{j\ma}(n)|^2 \ll_\eps 
        (qN)^\eps \left(1 + (q, n)^{1/2}\, \mu(\ma)\, n^{1/2}\right),
    \end{equation}
    for any $X \ll \max\left(1, \left((q, n)\, \mu(\ma)^2\, n\right)^{-1}\right)$.
\end{knownlemma}

\begin{proof}
These bounds roughly follow by combining \cref{lem:weil} with trace formulas like \cref{prop:kuznetsov}, for $m = n$ and suitable choices of $\varphi$.
See for example \cite[Lemmas 2.7 and 2.9]{topacogullari2017certain} with $q_0 = 1$ and $X = X_0$, noting the different normalizations of the Fourier coefficients.
\end{proof}

One of the key insights of Deshouillers--Iwaniec \cite{deshouillers1982kloosterman} was that the bounds in \cref{lem:coefficient-bounds} can be improved when averaging over $n \sim N$, by exploiting the bilinear structure in $m, n$ of the spectral side of the Kuznetsov formula \cref{eq:kuznetsov}. This leads to the so-called \emph{weighted large sieve inequalities} for the Fourier coefficients of automorphic forms, involving arbitrary sequences $(a_n)$; for $1$-bounded sequences, the result below saves a factor of roughly $N$ over the pointwise bounds in \cref{lem:coefficient-bounds}.

\begin{knownlemma}[Deshouillers--Iwaniec large sieve for the regular spectrum \cite{deshouillers1982kloosterman}] \label{lem:large-sieve-non-exceptional}
Let $K, N \ge 1/2$, $\eps > 0$, and $(a_n)$ be a sequence of complex numbers. With the notation of \cref{prop:kuznetsov}, each of the three expressions
\begin{gather*}
    \sum_{\substack{k \in 2\Z_+ \\ k \le K}} \frac{(k-1)!}{(4\pi)^{k-1}} \sum_{j = 1}^{h_k(q)} \left\vert \sum_{n \sim N} a_n n^{-(k-1)/2}\, \psi_{jk\ma}(n) \right\vert ^2,
    \qquad\qquad 
    \sum_{|\kappa_j| \le K} \frac{1}{\ch(\pi \kappa_j)} \left\vert \sum_{n \sim N} a_n\, \rho_{j\ma}(n) \right\vert^2,
    \\
    \sum_{\mc} \int_{-K}^K \left\vert \sum_{n \sim N} a_n\, n^{ir}\, \varphi_{\mc\ma}\left(n; \frac{1}{2} + ir\right) \right\vert^2\, dr
\end{gather*}
is bounded up to a constant depending on $\eps$ by 
\[
    \left(K^2 + \mu(\ma) N^{1+\eps} \right) \|a_n\|_2^2.
\]
\end{knownlemma}

\begin{proof}
This is \cite[Theorem 2]{deshouillers1982kloosterman}.
\end{proof}

\cref{lem:large-sieve-non-exceptional} includes the contribution of the exceptional Maass cusp forms, but is not the optimal result for handling it. Indeed, to temper the growth of the Bessel functions weighing the exceptional Fourier coefficients in \cref{eq:Maass-contribution}, one needs to incorporate factors of $X^{2\theta_j}$ into the sum over Maass forms, as in \cref{eq:coefficient-bound-exceptional}. The following is a preliminary result toward such bounds.

\begin{knowncorollary}[Preliminary bound for exceptional forms \cite{deshouillers1982kloosterman}] \label{cor:exc-bound}
Let $X, N \ge 1/2$, $\eps > 0$, $(a_n)_{n \sim N}$ be a complex sequence. Let $\Phi(t)$ be a nonnegative smooth function supported in $t \asymp 1$, with $\Phi^{(j)}(t) \ll_j 1$ for $j \ge 0$, and $\int \Phi(t)\, dt \gg 1$. Then with the notation of \cref{prop:kuznetsov}, one has
\begin{equation} \label{eq:exceptional-bound}
\begin{aligned}
    \sum_{\lambda_j < 1/4} X^{2\theta_j} \left\vert \sum_{n \sim N} a_n\, \rho_{j\ma}(n) \right\vert^2 
    &\ll 
    \left\vert \sum_{c \in \mC_{\ma\ma}} \frac{1}{c} 
    \sum_{m,n \sim N} \bar{a_m}\, a_n\, S_{\ma \ma}(m, n; c)\, \Phi\left(\frac{\sqrt{mn}}{c} X \right) \right\vert
    \\
    &+ 
    O_\eps\left(1 + \mu(\ma) N^{1+\eps}\right) \|a_n\|_2^2.
\end{aligned}
\end{equation}
\end{knowncorollary}

\begin{proof}
This is essentially present in \cite{deshouillers1982kloosterman} (see \cite[first display on p.\,271]{deshouillers1982kloosterman}, and \cite[(8.7)]{deshouillers1982kloosterman} for the case $\ma = \infty$), but let us give a short proof for completion. If $X \ll 1$, the result follows immediately from \cref{lem:large-sieve-non-exceptional} with $K = 1/4$, and the bound $\cosh(\pi \kappa) \asymp 1$ for $i\kappa \in [0, 1/4]$ (recall that $i \kappa_j = \theta_j \le \tfrac{7}{64}$ by \cref{thm:kim-sarnak}, but the weaker Selberg bound $\theta_j \le \tfrac{1}{4}$ suffices here).

Otherwise, let $\varphi(y) := \Phi(yX(4\pi)^{-1})$, which satisfies all the assumptions in \cref{lem:bessel-bounds} for $Y = 4\pi X^{-1}$; in particular, we have
\begin{align}
    \label{eq:phi-bound-1}
    \max(\hat{\mB}_\varphi(r), \tilde{\mB}_\varphi(r)) &\ll |r|^{-5/2}, \qquad\qquad\quad \text{for } |r| \ge 1,
    \\ 
    \label{eq:phi-bound-2}
    \hat{\mB}_\varphi(\kappa) &\ll \left(\kappa^2 + 1\right)^{-1}, \qquad\quad \text{for } \kappa \in \R \setminus \{0\},
    \\
    \label{eq:phi-bound-3}
    \hat{\mB}_\varphi(\kappa) &\gg X^{2i\kappa}, \qquad\qquad\qquad \text{for } 0 < i \kappa < 1/2.
\end{align}
Now apply \cref{prop:kuznetsov} with this choice of $\varphi$ and $\ma = \mb$, multiply both sides by $\bar{a_m}\, a_n$, and sum over $m, n \sim N$, to obtain
\[
\begin{aligned}
    \sum_{c \in \mC_{\ma\ma}} \frac{1}{c} 
    \sum_{m,n \sim N} \bar{a_m}\, a_n\, S_{\ma \ma}(m, n; c)\, &\varphi\left(\frac{4\pi \sqrt{mn}}{c} \right)
    \\
    &= 
    \sum_{j \ge 1} \frac{\hat{\mB}_\varphi(\kappa_j)}{\ch(\pi \kappa_j)} \left\vert \sum_{n \sim N} a_n\, \rho_{j\ma}(n) \right\vert^2
    \\
    &+
    \frac{1}{\pi} 
    \sum_{\mc} \int_{-\infty}^\infty \hat{\mB}_\varphi(r) \left\vert \sum_{n \sim N} a_n\, n^{ir}\, \varphi_{\mc \ma}\left(n; \frac{1}{2} + ir\right) \right\vert^2 dr
    \\
    &+\frac{1}{2\pi} 
    \sum_{k \in 2\Z_+}
    \tilde{\mB}_\varphi(k-1)
    \frac{(k-1)!}{(4\pi)^{k-1}} 
    \sum_{1 \le j \le h_k(q)} \left\vert \sum_{n \sim N} a_n\, n^{-\frac{k-1}{2}} \psi_{jk\ma}(n) \right\vert^2 .
\end{aligned}
\]
Bounding the contribution of non-exceptional Maass cusp forms, holomorphic cusp forms, and Eisenstein series via \cref{eq:phi-bound-1}, \cref{eq:phi-bound-2}, and \cref{lem:large-sieve-non-exceptional} (in dyadic ranges $K = 2^p$), this reduces to
\begin{equation} \label{eq:Kloosterman-and-exceptional}
\begin{aligned}
    \sum_{c \in \mC_{\ma\ma}} \frac{1}{c} 
    \sum_{m,n \sim N} \bar{a_m}\, a_n\, S_{\ma \ma}(m, n; c)\, \Phi\left(\frac{\sqrt{mn}}{c} X\right)
    &= 
    \sum_{\lambda_j < 1/4} \frac{\hat{\mB}_\varphi(\kappa_j)}{\ch(\pi \kappa_j)} \left\vert \sum_{n \sim N} a_n\, \rho_{j\ma}(n) \right\vert^2
    \\
    &+
    O_\eps\left(1 + \mu(\ma) N^{1+\eps}\right) \|a_n\|_2^2.
\end{aligned}
\end{equation}
Combining this with the lower bound $\hat{\mB}_\varphi(\kappa_j) \gg X^{2\theta_j}$ (due to \cref{eq:phi-bound-3}), we recover the desired bound in \cref{eq:exceptional-bound}. 
\end{proof}

Finally, for the results with averaging over the level $q$, we will also need the following theorem of Deshouillers--Iwaniec \cite{deshouillers1982kloosterman}.

\begin{knowntheorem}[Deshouillers--Iwaniec's large sieve with level averaging \cite{deshouillers1982kloosterman}] \label{thm:large-sieve-level-avg}
Let $\eps > 0$, $X > 0$, $N, Q \ge 1/2$, and $\omega \in \R/\Z$. Let $q \in \Z_+$ and $\infty_q$ denote the cusp at $\infty$ of $\Gamma_0(q)$, with the choice of scaling matrix $\sigma_{\infty_q} = \Id$. Then with the notation of \cref{prop:kuznetsov}, one has
\begin{equation}
    \sum_{q \sim Q} \sum_{\lambda_j(q) < 1/4}
    X^{2\theta_j(q)} 
    \left\vert 
    \sum_{n \sim N} e(n\omega)\, \rho_{j\infty_q}(n)
    \right\vert^2 
    \ll_\eps 
    (QN)^\eps
    \left(Q + N\right) N,
\end{equation}
for any
\begin{equation}
    X \ll \max\left(N, \frac{Q^2}{N}\right).
\end{equation}
\end{knowntheorem}

\begin{proof}
This follows immediately from \cite[Theorem 1.7]{deshouillers1982kloosterman} with $X \gets X^{1/2}$. As noted in previous works \cite{pascadi2025smooth,bombieri1987primes2,de2020niveau}, although \cite[Theorem 7]{deshouillers1982kloosterman} was only stated for $\alpha = 0$, the same proof holds uniformly in $\alpha \in \R/\Z$. 
\end{proof}

\section{Combinatorial bounds} \label{sec:combi}

In this section, we obtain bounds for bilinear sums of the form $\sum_m a_m \sum_n b_n\, S(m, n; c)$ (say, in the range $c^{1/4} \ll N \ll c$), saving over the P\'olya--Vinogradov and Weil bounds if the Fourier transforms $\hat{a}$ and $\hat{b}$ are concentrated enough. Our computations here are elementary (not requiring the spectral theory of automorphic forms yet, nor any other prerequisites beyond \cref{subsec:standard-notation}), and use a combinatorial argument inspired by \cite{cilleruelo2011concentration}; the latter was also used, e.g., in \cite{kerr2023bounds}.

We highlight the following non-standard notation.

\begin{notation}[Rational approximation] \label{not:tn}
Given $M, N > 0$, let $T_{M,N} : (\R/\Z)^2 \to \R$ denote the function
\begin{equation} \label{eq:tn}
\begin{gathered}
    T_{M,N}(\alpha, \beta) := \min_{t \in \Z_+} \left(t + M\|\alpha t\| + N\|\beta t\|\right)
    \\
    (\text{abbreviating } T_N := T_{N,N}, \,
    T_N(\alpha) := T_N(\alpha, \alpha)).
\end{gathered}
\end{equation}
\end{notation}
This measures how well $\alpha$ and $\beta$ can be simultaneously approximated by rational numbers with small denominators $t$, in terms of the balancing parameters $M, N$. The inverses of these parameters indicate the scales at which $T_{M,N}(\alpha, \beta)$ has roughly constant size, due to the following lemma.

\begin{lemma}[Basic properties of $T_{M,N}$]
Let $M, N > 0$ and $\alpha, \beta, \gamma, \delta \in \R/\Z$. One has $T_{N,M}(\beta, \alpha) = T_{M,N}(\alpha, \beta) = T_{M,N}(\pm \alpha, \pm \beta)$ and
\begin{equation} \label{eq:T-triangle}
    T_N(\alpha, \beta \pm \alpha) \asymp T_N(\alpha, \beta).
\end{equation}
Moreover,
\begin{equation} \label{eq:T-shift}
    T_{M,N}(\alpha + \gamma, \beta) \le (1 + M\|\gamma\|)\, T_{M,N}(\alpha, \beta).
\end{equation}
In particular, if $\|\gamma\| \ll M^{-1}$ and $\|\delta\| \ll N^{-1}$, then
\begin{equation} \label{eq:T-roughly-constant}
    T_{M,N}\left(\alpha + \gamma, \beta + \delta \right)
    \asymp T_{M,N}(\alpha, \beta).
\end{equation}
\end{lemma}

\begin{proof}
The first equalities are obvious, and \cref{eq:T-triangle} follows from the triangle inequalities 
\[
    \|(\beta \pm \alpha)t\| \le \|\alpha t\| + \|\beta t\|, \qquad \qquad \|\beta t\| \le \|\alpha t\| + \|(\beta \pm \alpha)t\|.
\]
For \cref{eq:T-shift}, we note that 
\[
\begin{aligned}
    t + M\|(\alpha + \gamma) t\| + N\|\beta t\| 
    &\le 
    t + M\|\gamma t\| + M\|\alpha t\| + N\|\beta t\| 
    \\
    &\le 
    t(1 + M\|\gamma\|) +  M\|\alpha t\| + N\|\beta t\| 
    \\
    &\le 
    (1 + M\|\gamma\|) \left(t + M\|\alpha t\| + N\|\beta t\|\right),
\end{aligned}
\]
and take a minimum of both sides over $t \in \Z_+$. Finally, \cref{eq:T-roughly-constant} follows immediately from \cref{eq:T-shift}.
\end{proof}

\begin{lemma}[Dirichlet-style approximation] \label{lem:pigeonhole}
Let $\alpha, \beta \in \R/\Z$. Given any parameters $A, B \gg 1$, there exists a positive integer $t$ such that
\[
    t \ll AB, \qquad\qquad 
    \|\alpha t\| \ll \frac{1}{A}, \qquad\qquad 
    \|\beta t\| \ll \frac{1}{B}.
\]
In particular, for $N \ge 1/2$, one has
\begin{equation} \label{eq:tn-bound}
    T_N(\alpha, \beta) \ll \min \left(\sqrt{N(1 + \|\alpha-\beta\| N)}, N^{2/3} \right).
\end{equation}
\end{lemma}

\begin{proof}
Consider the sequence of points $\{(t\alpha, t\beta)\}_{t \le \lceil A \rceil \lceil B \rceil + 2}$ in $(\R/\Z)^2$; by the pigeonhole principle, at least two of these must lie in a box of dimensions $A^{-1} \times B^{-1}$, say $(t_i \alpha, t_i \beta)$ for $i \in \{1, 2\}$. Then we can pick $t := |t_1 - t_2|$ to establish the first claim.

Using $A = B = N^{1/3}$, we find that
\[
    T_N(\alpha, \beta) \ll N^{2/3},
\]
uniformly in $\alpha, \beta \in \R/\Z$. Using $A = \sqrt{N/(1+ \|\beta\| N)}$ and $B = 1$, we also have
\[
\begin{aligned}
    T_N(\alpha, \beta) 
    \le 
    \min_{t \in \Z_+} 
    \left(t + N\left(\|\alpha t\| + \|\beta\| t\right) \right)
    &\ll 
    A + \frac{N}{A} + N \|\beta\| A
    \\
    &\ll 
    \sqrt{N(1 + \|\beta\|N)},
\end{aligned}
\]
and thus
\[
    T_N(\alpha, \beta) \ll T_N(\alpha, \alpha - \beta) \ll 
    \sqrt{N(1 + \|\alpha - \beta\|N)}.
\]
This proves \cref{eq:tn-bound}.
\end{proof}

\begin{lemma}[Concentration of points on modular hyperbolas, following Cilleruelo--Garaev \cite{cilleruelo2011concentration}] \label{lem:cilleruelo-garaev}
Let $c \in \Z_+$, $a, b, \lambda \in \Z/c\Z$, $0 < X,Y \ll c$, and $I, J \subset \R$ be intervals of lengths $|I| = X$, $|J| = Y$. Then for any $\eps > 0$ and any $(c\alpha, c\beta) \in I \times J$, one has
\begin{equation} \label{eq:cilleruelo-garaev}
    \# \left\{
    (x, y) \in (I \cap \Z) \times (J \cap \Z) : xy \equiv \lambda \pmod{c}
    \right\}
    \ll_\eps 
    c^\eps \left(\frac{XY}{c}\, T_{\frac{c}{X}, \frac{c}{Y}}\left(\alpha, \beta\right) + \gcd(\lambda, c)\right),
\end{equation}
with $T_{M,N}(\alpha, \beta)$ as in \cref{not:tn}.
\end{lemma}

\begin{remark}
    \cref{lem:cilleruelo-garaev} counts solutions to the congruence $xy \equiv \lambda \pmod{c}$ in short intervals. On average over intervals of length $X, Y\gg \sqrt{c}$, one should expect around $XY/c$ solutions; \cref{eq:cilleruelo-garaev} essentially recovers this average bound when $\alpha$ and $\beta$ can be simultaneously approximated by rational numbers with a bounded denominator.
\end{remark}

\begin{remark}
    One can also interpret \cref{lem:cilleruelo-garaev} in terms of sum-product phenomena over $\Z/c\Z$. Indeed, the intervals $a + [-X, X]$ and $b + [-Y, Y]$ have many ``additive collisions'' of the form $x_1 + y_1 \equiv x_2 + y_2 \pmod{c}$ (with $x_1, x_2 \in a + [-X, X]$ and $y_1, y_2 \in b + [-Y, Y]$), so they should have few ``multiplicative collisions'' of the form $x_1 y_1 \equiv \lambda \equiv x_2 y_2  \pmod{c}$.
\end{remark}

\begin{proof}
If $I \cap \Z = \emptyset$ or $J \cap \Z = \emptyset$, the claim is trivial. So let $a \in I \cap \Z$ and $b \in J \cap \Z$; by a change of variables, we have
\[
    \#\{(x, y) \in I \times J : xy \equiv \lambda \pmod{c}\} 
    \le
    \# S(a, b),
\]
where
\[
    S(a, b) := \{(x, y) \in ([-X, X] \cap \Z) \times ([-Y, Y] \cap \Z) : (x+a)(y+b) \equiv \lambda \pmod{c}\}.
\]
The key idea, borrowed from \cite[Theorem 1]{cilleruelo2011concentration} (and also used, for example, in \cite[Lemma 5.3]{kerr2023bounds}), is to effectively reduce the size of $a$ and $b$ by appropriately scaling the congruence $(x+a)(y+b) \equiv \lambda \pmod{c}$, and then to pass to an equation in the integers. Indeed, let $t \in \Z_+$ be a scalar, and let $a', b'$ be the integers with minimal absolute values such that
\begin{equation} \label{eq:reduce-a}
    a t \equiv a' \pmod{c}
    \qquad\quad \text{and} \qquad\quad
    b t \equiv b' \pmod{c}.
\end{equation}
Then any given pair $(x, y) \in S(a, b)$ also satisfies the scaled congruence
\[
    t(x+a)(y+b) \equiv t\lambda \pmod{c}
    \qquad 
    \iff 
    \qquad 
    txy + b'x + a'y \equiv t(\lambda - ab) \pmod{c}.
\]
Denoting by $r \in \{0, 1, \ldots, c-1\}$ the residue of $t(\lambda-ab) \pmod{c}$, and 
\[
    z = z(x, y) := \frac{txy + b'x + a'y - r}{c},
\]
it follows that $(x, y, z)$ is an integer solution to the equation
\[
    txy + b'x + a'y = cz + r
    \qquad 
    \Leftrightarrow 
    \qquad 
    (tx + a')(ty + b') = t(cz+r) + a'b'.
\]
Note that
\[
\begin{aligned}
    z &\ll \frac{tXY + |b'|X + |a'| Y + c}{c}
    \\ 
    &\ll 
    \frac{t}{c}XY + \left\| \frac{bt}{c}\right\| X + \left\| \frac{at}{c}\right\| Y + 1 
    =: Z(t).
\end{aligned}
\] 
Now let $n(z) := t(cz+r) + a'b'$. The number of pairs $(x, y) \in S(a, b)$ with $n(z) \neq 0$ is at most
\[
\begin{aligned}
    \sum_{\substack{z \ll Z(t) \\ n(z) \neq 0}}
    \sum_{\substack{x, y \in \Z \\ (tx+a')(ty+b') = n(z) \\ (x+a)(y+b) \equiv \lambda \pmod{c}}}
    1
    \ll_\eps 
    (ct)^\eps Z(t),
\end{aligned}
\]
by the divisor bound. On the other hand, if $(x, y) \in S(a, b)$ satisfies $n(z) = (tx + a')(ty + b') = 0$, this forces $tx = -a'$ or $ty = -b'$, determining one of $x$ and $y$ uniquely. Suppose $x$ is determined; the condition $c \mid (x + a)(y + b) - \lambda$ implies $d := \gcd(x+a, c) \mid \gcd(\lambda, c)$, so 
\[
    \frac{c}{d}\ \bigg\vert \ \frac{x+a}{d} (y+b) - \frac{\lambda}{d}.
\]
Since $\gcd(c/d, (x+a)/d) = 1$, this uniquely determines the value of $y \pmod{c/d}$, leading to a total contribution of $1 + Yd/c$. Putting things together, we conclude that
\[
\begin{aligned}
    \# S(a, b) 
    &\ll_\eps 
    c^\eps \min_{t \in \Z_+} \left(t^\eps Z(t) \right) + 1 + \frac{X+Y}{c} \gcd(\lambda, c)
    \\
    &\ll
    c^\eps \left(\frac{XY}{c} \min_{t \in \Z_+} t^\eps \left(t + \frac{c}{X}\left\| \frac{at}{c}\right\| + \frac{c}{Y}\left\| \frac{bt}{c}\right\|\right) + 1 + \frac{X+Y}{c} \gcd(\lambda, c) \right)
    \\
    &\ll 
    c^{2\eps} \left(\frac{XY}{c} T_{\frac{c}{X},\frac{c}{Y}} \left(\frac{a}{c}, \frac{b}{c}\right) + \gcd(\lambda, c)\right),
\end{aligned}
\]
where we used that $X, Y \ll c$ in the last line (and implicitly that the minimum of $t + \frac{c}{X} \|at/c\| + \frac{c}{Y} \|bt/c\|$ is attained for $t \ll c$). Now if $\alpha, \beta \in \R$ satisfy $(c\alpha, c\beta) \in I \times J$, then we have $|a - c\alpha| \le X$ and $|b - c\beta| \le Y$, i.e.,
\[
    \left\vert\frac{a}{c} - \alpha \right\vert \le \frac{X}{c},
    \qquad\qquad
    \left\vert\frac{b}{c} - \beta \right\vert \le \frac{Y}{c}.
\]
So by \cref{eq:T-roughly-constant}, we have
\[
    T_{\frac{c}{X}, \frac{c}{Y}}\left(\frac{a}{c}, \frac{b}{c}\right)
    \asymp 
    T_{\frac{c}{X}, \frac{c}{Y}}(\alpha, \beta).
\]
We thus obtain the desired bound, up to a rescaling of $\eps$. 
\end{proof}

We now work towards our bilinear Kloosterman bound for sequences with sparse Fourier transforms, reminding the reader of the Fourier-analytic notation in \cref{subsec:standard-notation}. The connection to counting solutions to congruences of the form $xy \equiv 1 \pmod{c}$ comes from the identity 
\begin{equation} \label{eq:kloosterman-fourier}
    \sum_m a_m \sum_n b_n\, S(m, n; c) 
    = \sum_{x, y \pmod{c}} \hat{a}\left(\frac{x}{c}\right) \hat{b}\left(\frac{y}{c}\right) \one_{xy \equiv 1 \pmod{c}},
\end{equation}
obtained by expanding $S(m, n; c)$ and swapping sums. One can interpret this as a Parseval--Plancherel identity, the Kloosterman sum $S(m, n; c)$ being dual to the function $\one_{xy \equiv 1 \pmod{c}}$; this duality is often exploited in the converse direction (see, e.g., \cite[Chapter 6]{maynard2025primes2} and \cite{fouvry2015exponent}), but it turns out to also be a useful input for methods from the spectral theory of automorphic forms.

\begin{proposition}[Bilinear Kloosterman bound with exponential phases] \label{prop:bilinear-expo-phases}
Let $c, a \in \Z_+$, $\alpha, \beta \in \R/\Z$, $1 \ll M, N \ll c$, and $I, J \subset \Z$ be nonempty discrete intervals of lengths $|I| = M$, $|J| = N$. Then for any $\eps > 0$, one has
\[
    \sum_{m \in I} e(m\alpha) \sum_{n \in J} e(n\beta)\, S(am, an; c) 
    \ll_\eps 
    c^\eps \left(c\, T_{M,N}(\alpha, \beta) + \gcd(a, c) MN\right).
\]
\end{proposition}

\begin{remark}
When $\alpha = \beta = 0$, this recovers a result of Shparlinski and Zhang \cite{shparlinski2016cancellations}.
A similar argument produces the more general bound
\[
    \sum_{m \in I} e(m\alpha) \sum_{n \in J} e(n\beta)\, S(am+r, bn+s; c) 
    \ll_\eps 
    c^\eps \left(c\, T_{M,N}(\alpha, \beta) + \gcd\left(\frac{ab}{\gcd(a,b,c)}, c\right)MN\right),
\]
for entries of Kloosterman sums in arithmetic progressions, where $a, b \in \Z \setminus \{0\}$, $r, s \in \Z$.
%The factor of $d$ is seen to be necessary by taking $\alpha = \beta = 0$ and $a = b = c$.
\end{remark}

\begin{proof}
We first note that if $\gcd(a, c) > 1$ and $a' := a/\gcd(a,c)$, $c' := c/\gcd(a, c)$, the Chinese remainder theorem yields $S(am, an; c) = \frac{\varphi(c)}{\varphi(c')} S(a'm, a'n; c')$. Since $\frac{\varphi(c)}{\varphi(c')} \le \frac{c}{c'} = \gcd(a, c)$, one can deduce the desired bound from the same bound for $(a, c) \mapsto (a', c')$. This allows us to assume without loss of generality that $\gcd(a, c) = 1$.

Let $\mS$ denote the sum in \cref{prop:bilinear-expo-phases}; as in \cref{eq:kloosterman-fourier}, we expand $S(am, an; c)$ and swap sums to obtain
\[
    \mS = \sum_{x \in (\Z/c\Z)^\times} 
    \sum_{m \in I} e\left(m\alpha + m\frac{ax}{c}\right)
    \sum_{n \in J} e\left(n\beta + n\frac{a\bar{x}}{c}\right).
\]
We note that
\[
    \sum_{m \in I} e\left(m\alpha + m\frac{ax}{c}\right) \ll 
    \min\left(M, \left\| \alpha + \frac{ax}{c}\right\|^{-1}\right),
\]
and put $M \|\alpha + \tfrac{ax}{c}\|$ into dyadic ranges
\[
    A_0 := \left[0, 2\right], \qquad\qquad 
    A_j := \left(2^j, 2^{j+1} \right].
\]
Proceeding similarly for the sum over $n$, we obtain
\[
\begin{aligned}
    \mS &= 
    \sum_{\substack{0 \le j \le \log_2 M \\ 0 \le k \le \log_2 N}}
    \sum_{\substack{x \in (\Z/c\Z)^\times \\ M \|\alpha + \frac{ax}{c}\| \in A_j \\ N \|\beta + \frac{a\bar{x}}{c}\| \in A_k}}
    \sum_{m \in I} e\left(m\alpha + m\frac{ax}{c}\right)
    \sum_{n \in J} e\left(n\beta + n\frac{a\bar{x}}{c}\right)
    \\
    &\ll 
    \sum_{\substack{0 \le j \le \log_2 M \\ 0 \le k \le \log_2 N}} 
    \sum_{x \in (\Z/c\Z)^\times}
    \one_{M \|\alpha + \frac{ax}{c}\| \in A_j}
    \one_{N \|\beta + \frac{a\bar{x}}{c}\| \in A_k}\,
    \frac{M N}{2^{j+k}}
    \\
    &\le 
    \sum_{\substack{0 \le j \le \log_2 M \\ 0 \le k \le \log_2 N}} \frac{M N}{2^{j+k}}
    \sum_{\substack{x, y \in \Z/c\Z \\ xy \equiv a^2 \pmod{c}}}
    \one_{M\|\alpha + \frac{x}{c}\| \in A_j}\,
    \one_{N\|\beta + \frac{y}{c}\| \in A_k}.
    \\
    &\le 
    \sum_{\substack{0 \le j \le \log_2 M \\ 0 \le k \le \log_2 N}} \frac{M N}{2^{j+k}}
    \sum_{\substack{x, y \in \Z \\ xy \equiv a^2 \pmod{c}}}
    \one_{|x + c\alpha| \le c\frac{2^{j+1}}{M}}\,
    \one_{|y + c\beta| \le c\frac{2^{k+1}}{N}},
\end{aligned}
\]
where we noted that for any $x_0, y_0 \in \Z/c\Z$, there exist $x, y \in \Z$ with $x \equiv x_0 \pmod{c}$, $y \equiv y_0 \pmod{c}$, and $\|\alpha + \frac{x_0}{c}\| = |\alpha + \frac{x}{c}|$, $\|\beta + \frac{y_0}{c}\| = |\beta + \frac{y}{c}|$.

We can bound the inner sum using \cref{lem:cilleruelo-garaev} with $X = c 2^{j+2} M^{-1}$, $Y = c 2^{k+2} N^{-1}$, and $\lambda = a^2$. Since the function $T_{M,N}$ is non-decreasing in $M, N$, this gives
\[
\begin{aligned}
    \mS &\ll_\eps c^\eps \sum_{\substack{0 \le j \le \log_2 M \\ 0 \le k \le \log_2 N}} \frac{MN}{2^{j+k}} \left(\frac{(c 2^j M^{-1})(c 2^k N^{-1})}{c} T_{\frac{M}{2^{j+1}}, \frac{N}{2^{k+1}}}\left(-\alpha, -\beta\right) + \gcd(a^2, c)\right)
    \\
    &\ll_\eps 
    c^{2\eps} \left(c T_{M, N}\left(\alpha, \beta\right) + MN\right).
\end{aligned}
\]
This yields the desired bound up to a rescaling of $\eps$.
\end{proof}

\begin{proposition}[Bilinear Kloosterman bound with frequency concentration] \label{prop:bilinear-freq-concentration}
Let $c, a \in \Z_+$, $1 \ll M, N \ll c$, and $I, J \subset \Z$ be nonempty discrete intervals of lengths $|I| = M$, $|J| = N$. Let $(a_m)_{m \in I}, (b_n)_{n \in J}$ be complex sequences, and $\mu, \nu$ be bounded-variation complex Borel measures on $\R/\Z$, such that $\check{\mu}(m) = a_m$ for $m \in I$ and $\check{\nu}(n) = b_n$ for $n \in J$. Then for any $\eps > 0$, one has
\begin{equation} \label{eq:bilinear-freq-concentration}
    \sum_{m \in I} a_m \sum_{n \in J} b_n\, S(am, an; c) 
    \ll_\eps c^\eps
    \iint_{(\R/\Z)^2} \left(c\, T_{M,N}(\alpha, \beta) + \gcd(a, c) MN\right) d|\mu|(\alpha)\, d|\nu|(\beta),
\end{equation}
By \cref{eq:tn-bound}, when $M = N$, this bound is $\ll c^\eps (cN^{2/3} + \gcd(a, c) N^2)\, |\mu|(\R/\Z)\, |\nu|(\R/\Z)$.
\end{proposition}

\begin{proof}
By Fourier inversion, expand
\[
    a_m = \int_{\R/\Z} e(m\alpha)\, d\mu(\alpha),
    \qquad\qquad 
    b_n = \int_{\R/\Z} e(n\beta)\, d\mu(\beta),
\]
then swap sums and integrals, and apply \cref{prop:bilinear-expo-phases}.
\end{proof}

\begin{remark}
Suppose $M = N$ and $a = 1$.
By comparison, the pointwise Weil bound would yield a right-hand side in \cref{eq:bilinear-freq-concentration} of roughly $N \sqrt{c}\, \|a_m\|_2 \|b_n\|_2$, while applying Cauchy--Schwarz after \cref{eq:kloosterman-fourier} gives the bound $c \|a_m\|_2 \|b_n\|_2$ (these essentially lead to the ranges in \cref{thm:large-sieve-general}). It is a very difficult problem \cite{kowalski2017bilinear,kerr2023bounds} to improve these bounds for general sequences $(a_m), (b_n)$, but it becomes easier given suitable information in the frequency space. Indeed, with the natural choice of measures $d\mu = \hat{a}\, d\lambda$, $d\nu = \hat{b}\, d\lambda$ (where $\lambda$ is the Lebesgue measure), \cref{prop:bilinear-freq-concentration} saves over the relevant bound $c \|a_m\|_2 \|b_n\|_2 = c \|\hat{a}\|_{L^2} \|\hat{b}\|_{L^2}$ whenever $\hat{a}, \hat{b}$ satisfy the concentration inequality
\[
    \frac{\|\hat{a}\|_{L^1}}{\|\hat{a}\|_{L^2}} \cdot \frac{\|\hat{b}\|_{L^1}}{\|\hat{b}\|_{L^2}} = o\left(\frac{1}{N^{2/3} + N^2 c^{-1}}\right).
\]
For reference, the left-hand side is always $\gg N^{-1}$.
One may do better by treating the integral in \cref{eq:bilinear-freq-concentration} more carefully, or by including the contribution of other frequencies into $\mu$ and $\nu$ (this liberty is due to the handling of sharp cutoffs in \cref{prop:bilinear-expo-phases}). For instance, one could extend the sequences $(a_m), (b_n)$ with a smooth decay beyond $I$ and $J$ before taking their Fourier transforms, or one could construct $\mu, \nu$ out of Dirac delta measures (in particular, one recovers \cref{prop:bilinear-expo-phases} this way).
\end{remark}

We will ultimately use \cref{prop:bilinear-freq-concentration} for sequences $(a_n)$ of the shape in \cref{eq:dispersion-coeffs}, so it is necessary to understand their Fourier transforms. The case of exponential phases $a_n = e(n\alpha)$ is trivial, but the dispersion coefficients from \cref{thm:large-sieve-dispersion-coeffs} are more interesting, warranting a separate lemma.

\begin{lemma}[Fourier transform of dispersion coefficients] \label{lem:fourier-dispersion}
Let $\eps > 0$ and $H, L \gg 1$.
For $i \in \{1, 2\}$, let $\ell_i \in \Z_+$ with $\ell_i \asymp L$ and $(\ell_1, \ell_2) = 1$, $\alpha_i \in \R/\Z$,
and $\Phi_i : (-\infty, \infty) \to \C$ be smooth functions, with $\Phi_i(t)$ supported in $t \ll 1$ and $\Phi_i^{(j)} \ll_j 1$ for all $j \ge 0$. Then for any $\eps > 0$, the sequence 
\[
    a_n := \sum_{\substack{h_1, h_2 \in \Z \\ h_1 \ell_1 + h_2 \ell_2 = n}} \Phi_1\left(\frac{h_1}{H}\right) \Phi_2\left(\frac{h_2}{H}\right) e(h_1 \alpha_1 + h_2 \alpha_2),
\]
supported in $n \ll HL$, has Fourier transform bounds
\begin{equation} \label{eq:fourier-dispersion-bounds}
    \hat{a} \ll H^2, 
    \qquad\qquad 
    \hat{a}(\alpha) \ll_\eps H^{-100} \ 
    \text{ unless } \
    \|\ell_i \alpha - \alpha_i\| \le H^{\eps-1}\ 
    \forall i \in \{1, 2\}.
\end{equation}
In consequence,
\[
    \|\hat{a}\|_{L^1} \ll_\eps H^{\eps} \left(1 + \frac{H}{L}\right),
    \qquad\qquad 
    \|\hat{a}\|_{L^2} \ll_\eps H^{\eps} \left(H + \frac{H^{3/2}}{L^{1/2}}\right).
\]
\end{lemma}

\begin{proof}[Proof of \cref{lem:fourier-dispersion}]
We take $\eps \in (0, 1)$ without loss of generality.
The sequence $(a_n)$ can be expressed as a discrete convolution,
\begin{equation} \label{eq:a-split}
    a_n = a(n) = \sum_{m \in \Z} b_1(m)\, b_2(n-m)
    \qquad\quad 
    \Rightarrow 
    \qquad\quad 
    \hat{a}(\alpha) = \hat{b_1}(\alpha) \cdot \hat{b_2}(\alpha),
\end{equation}
where for $i \in \{1, 2\}$, 
\[
    b_i(n) := \one_{n \equiv 0 \pmod{\ell_i}}\,  \Phi_i\left(\frac{n}{H \ell_i}\right) e\left(\frac{n}{\ell_i} \alpha_i\right).
\]
But we further have
\begin{equation} \label{eq:b-split}
    \hat{b_i}(\alpha) = \hat{c_i}(\ell_i \alpha - \alpha_i),
\end{equation}
where $c_i(h) := \Phi_i(h/H)$. By Poisson summation and the Schwarz decay of $\hat{\Phi}_i$, identifying $\alpha \in \R/\Z$ with $\alpha \in (-1/2, 1/2]$, we have
\[
\begin{aligned}
    \hat{c_i}(\alpha) = \sum_{h \in \Z} \Phi_i\left(\frac{h}{H}\right) e(-h\alpha) 
    &= 
    \sum_{n \in \Z} H\hat{\Phi}_i\left(H(n + \alpha) \right)
    \\
    &=
    H\hat{\Phi}_i(H\alpha) + O\left(H^{-200}\right).
\end{aligned}
\]
In fact, we also have $H\hat{\Phi_i}(H\alpha) = O_\eps(H^{-200})$ when $|H\alpha| > H^{\eps}$. So overall,
\[
\begin{aligned}
    \hat{c_i}(\alpha) 
    &\ll H, \qquad\qquad\quad\ \ \, \forall\, \alpha \in \R/\Z,
    \\
    \hat{c_i}(\alpha) 
    &\ll
    O_\eps(H^{-200}), \qquad \text{if } \|\alpha\| > H^{\eps-1}.
\end{aligned}
\]
Thus by \cref{eq:a-split,eq:b-split}, we obtain
\begin{equation} \label{eq:a-fourier-bound}
    \hat{a}(\alpha) 
    \ll 
    \begin{cases}
        H^2, & \max\left(\|\ell_1 \alpha - \alpha_1\|, \|\ell_2 \alpha - \alpha_2\|\right) \le H^{\eps-1}, \\
        O_\eps(H^{-100}), & \max\left(\|\ell_1 \alpha - \alpha_1\|, \|\ell_2 \alpha - \alpha_2\|\right) > H^{\eps-1},
    \end{cases}
\end{equation}
which proves \cref{eq:fourier-dispersion-bounds}.
Now suppose that $\max\left(\|\ell_1 \alpha - \alpha_1\|, \|\ell_2 \alpha - \alpha_2\|\right) \le H^{\eps-1}$; we would like to estimate how often this happens. Identifying $\alpha, \alpha_i \in \R/\Z$ with $\alpha, \alpha_i \in (-1/2, 1/2]$, there must exist integers $m_i(\alpha) \ll L$ such that
\[
    \ell_1 \alpha - \alpha_1 = m_1 + O\left(H^{\eps-1}\right), 
    \qquad\qquad 
    \ell_2 \alpha - \alpha_2 = m_2 + O\left(H^{\eps-1}\right),
\]
so in particular,
\begin{equation} \label{eq:m-pair}
    \ell_1 m_2 - \ell_2 m_1 = \ell_2\alpha_1 - \ell_1 \alpha_2
     + O\left(H^{\eps-1} L\right).
\end{equation}
Since $\gcd(\ell_1, \ell_2) = 1$, as $m_1, m_2 \ll L$ vary, the difference $\ell_1 m_2 - \ell_2 m_1$ can only cover any given integer $O(1)$ times; thus there are a total of $O(1 + H^{\eps-1} L)$ pairs $(m_1, m_2) \in \Z^2$ satisfying \cref{eq:m-pair}. Moreover, to each such pair $(m_1, m_2)$ there can correspond an interval of $\alpha$'s of length at most $O(H^{\eps-1} L^{-1})$, since 
\[
    \alpha = \frac{m_1(\alpha) + \alpha_1}{\ell_1} + O\left(H^{\eps-1} L^{-1}\right).
\]
Overall, we obtain that the set
\[
    \left\{\alpha \in \R/\Z : \max\left(\|\ell_1 \alpha - \alpha_1\|, \|\ell_2 \alpha - \alpha_2\|\right) \le H^{\eps-1} \right\}
\]
has Lebesgue measure at most
\[
    O\left(\left(1 + H^{\eps-1} L\right) \cdot H^{\eps-1}L^{-1}\right) 
    =
    O\left(H^{\eps-1} L^{-1} + H^{2\eps-2}\right).
\]
By \cref{eq:a-fourier-bound}, we conclude that for any $p \ge 1$,
\[
\begin{aligned}
    \|\hat{a}\|_{L^p} &\ll_\eps H^{O(\eps)} \left(H^{2p} \left(H^{-1} L^{-1} + H^{-2}\right) + 1\right)^{\frac{1}{p}}
    \\
    &\ll_p H^{O(\eps)} \left(H^{2-\frac{2}{p}} + \frac{H^{2-\frac{1}{p}}}{L^{\frac{1}{p}}}\right),
\end{aligned}
\]
which completes our proof up to a rescaling of $\eps$.
\end{proof}

\begin{remark}
As in \cite{shparlinski2018character}, the arguments in this subsection extend immediately to sums of weighted Kloosterman sums %(such as the Sali\'e sums \cite{kerr2021energy})
\[
    S_w(m, n; c) := \sum_{x \in (\Z/c\Z)^\times} w(x)\, e\left(\frac{mx + n\bar{x}}{c}\right),
\]
for arbitrary $1$-bounded coefficients $w(x)$. In particular, choosing $w(x)$ in terms of a Dirichlet character $\chi$ mod $q_0$, where $q_0 \mid q \mid c$, should ultimately extend our large sieve inequalities to the exceptional Maass forms of level $q$ associated to a general nebentypus $\chi$, rather than the trivial one.
\end{remark}

\section{Spectral bounds}
\label{sec:spectral}

We now combine the combinatorial arguments from the previous section with techniques from the spectral theory of automorphic forms (inspired by \cite{deshouillers1982kloosterman}), to prove new large sieve inequalities for exceptional Maass cusp forms, and then to deduce bounds for multilinear forms of Kloosterman sums. The reader should be familiar with the prerequisites in all of \cref{sec:notation}, especially \cref{subsec:automorphic-fourier-bounds}.

\subsection{A general large sieve for exceptional Maass forms}
\label{subsec:large-sieve}
Our common generalization of \cref{thm:large-sieve-expo-phases,thm:large-sieve-dispersion-coeffs} requires the following notation, applied to the Fourier transform of a sequence $(a_n)$. 

\begin{notation}[Rational-approximation integrals] \label{not:concentration}
Given $N \ge 1/2$ and a bounded-variation complex Borel measure $\mu$ on $\R/\Z$, we denote
\[
    \mI_N(\mu) := \iint_{(\R/\Z)^2} T_N(\alpha, \beta)\,
        d|\mu|(\alpha)\, d|\mu|(\beta),
\]
recalling the definition of $T_N(\alpha, \beta)$ from \cref{not:tn}.
\end{notation}

In general, the bound in \cref{eq:tn-bound} ensures that
\begin{equation} \label{eq:I-bound-1}
    \mI_N(\mu) \ll \iint_{(\R/\Z)^2} \min \left(\sqrt{N(1 + \|\alpha-\beta\| N)}, N^{2/3} \right)\,
        d|\mu|(\alpha)\, d|\mu|(\beta),
\end{equation}
which is invariant under translations of $\mu$. Noting the trivial lower bound $T_N(\alpha, \beta) \ge 1$, this implies
\begin{equation} \label{eq:I-bound-2}
    |\mu|(\R/\Z)^2 \ll \mI_N(\mu) \ll 
    N^{2/3} |\mu|(\R/\Z)^2.
\end{equation}

\begin{theorem}[Large sieve with frequency concentration] \label{thm:large-sieve-freq-concentration}
Let $\eps > 0$, $X, A > 0$, $N \ge 1/2$, $q, a \in \Z_+$, and $(a_n)_{n \sim N}$ be a complex sequence. Let $f : (0, 4) \to \C$ be a smooth function with $f^{(j)} \ll_j 1$ for $j \ge 0$, and $\mu$ be a bounded-variation complex Borel measure on $\R/\Z$, such that\footnote{We slightly abuse notation in this section: the measure $\mu$ should not be confused with the cusp parameter $\mu(\ma) = q^{-1}$, and the scalar $a$ should not be confused with the sequence $(a_n)$.}
\[
    a_n = f\left(\frac{n}{N}\right) \check{\mu}\left(n\right),
\]
for all $n \sim N$ (in particular, one can take $f \equiv 1$, $d\mu = \hat{a}\, d\lambda$). Let $\ma, \rho_{j\ma}(n), \lambda_j,\theta_j$ be as in \cref{thm:large-sieve-general}, with $\mu(\ma) = q^{-1}$ and the choice of scaling matrix $\sigma_{\ma}$ in \cref{eq:scaling-choices}. Then one has
\begin{equation} \label{eq:large-sieve-freq-concentration}
    \sum_{\lambda_j < 1/4}
    X^{2\theta_j} 
    \left\vert 
    \sum_{n \sim N} a_n\, \rho_{j\ma}(an)
    \right\vert^2 
    \ll_\eps 
    (qaNX)^{2\eps}
    \left(1 + \frac{aN}{q}\right) A^2,
\end{equation}
provided that both of the following hold:
\begin{gather} 
    \label{eq:A-condition-freq}
    A \gg \|a_n\|_2 + \frac{\sqrt{\gcd(a, q)} N}{\sqrt{q+aN}} |\mu|(\R/\Z),
    \\ \label{eq:X-condition-freq}
    X \ll \max\left(1, \frac{q}{aN}\right) \max\left(1, \frac{A^2}{\mI_N(\mu)}\right).
\end{gather}
\end{theorem}

\begin{remark}
\cref{thm:large-sieve-freq-concentration} obtains a saving over \cref{thm:large-sieve-general} whenever we can take $A \ll (qN)^{o(1)}\|a_n\|_2$ and $X > \max(1, \tfrac{q}{aN})$. To satisfy \cref{eq:X-condition-freq,eq:A-condition-freq} in this context, assuming $\gcd(a, q) = 1$, we need
\begin{equation} \label{eq:some-freq-concentration}
    |\mu|(\R/\Z) 
    \ll 
    (qN)^{o(1)} \frac{\sqrt{q+aN}}{N} \|a_n\|_2
    \qquad\quad 
    \text{and}
    \qquad\quad
    \mI_N(\mu) = o\left(\|a_n\|_2^2\right).
\end{equation}
These should be compared with the lower bound
\begin{equation} \label{eq:mu-lower-bound}
    |\mu|(\R/\Z) \gg N^{-1/2}\, \|a_n\|_2,
\end{equation}
which always holds, by Fourier expansion and Cauchy--Schwarz.
This has the following implications:
\begin{itemize}
    \item[(1).] From \cref{eq:mu-lower-bound} and the lower bound in \cref{eq:I-bound-2}, we have $\mI_N(\mu) \gg N^{-1/2} \|a_n\|_2$. With $A \ll (qN)^{o(1)}\|a_n\|_2$, this limits the range of $X$ in \cref{eq:X-condition-freq} to the best-case scenario $X \ll \max(N, \tfrac{q}{a})$. This is indeed achieved by \cref{thm:large-sieve-expo-phases} when $\alpha = 0$.
    \item[(2).] When $a \ll 1$ and $q \approx N$, \cref{eq:some-freq-concentration} requires nearly-optimal concentration for $\mu$, in the sense that $|\mu|(\R/\Z)$ is almost as small as possible; this happens to hold for the sequences in \cref{eq:disp-coeff-informal}.
    \item[(3).] Using the upper bound $\mI_N(\mu) \ll N^{2/3}|\mu|(\R/\Z)^2$ from \cref{eq:I-bound-2} and choosing $f \equiv 1$, $d\mu = \hat{a}\, d\lambda$ (so that $|\mu|(\R/\Z) = \|\hat{a}\|_{L^1}$ and $\|a_n\|_2 = \|\hat{a}\|_{L^2}$), we see that \cref{eq:some-freq-concentration} holds in particular when
    \[
        \frac{\|\hat{a}\|_{L^1}}{\|\hat{a}\|_{L^2}} = o\left(\frac{\min(q^{1/2+o(1)}, (aN)^{1/2+o(1)}, N^{2/3})}{N}\right),
    \]
    which gives a more palpable concentration condition on the Fourier transform $\hat{a}$. The weights of $T_N(\alpha, \beta)$ inside $\mI_N(\mu)$, combined with the liberty to choose other measures $\mu$ and functions $f$, allow for additional flexibility when more information about the sequence $(a_n)$ is available. 
\end{itemize}
\end{remark}

\begin{proof}[Proof of \cref{thm:large-sieve-freq-concentration}]
We assume without loss of generality that $\eps < 1$, and that $f$ is supported in $[0.5, 3]$ (otherwise, multiply $f$ by a fixed smooth function supported in $[0.5, 3]$ and equal to $1$ on $[1, 2]$; then the identity $a_n = f(n/N)\, \check{\mu}(n)$ remains true for $n \sim N$).

In light of \cref{lem:large-sieve-non-exceptional}, we are immediately done if $X \le 1$, so assume $X > 1$. Let $\Phi$ be a fixed nonnegative smooth function supported in $[2, 4]$, with positive integral. Then by \cref{cor:exc-bound} and the fact that $A \gg \|a_n\|_2$ (from \cref{eq:A-condition-freq}), it suffices to show that
\begin{equation} \label{eq:to-show-large-sieve}
    \mS := \sum_{c \in \mC_{\ma\ma}} \frac{1}{c} 
    \sum_{m,n \sim N} \bar{a_m}\, a_n\, S_{\ma \ma}(am, an; c)\, \Phi\left(\frac{a\sqrt{mn}}{c} X \right)
    \ll_\eps 
    (qaNX)^{2\eps} \left(1 + \frac{aN}{q}\right) A^2,
\end{equation}
subject to \cref{eq:A-condition-freq,eq:X-condition-freq}. Since $\mu(\ma) = q^{-1}$, \cref{lem:explicit-kloosterman} implies that
\begin{equation} \label{eq:S-split}
    \mS = \sum_{\substack{c \in (aNX/4, aNX) \\ c \equiv 0 \pmod{q} \\ }} \frac{\mS(c)}{c},
\end{equation}
where 
\[
    \mS(c) := \sum_{m,n \sim N} \bar{a_m} \, a_n\, S(am, an; c)\, \Phi\left(\frac{a\sqrt{mn}}{c} X \right).
\]
If $aNX \le q$, the sum over $c$ is void; so we may assume that $X > \max\left(1, \frac{q}{aN}\right)$, which by \cref{eq:X-condition-freq} implies
\begin{equation} \label{eq:wlog-conc}
    \mI_N(\mu) \ll A^2.
\end{equation}
We aim to bound each of the $\asymp aNX/q$ inner sums $\mS(c)$
separately, using \cref{prop:bilinear-freq-concentration}. To this end, we need to separate the variables $m, n, c$; we can rewrite
\begin{equation} \label{eq:sc-again}
    \mS(c) = \sum_{m,n \sim N} \bar{\check{\mu}(m)} \, \check{\mu}(n)\, S(am, an; c)\, \Psi_c\left(\frac{m}{N}, \frac{n}{N}\right),
\end{equation}
where 
\[
    \Psi_c(x_1, x_2) := \bar{f(x_1)} f(x_2)\, \Phi\left(\sqrt{x_1 x_2} \frac{aNX}{c} \right)
\]
is a compactly-supported smooth function with bounded derivatives (since $c \asymp aNX$ and we assumed WLOG that $f$ is supported in $[0.5, 3]$).
By two-dimensional Fourier inversion, we have
\[
    \Psi_c(x_1, x_2) = \iint_{\R^2} \hat{\Psi_c}(t_1, t_2)\, e(t_1 x_1 + t_2 x_2)\, dt_1\, dt_2,
\]
where
\[
    \hat{\Psi_c}(t_1, t_2) = \iint_{(0, \infty)^2} \Psi_c(x_1, x_2)\, e(- t_1 x_1 - t_2 x_2)\, dx_1\, dx_2.
\]
Since $\Psi_c(x_1, x_2)$ is Schwarz, so is $\hat{\Psi}(t_1, t_2)$; in particular, we have $\hat{\Psi_c}(t_1, t_2) \ll (1+t_1^4)^{-1}(1+t_2^4)^{-1}$ with an absolute implied constant. Plugging the inversion formula into \cref{eq:sc-again} and swapping sums and integrals, we obtain
\begin{equation} \label{eq:Sc-split}
    \mS(c) = \iint_{\R^2} \hat{\Psi_c}(t_1, t_2)\, \mS(c, t_1, t_2)\, dt_1\, dt_2,
\end{equation}
where
\[
    \mS(c, t_1, t_2) := \sum_{m,n \sim N} \bar{\check{\mu}(m)\, e\left(\frac{-mt_1}{N}\right)} \, \check{\mu}(n)\, e\left(\frac{nt_2}{N}\right) S(am, an; c).
\]
Note that translating $\mu$ corresponds to multiplying $\check{\mu}(n)$ by exponential factors $e(n\alpha)$, so \cref{prop:bilinear-freq-concentration} and a change of variables yield
\[
\begin{aligned}
    &\mS(c, t_1, t_2) 
    \\
    &\ll_\eps 
    c^\eps \iint_{(\R/\Z)^2} \left(c T_N(\alpha, \beta) + \gcd(a, c) N^2 \right) d|\mu|\left(-\alpha+\frac{t_1}{N}\right) d|\mu|\left(\beta-\frac{t_2}{N}\right)
    \\
    &=
    c^\eps \iint_{(\R/\Z)^2} \left(c T_N\left(-\alpha + \frac{t_1}{N}, \beta + \frac{t_2}{N}\right) + \gcd(a, c) N^2 \right) d|\mu|(\alpha)\, d|\mu|(\beta).
\end{aligned}
\]
Recalling that $T_N(\alpha, \beta) = T_N(-\alpha, \beta)$ and the bound \cref{eq:T-shift}, we have
\[
    T_N\left(-\alpha + \frac{t_1}{N}, \beta + \frac{t_2}{N}\right)
    \ll 
    (1 + |t_1|)(1 + |t_2|)\, T_{N}(\alpha, \beta),
\]
so that
\[
    \mS(c, t_1, t_2) \ll_\eps (1 + |t_1|)(1 + |t_2|)\,c^\eps\, \left(c\, \mI_{N} (\mu) + \gcd(a, c) N^2 |\mu|(\R/\Z)^2 \right).
\]
Together with \cref{eq:Sc-split} and the bound $\hat{\Psi_c}(t_1, t_2) \ll (1+t_1^4)^{-1}(1+t_2^4)^{-1}$, we obtain
\[
    \mS(c) \ll_\eps c^\eps\, \left(c\, \mI_N \left(\mu\right) + \gcd(a, c) N^2 |\mu|(\R/\Z)^2 \right),
\]
and by \cref{eq:S-split} we conclude that
\begin{equation} \label{eq:final-s-bound}
    \mS \ll_\eps (aNX)^{2\eps} \left(\frac{aNX}{q} \mI_{N} (\mu) + \frac{\gcd(a, q) N^2}{q} |\mu|(\R/\Z)^2 \right).
\end{equation}
By the assumed lower bound for $A$ in \cref{eq:A-condition-freq}, the contribution of the second term is 
\[
    \ll_\eps (aNX)^{2\eps} \left(1 + \frac{aN}{q}\right) A^2,
\]
which is acceptable in \cref{eq:to-show-large-sieve}. Similarly, the first term in \cref{eq:final-s-bound} is acceptable provided that
\[
    \frac{aNX}{q} \mI_N(\mu) \ll \left(1 + \frac{aN}{q}\right) A^2,
\]
i.e.,
\[
    X \ll \max\left(1, \frac{q}{aN}\right) \frac{A^2}{\mI_N(\mu)},
\]
which follows from \cref{eq:X-condition-freq} and \cref{eq:wlog-conc}.
\end{proof}

\subsection{Proofs of \cref{thm:large-sieve-expo-phases,thm:large-sieve-dispersion-coeffs}} We can now deduce the large sieve inequalities promised in \cref{subsec:intro-large-sieve}, starting from \cref{thm:large-sieve-freq-concentration}.

\begin{proof}[Proof of \cref{thm:large-sieve-expo-phases}]
Consider the sequence $a_n := \Phi(n/N)\, e(n\alpha)$ for $n \sim N$ and some $\alpha \in \R/\Z$, which has $\|a_n\|_2 \asymp \sqrt{N} =: A$. Choosing $\mu := \delta_{\{\alpha\}}$, we have $a_n = \Phi(n/N)\, \check{\mu}(n)$ for $n \sim N$, and $|\mu|(\R/\Z) = 1$. In particular, the lower bound for $A$ in \cref{eq:A-condition-freq} holds for any values of $q$ and $a$, since
\[
    |\mu|(\R/\Z) = 1 \ll N^{-1/2}\, \|a_n\|_2.
\]
Finally, we have
\[
    \mI_{N}(\mu) = T_N(\alpha, \alpha) \asymp \min_{t \in \Z_+} \left(t + N\|t\alpha\|\right),
\]
so \cref{thm:large-sieve-freq-concentration} (i.e., \cref{eq:X-condition-freq}) recovers the large sieve range
\[
    X \ll \max\left(1, \frac{q}{aN}\right) \frac{N}{\min_{t \in \Z_+} \left(t + N\|t\alpha\|\right)}
\]
from \cref{eq:expo-X-range}. In particular, we can recall from \cref{eq:tn-bound} that $T_N(\alpha, \alpha) \ll \sqrt{N}$, so this includes the range $X \ll (\sqrt{N}, \tfrac{q}{a\sqrt{N}})$ uniformly in $\alpha$. Since varying the choice of scaling matrix $\sigma_{\ma}$ is equivalent to varying $\alpha$, we can use the same range $X \ll (\sqrt{N}, \tfrac{q}{a\sqrt{N}})$ for an arbitrary scaling matrix.
\end{proof}

\begin{proof}[Proof of \cref{thm:large-sieve-dispersion-coeffs}]
Assume without loss of generality that $\eps \in (0, 1)$. By changing $h_2 \leftrightarrow -h_2$, $\Phi_2(t) \leftrightarrow \Phi_2(-t)$ and $\alpha_2 \leftrightarrow -\alpha_2$, we can equivalently consider the sequence $(a_n)_{n \sim N}$ given by
\[
    a_n = 
    \sum_{\substack{h_1, h_2 \in \Z \\ h_1\ell_1 + h_2\ell_2 = n}}
    \Phi_1\left(\frac{h_1}{H}\right) \Phi_2\left(\frac{h_2}{H}\right) e(h_1 \alpha_1 + h_2 \alpha_2).
\]
We may of course assume that $N \ll HL$, since otherwise $(a_n)_{n\sim N}$ vanishes. Note that the extension $(a_n)_{n \in \Z}$ is exactly the sequence considered in \cref{lem:fourier-dispersion}. Thus letting $\varphi : \R/\Z \to \C$ be the Fourier transform of $(a_n)_{n \in \Z}$, and $d\mu := \varphi\, d\lambda$ (where $\lambda$ is the Lebesgue measure on $\R/\Z$), we have
\[
    \check{\mu}(n) = 
    \check{\varphi}(n) =
    a_n, \qquad \forall n \sim N.
\]
Moreover, \cref{lem:fourier-dispersion} implies that 
\begin{equation} \label{eq:phi-fourier-dispersion-bounds}
    \varphi \ll H^2, 
    \qquad\qquad 
    \varphi(\alpha) \ll_\eps H^{-100} \ 
    \text{ unless } \
    \|\ell_i \alpha - \alpha_i\| \le H^{\eps-1}\ 
    \forall i \in \{1, 2\},
\end{equation}
and
\begin{equation} \label{eq:mu-total-mass}
    |\mu|(\R/\Z) = \|\varphi\|_{L^1} \ll_\eps H^\eps \left(1 + \frac{H}{L}\right).
\end{equation}
To compute the integral 
\[
    \mI_N(\mu)
    =
    \iint_{(\R/\Z)^2} T_N(\alpha, \beta)\, \varphi(\alpha)\, \varphi(\beta)\, d\alpha\, d\beta,
\]
we first consider the contribution of $\alpha, \beta$ which have $\|\ell_i \alpha - \alpha_i\| >  H^{\eps-1}$ or $\|\ell_i \beta - \alpha_i\| > H^{\eps-1}$ for some $i \in \{1, 2\}$. By \cref{eq:phi-fourier-dispersion-bounds}, either $\varphi(\alpha)$ or $\varphi(\beta)$ is $\ll_\eps H^{-100}$ in this case, so the total contribution to $\mI_{N}(\mu)$ is
\[
    \ll_\eps N^{2/3} H^{-100} H^2
    \ll 
    L H^{-90}.
\]
On the other hand, when $\max_{i \in \{1,2\}} \max(\|\ell_i \alpha - \alpha_i\|, \|\ell_i \beta - \alpha_i\|) \le H^{\eps-1}$, we have by definition (\cref{not:tn}) that for any $t \in \Z_+$,
\[
\begin{aligned}
    T_N(\alpha, \beta) 
    &\le 
    t\ell_i + N \|t\ell_i \alpha\| + N \|t\ell_i \beta\| 
    \\
    &\ll 
    tL + N \|t(\ell_i \alpha - \alpha_i)\| + N \|t(\ell_i \beta - \alpha_i)\| + N \|t\alpha_i\|
    \\
    &\ll
    tL + N tH^{\eps-1} + N \|t\alpha_i\|
    \\
    &\ll 
    H^\eps t L + N\|t\alpha_i\|
    \\
    &\ll
    H^\eps L \left(t + \frac{N}{L}\|t\alpha_i\|\right).
\end{aligned}
\]
Taking a minimum over $t \in \Z_+$ and $i \in \{1, 2\}$, we obtain
\[
    T_N(\alpha, \beta) \ll H^\eps L M,
    \qquad\qquad 
    M := \min_{i \in \{1, 2\}} T_{N/L}(\alpha_i).
\]
Using \cref{eq:mu-total-mass}, we conclude that
\begin{equation} \label{eq:In-total-mass}
\begin{aligned}
    \mI_{N}(\mu) 
    =
    \iint_{(\R/\Z)^2} T_{N}(\alpha, \beta)\, d|\mu|(\alpha)\, d|\mu|(\beta) 
    &\ll_{\eps}
    L H^{-90}
    +
    H^\eps L M\, |\mu|(\R/\Z)^2
    \\
    &\ll_{\eps} 
    H^{2\eps} L M\left(1 + \frac{H}{L}\right)^2.
\end{aligned}
\end{equation}
We are now in a position to apply \cref{thm:large-sieve-freq-concentration}, with
\[
    \frac{A}{C_\eps H^\eps} := \|a_n\|_2 + \sqrt{\gcd(a, q) N} \left(\sqrt{\frac{H}{L}} + \frac{H}{L}\right),
\]
where $C_\eps$ is a sufficiently large constant.
Note that by \cref{eq:mu-total-mass}, the assumption $q \gg L^2$, and the fact that $N \ll HL$, we have
\[
\begin{aligned}
    |\mu|(\R/\Z)
    &\ll C_\eps H^\eps \left(1 + \frac{H}{L}\right)
    \\
    &\ll 
    C_\eps
    H^\eps
    \frac{L + \sqrt{N}}{\sqrt{N}} \left(\sqrt{\frac{H}{L}} + \frac{H}{L}\right)
    \ll
    \frac{\sqrt{q + aN}}{\sqrt{\gcd(a, q)} N} A,
\end{aligned}
\]
so the lower bound for $A$ in \cref{eq:A-condition-freq} holds (above we used that $\frac{L}{\sqrt{N}} \sqrt{\frac{H}{L}} = \sqrt{\frac{HL}{N}} \gg 1$). It follows that the large sieve bound \cref{eq:large-sieve-freq-concentration} holds for all
\[
    X \ll \max \left(1, \frac{q}{aN}\right) \max\left(1, \frac{A^2}{\mI_{N}(\mu)}\right),
\]
where by \cref{eq:In-total-mass},
\[
    \frac{A^2}{\mI_{N}(\mu)} \gg \frac{H^{2\eps} N \left(\frac{H}{L} + \frac{H^2}{L^2}\right)}{H^{2\eps} L M\left(1 + \frac{H}{L}\right)^2} 
    =
    \frac{NH}{(H+L)LM}.  
\]
This proves \cref{eq:large-sieve-disp}.
\end{proof}

\subsection{Multilinear Kloosterman bounds} \label{subsec:multilinear-kloosterman}

In contrast to the ``vertical'' bilinear averages of Kloosterman sums $S(m, n; c)$ over $m, n$ from \cref{sec:combi} (or from \cite{kowalski2017bilinear,kerr2023bounds}), the bounds in this subsection also require ``horizontal'' averaging over the modulus $c$---crucially, with a smooth weight in this variable. Generally, it is such horizontal averages that make use of the Kuznetsov trace formula for $\Gamma_0(q)$, leading to dependencies on the spectral parameter $\theta(q) = \sqrt{\max(0, \tfrac{1}{4}-\lambda_1(q))} \le \tfrac{7}{64}$; we recall that the purpose of large sieve inequalities for the exceptional spectrum, like \cref{thm:large-sieve-freq-concentration}, is to improve the dependency on $\theta(q)$.

Throughout this subsection, we will work with sequences obeying the following condition.

\begin{assumption}[Large sieve for the tuple $(q, N, Z, (a_n)_{n \sim N}, A_N, Y_N)$] \label{ass:large-sieve}
This applies to complex sequences $(a_n)_{n \sim N}$ and parameters $q \in \Z_+$, $N \ge 1/2$, $Z \gg 1$, $A_N \gg \|a_n\|_2$, $Y_N > 0$. For any $\eps > 0, \xi \in \R$, any cusp $\ma$ of $\Gamma_0(q)$ with $\mu(\ma) = q^{-1}$ and $\sigma_{\ma}$ chosen as in \cref{eq:scaling-choices}, and any orthonormal basis of Maass cusp forms for $\Gamma_0(q)$, with eigenvalues $\lambda_j$ and Fourier coefficients $\rho_{j\ma}(n)$, one has
\begin{equation} \label{eq:large-sieve-rephrased}
    \sum_{\lambda_j < 1/4} X^{2\theta_j} \left\vert \sum_{n \sim N} e\left(\frac{n}{N} \xi \right) a_n\, \rho_{j\ma}(n) \right\vert^2 
    \ll_\eps
    (qNZ)^\eps 
    \left(1 + \frac{N}{q}\right) 
     A_N^2,
\end{equation}
for all $X \ll \max\left(1, \frac{q}{N}\right) \frac{Y_N}{1+|\xi|^2}$. %Here, $\theta_j := \sqrt{\tfrac{1}{4}-\lambda_j}$ and $\theta(q) := \max_{\lambda_j < 1/4} \theta_j(q)$.
\end{assumption}

For example, \cref{thm:large-sieve-general} shows that the tuple $(q, N, 1, (a_n)_{n \sim N}, \|a_n\|_2^2, 1)$ satisfies \cref{ass:large-sieve} for any $q \in \Z_+$, $N \ge 1/2$ and any complex sequence $(a_n)_{n \sim N}$; attaining higher values of $Y_N$ requires more information about $(a_n)$. \cref{thm:large-sieve-expo-phases} implies that another suitable choice of parameters is
\begin{equation} \label{eq:sequence-choice-1}
    a_n := e(n\alpha), \qquad\qquad Y_N := \frac{N}{T_N(\alpha)} \gg \sqrt{N},
    \qquad\qquad A_N := \sqrt{N},
\end{equation}
for any $\alpha \in \R/\Z$ and $q \in \Z_+$, $N \ge 1/2$, $Z = 1$; note that the phase $\xi/N$ can be incorporated into $\alpha$, and we implicitly used that $T_N(\alpha + \xi/N) \ll (1 + |\xi|^2)\, T_N(\alpha)$ by \cref{eq:T-shift}. Likewise, incorporating $\ell_i \xi/N$ into $\alpha_i$, \cref{thm:large-sieve-dispersion-coeffs} shows that we can choose
\begin{equation} \label{eq:sequence-choice-2}
\begin{gathered}
    a_n := 
    \sum_{\substack{h_1, h_2 \in \Z \\ h_1\ell_1 - h_2\ell_2 = n}} \Phi_1\left(\frac{h_1}{H}\right) \Phi_2\left(\frac{h_2}{H}\right) e(h_1\alpha_1 + h_2\alpha_2),
    \\
    Y_N := \max\left(1, \frac{NH}{(H+L)L \min_i T_H(\alpha_i)}\right) 
    ,
    \qquad
    A_N := \|a_n\|_2 + \sqrt{N}\sqrt{\frac{H}{L} + \frac{H^2}{L^2}},
\end{gathered}
\end{equation}
where $1 \ll L^2 \ll q$, $1 \ll H \ll Z$, $\alpha_i \in \R/\Z$, $\ell_i \asymp L$, $(\ell_1, \ell_2) = 1$, and $\Phi_i(t)$ are smooth functions supported in $t \ll 1$ with $\Phi_i^{(j)} \ll_j 1$. Other than the input from \cref{ass:large-sieve} (and implicitly \cref{thm:large-sieve-expo-phases,thm:large-sieve-dispersion-coeffs}), all arguments in this subsection are fairly standard \cite{deshouillers1982kloosterman,drappeau2017sums,de2020niveau}.

\begin{corollary}[Kloosterman bounds with averaging over $n, c$] \label{cor:kloosterman-averaging-nc}
Let $(q, N, Z, (a_n)_{n \sim N}, A_N, Y_N)$ satisfy \cref{ass:large-sieve}. Let $\eps > 0$, $C \gg 1$, $m \in \Z_+$, and $\ma, \mb$ be cusps of $\Gamma_0(q)$, with $\mu(\ma) = \mu(\mb) = q^{-1}$ and $\sigma_{\mb}$ as in \cref{eq:scaling-choices}. Let $\Phi : (0, \infty)^2 \to \C$ be a smooth function, with $\Phi(x, y)$ supported in $x, y \asymp 1$, and $\partial_x^j \partial_y^k \Phi(x, y) \ll_{j,k,\eps} Z^{j\eps}$ for $j, k \ge 0$. Then with a consistent choice of the $\pm$ sign, one has
\begin{equation} \label{eq:kloosterman-nc}
\begin{aligned}
    \sum_{n \sim N} a_n
    \sum_{c \in \mC_{\ma\mb}} \Phi\left(\frac{n}{N}, \frac{c}{C}\right) S_{\ma\mb}(m, \pm n; c)
    \ll_\eps
    (qmNCZ)^{O(\eps)}
    \left(1 + T\right)^{2\theta(q)}
    \frac{C^2 A_N}{C + \sqrt{mN}} 
    \\
    \times 
    \left(1 + \frac{mN}{C^2} +\frac{\sqrt{(q,m)m}}{q} \right)^{1/2}
    \left(1 + \frac{mN}{C^2} + \frac{N}{q}\right)^{1/2}
    ,
\end{aligned}
\end{equation}
for
\[
    T = \frac{T_0}{\sqrt{Y_N}}, 
    \qquad\qquad 
    T_0 := \frac{C}{\max\left(m, q^2 (q, m)^{-1} \right)^{1/2}\max\left(N, q\right)^{1/2}}
    \le
    \frac{C}{q^{3/2} (q,m)^{-1/2}}.
\]
\end{corollary}

\begin{remark}
The parameter $T_0$ indicates the best known dependency on $\theta = \theta(q)$ that one could achieve without our large sieve inequalities; for example, when $a_n = e(n\alpha)$ and $Y_N = \sqrt{N}$, \cref{cor:kloosterman-averaging-nc} saves a total factor of $N^{\theta/2}$ over previous bounds (and up to $N^{\theta}$ if $\alpha$ is close to a rational number of small denominator). We note that in practice, the second term in each maximum from $T_0$ is usually dominant, and the factors in the second line of \cref{eq:kloosterman-nc} are typically $\asymp 1$. 
\end{remark}

\begin{remark}
While the smooth weight in the $c$ variable is necessary here (stemming from \cref{prop:kuznetsov}), the smooth weight in $n$ only confers additional flexibility. Indeed, one can take $\Phi(x, y) = f(x) g(y)$ for compactly-supported functions $f, g : (0, \infty) \to \C$, where $f \equiv 1$ on $(1, 2)$; this effectively replaces $\Phi(n/N, c/C)$ with $g(c/C)$ in \cref{eq:kloosterman-nc}. The same remark applies to the next results.
\end{remark}

\begin{proof}[Proof of \cref{cor:kloosterman-averaging-nc}]
Denote by $\mS$ the sum in \cref{eq:kloosterman-nc}. 
Letting $\Psi(x; y) := \sqrt{x}\, \Phi(x, \sqrt{x}/y)$, we can Fourier expand
\[
    \sqrt{x}\, \Phi\left(x, \frac{\sqrt{x}}{y} \right) = \int_\R \hat{\Psi}(\xi; y) \, e(x\xi)\, d\xi,
\]
where the Fourier transform is taken in the first variable. Integrating by parts in $x$, we note that for $k \ge 0$,
\[
    \partial_y^k\, \hat{\Psi}(\xi; y) \ll_{j,\eps} \frac{Z^{O(\eps)}}{1 + \xi^4},
\]
where the implied constant in $O(\eps)$ (say, $K > 0$) does not depend on $k$. Then we can let
\[
    \varphi_{\xi}(y) := Z^{-K\eps}\, \left(1 + \xi^4 \right) \hat{\Psi}\left(\xi; y \frac{C}{4\pi \sqrt{mN}} \right) \frac{4\pi \sqrt{mN}}{C y},
\]
which is supported in $y \asymp X^{-1}$ and satisfies $\varphi_\xi^{(k)} \ll_{k,\eps} X^k$, for
\begin{equation} \label{eq:X-CmN}
    X := \frac{C}{\sqrt{mN}}.
\end{equation}
This way, we can rewrite
\[
\begin{aligned}
    \Phi\left(\frac{n}{N}, \frac{c}{C}\right)
    &=
    \int_\R \sqrt{\frac{N}{n}} \hat{\Psi}\left(\xi; \frac{C}{c} \sqrt{\frac{n}{N}}\right) e\left(\frac{n}{N} \xi \right)\, d\xi
    \\
    &=
    Z^{K\eps} \frac{C}{c}
    \int_\R \frac{1}{1 + \xi^4} e\left(\frac{n}{N} \xi \right) \varphi_\xi\left(\frac{4\pi \sqrt{mn}}{c}\right) d\xi,
\end{aligned}
\]
and thus
\begin{equation} \label{eq:S-split-xi}
    \mS \ll_\eps Z^{O(\eps)} C \int_\R \frac{|S(\xi)|\, d\xi}{1+\xi^4}, 
\end{equation}
where
\[
    \mS(\xi) := \sum_{n \sim N} e\left(\frac{n}{N} \xi\right) a_n
    \sum_{c \in \mC_{\ma\mb}}\,
    \frac{\mS_{\ma\mb}(m, \pm n; c)}{c} 
    \varphi_\xi\left(\frac{4\pi \sqrt{mn}}{c}\right).
\]
The inner sum is in a suitable form to apply the Kuznetsov trace formula from \cref{prop:kuznetsov}. We only show the case when the choice of the $\pm$ sign is positive; the negative case is analogous (and in fact simpler due to the lack of holomorphic cusp forms). The resulting contribution of the Maass cusp forms to $\mS(\xi)$ is
\[
\begin{aligned}
    \mS_{\mM}(\xi) 
    &\ll \sum_{j=1}^\infty \frac{|\hat{\mB}_{\varphi_\xi}(\kappa_j)|}{\cosh(\pi \kappa_j)} 
    |\rho_{j\ma}(m)| \left\vert \sum_{n \sim N} e\left(\frac{n}{N} \xi\right) a_n\,  
    \rho_{j\mb}(n)\right\vert
    =:
    \mS_{\mM,\text{exc}}(\xi) + \mS_{\mM,\text{reg}}(\xi),
\end{aligned}
\]
where $\mS_{\mM,\text{exc}}$ contains the terms with $\lambda_j < 1/4$ and $\mS_{\mM,\text{reg}}$ contains the rest.
We first bound $\mS_{\mM,\text{reg}}$; the contribution of the holomorphic cusp forms and Eisenstein series is bounded analogously. For the Bessel transforms, we apply \cref{eq:bessel-bound-2} if $|r| \le R$ and \cref{eq:bessel-bound-3} otherwise, where $R \ge 1$ will be chosen shortly. Together with Cauchy--Schwarz and the bounds in \cref{lem:coefficient-bounds} (in $m$) and \cref{lem:large-sieve-non-exceptional} (in $n \sim N$), this yields
\[
\begin{aligned}
    \mS_{\mM,\text{reg}}(\xi)
    \ll_\eps (qmNR)^\eps 
    &\left(\frac{1 + |\log X|}{1 + X^{-1}} + R^{-5/2} + R^{-3} X^{-1}\right) 
    \\
    &\times \left(R^2 + \frac{\sqrt{(q,m)m}}{q}\right)^{1/2} \left(R^2 + \frac{N}{q}\right)^{1/2} \|a_n\|_2.
\end{aligned}
\]
Picking $R := 1 + X^{-1}$, we get
\begin{equation} \label{eq:S-bound-regular}
    \mS_{\mM,\text{reg}}(\xi)
    \ll_\eps (qmNC)^{O(\eps)} \frac{1}{1 + X^{-1}}
    \left(1 + X^{-2} + \frac{\sqrt{(q,m)m}}{q}\right)^{1/2} \left(1 + X^{-2} + \frac{N}{q}\right)^{1/2} \|a_n\|_2.
\end{equation}

For the exceptional spectrum, we let $X = X_0 \sqrt{X_1 X_2}$ for $X_1, X_2 \gg 1$ to be chosen shortly, and note the bound
\[
    1 + X^{2\theta_j} \ll (1 + X_0)^{2\theta_j} X_1^{\theta_j} X_2^{\theta_j} 
    \ll 
    (1 + X_0)^{2\theta(q)} X_1^{\theta_j} X_2^{\theta_j}.
\]
Then by \cref{eq:bessel-bound-1} and Cauchy--Schwarz, we obtain
\begin{equation} \label{eq:exceptional-contribution}
\begin{aligned}
    \mS_{\mM,\text{exc}}(\xi)
    &\ll 
    \frac{1}{1+X^{-1}}
    \sum_{\lambda_j < 1/4} \frac{1 + X^{2\theta_j}}{\cosh(\pi \kappa_j)} 
    |\rho_{j\ma}(m)| \left\vert \sum_{n \sim N} e\left(\frac{n}{N} \xi\right) a_n\, 
    \rho_{j\mb}(n)\right\vert
    \\
    &\ll
    \frac{(1 + X_0)^{2\theta(q)}}{1+X^{-1}}
    \left(
    \sum_{\lambda_j < 1/4} X_1^{2\theta_j} 
    |\rho_{j\ma}(m)|^2 \right)^{1/2}
    \left( 
    \sum_{\lambda_j < 1/4} X_2^{2\theta_j} \left\vert \sum_{n \sim N} e\left(\frac{n}{N} \xi\right) a_n\,  
    \rho_{j\mb}(n)\right\vert^2 \right)^{1/2}.
\end{aligned}
\end{equation}
We pick $X_1$ and $X_2$ as large as \cref{eq:coefficient-bound-exceptional} and \cref{ass:large-sieve} allow, specifically
\begin{equation} \label{eq:X12}
    X_1 := \max\left(1, \frac{q^2}{(q,m) m}\right),
    \qquad\qquad 
    X_2(\xi) := \max\left(1, \frac{q}{N}\right) \frac{Y_N}{1 + |\xi|^2}.
\end{equation}
Then by \cref{lem:coefficient-bounds} and \cref{ass:large-sieve}, we obtain
\begin{equation} \label{eq:S-bound-exceptional}
    \mS_{\mM,\text{exc}}(\xi)
    \ll_\eps 
    (qmNC)^{O(\eps)}
    \left(1 + \frac{X}{\sqrt{X_1 X_2(\xi)}} \right)^{2\theta(q)}
    \frac{1}{1+X^{-1}} \left(1 + \frac{\sqrt{(q,m)m}}{q} \right)^{1/2}
    \left(1 + \frac{N}{q}\right)^{1/2} A_N.
\end{equation}

Putting together \cref{eq:S-bound-regular} (and the identical bounds for Eisenstein series and holomorphic cusp forms) with \cref{eq:S-bound-exceptional} and \cref{eq:S-split-xi}, while noting that $\|a_n\|_2 \ll A_N$ by \cref{ass:large-sieve}, we conclude that
\begin{equation} \label{eq:S-final-bound}
\begin{aligned}
    \mS \ll_\eps (qmNCZ)^{O(\eps)}
    &\left(1 + \frac{X}{\sqrt{X_1 X_2(0)}} \right)^{2\theta(q)}
    \frac{C}{1+X^{-1}} 
    \\
    &\times 
    \left(1 + X^{-2} +\frac{\sqrt{(q,m)m}}{q} \right)^{1/2}
    \left(1 + X^{-2} + \frac{N}{q}\right)^{1/2} A_N,
\end{aligned}
\end{equation}
where the factor of $1 + |\xi|^2$ inside $X_2(\xi)$ disappeared in the integral over $\xi$ with a greater decay.
This recovers the desired bound after plugging in the values of $X, X_1, X_2$ from \cref{eq:X-CmN,eq:X12}.
\end{proof}

\begin{remark}
In treating the regular spectrum, we picked a slightly sub-optimal value of $R$ (following \cite[p.\,268]{deshouillers1982kloosterman}), to simplify the final bounds; in practice, this does not usually matter since one has $X \gg 1$.
\end{remark}

\begin{corollary}[Kloosterman bounds with averaging over $m, n, c$] \label{cor:kloosterman-averaging-mnc}
Let $(q, M, Z, (a_m)_{m \sim M}, A_M, Y_M)$ and $(q, N, Z, (b_n)_{n \sim N}, A_N, Y_N)$ satisfy \cref{ass:large-sieve}. Let $\eps > 0$, $C \gg 1$, $m \in \Z_+$, and $\ma, \mb$ be cusps of $\Gamma_0(q)$, with $\mu(\ma) = \mu(\mb) = q^{-1}$ and $\sigma_{\ma}, \sigma_{\mb}$ as in \cref{eq:scaling-choices}. Let $\Phi : (0, \infty)^3 \to \C$ be a smooth function, with $\Phi(x, y, z)$ supported in $x, y, z \asymp 1$, and $\partial_x^j \partial_y^k \partial_z^\ell \Phi(x, y, z) \ll_{j,k,\ell,\eps} Z^{(j+k)\eps}$ for $j, k, \ell \ge 0$. Then with a consistent choice of the $\pm$ sign, one has 
\begin{equation} \label{eq:kloosterman-mnc}
\begin{aligned}
    \sum_{m \sim M} a_m
    \sum_{n \sim N} b_n
    \sum_{c \in \mC_{\ma\mb}} \Phi\left(\frac{m}{M}, \frac{n}{N}, \frac{c}{C}\right) S_{\ma\mb}(m, \pm n; c)
    \ll_\eps
    (qMNCZ)^{O(\eps)}
    \left(1 + T\right)^{2\theta(q)}
    \\
    \times 
    \frac{C^2 A_M A_N}{C + \sqrt{MN}} 
    \left(1 + \frac{MN}{C^2} + \frac{M}{q} \right)^{1/2}
    \left(1 + \frac{MN}{C^2} + \frac{N}{q}\right)^{1/2},
\end{aligned}
\end{equation}
for
\[
    T = \frac{T_0}{\sqrt{Y_M Y_N}}, 
    \qquad\qquad 
    T_0 := \frac{C}{\max\left(M, q \right)^{1/2}\max\left(N, q\right)^{1/2}}
    \le
    \frac{C}{q}.
\]
In particular, for relatively prime positive integers $r, s$ with $rs = q$, one has
\begin{equation} \label{eq:kloosterman-mnc-explicit}
\begin{aligned}
    \sum_{m \sim M} a_m
    \sum_{n \sim N} b_n
    \sum_{(c, r) = 1} \Phi\left(\frac{m}{M}, \frac{n}{N}, \frac{c}{C}\right) S(m\bar{r}, \pm n; sc)
    \ll_\eps
    (rsMNCZ)^{O(\eps)} \left(1 + \frac{C}{\sqrt{r Y_M Y_N}} \right)^{2\theta(q)}
    \\
    \times
    A_M A_N
    \frac{\left(s\sqrt{r}C + \sqrt{MN} + \sqrt{sM} C\right)\left(s\sqrt{r}C + \sqrt{MN} + \sqrt{sN} C\right)}{s\sqrt{r}C + \sqrt{MN}}.
\end{aligned}
\end{equation}
\end{corollary}

\begin{remark}
Once again, $T_0$ represents the smallest value of $T$ that one could use prior to this work; see \cite[Theorem 9]{deshouillers1982kloosterman}. When $a_m = e(m\alpha)$ and $b_n = e(n\beta)$, \cref{cor:kloosterman-averaging-mnc} saves a factor of $(MN)^{\theta/2}$ over previous bounds (and up to $(MN)^{\theta}$ if $\alpha, \beta$ are close to rational numbers with small denominators).
\end{remark}

\begin{proof}[Proof of \cref{cor:kloosterman-averaging-mnc}]
We only mention what changes from the proof of \cref{cor:kloosterman-averaging-nc}. We expand the sum $\mS$ in the left-hand side of \cref{eq:kloosterman-mnc} as a double integral in $\zeta, \xi$, using the Fourier inversion formula 
\[
    \sqrt{xy}\,\Phi\left(x, y, \frac{\sqrt{xy}}{z} \right) = \iint_{\R^2} \hat{\Psi}(\zeta, \xi; z) \, e(x\zeta + y\xi)\, d\zeta\, d\xi,
\]
for $\Psi(x, y; z) := \sqrt{xy}\, \Phi(x, y, \sqrt{xy}/z)$, where the Fourier transform is taken in the first two variables. This yields
\[
    \mS \ll_\eps Z^{O(\eps)} C \iint_{\R^2} \frac{|S(\zeta, \xi)|\, d\zeta\, d\xi}{(1 + \zeta^4)(1 + \xi^4)},
\]
where 
\[
    \mS(\zeta, \xi) := \sum_{m \sim M} a_m\, e\left(m\frac{\zeta}{M}\right) \sum_{n \sim N} b_n\, e\left(n\frac{\xi}{N}\right) 
    \sum_{c \in \mC_{\ma\mb}}\,
    \frac{\mS_{\ma\mb}(m, \pm n; c)}{c} 
    \varphi_{\zeta,\xi}\left(\frac{4\pi \sqrt{mn}}{c}\right),
\]
and $\varphi_{\zeta,\xi}(z)$ is a smooth function supported in $z \asymp X^{-1}$, satisfying $\varphi_{\zeta,\xi}^{(\ell)} \ll_\ell X^\ell$ for
\[
    X := \frac{C}{\sqrt{MN}}.
\]
We proceed as before, applying the Kuznetsov formula from \cref{prop:kuznetsov} to the inner sum, then using the Bessel transform bounds from \cref{lem:bessel-bounds}. When applying Cauchy--Schwarz we keep the variable $m$ inside (as for $n$), and in consequence we use large sieve inequalities for the sequence $(a_m)$ (i.e., \cref{lem:large-sieve-non-exceptional,ass:large-sieve}). The resulting bounds are symmetric in $M, N$, with
\[
    X_1(\zeta) := \max\left(1, \frac{q}{M}\right) \frac{Y_M}{1 + |\zeta|^2}
    \qquad 
    \text{and}
    \qquad
    X_2(\xi) := \max\left(1, \frac{q}{N}\right) \frac{Y_N}{1 + |\xi|^2}.
\]
Instead of \cref{eq:S-final-bound}, we thus obtain
\begin{equation} \label{eq:S-final-bound-2}
\begin{aligned}
    \mS \ll_\eps (qMNCZ)^{O(\eps)}
    &\left(1 + \frac{X}{\sqrt{X_1(0) X_2(0)}} \right)^{2\theta(q)}
    \frac{C}{1+X^{-1}} 
    \\
    &\times 
    \left(1 + X^{-2} + \frac{M}{q}\right)^{1/2}
    \left(1 + X^{-2} + \frac{N}{q}\right)^{1/2} \|a_m\|_2\, \|b_n\|_2,
\end{aligned}
\end{equation}
which recovers \cref{eq:kloosterman-mnc} after plugging in the values of $X, X_1, X_2$.

Finally, to prove \cref{eq:kloosterman-mnc-explicit} for $q = rs$,
we pick $\ma = \infty$, and $\mb = 1/s$, keeping the scaling matrices in \cref{eq:scaling-choices}, and use \cref{eq:explicit-kl-diff} to rewrite $S(m\bar{r}, n; sc)$ as $S_{\infty\,1/s}(m, \pm n; s\sqrt{r}c)$ when $(c, r) = 1$. After substituting $C \gets s\sqrt{r}C$, the value of $T$ inside the $\theta$-factor becomes
\[
    \frac{s\sqrt{r} C}{\max(M, q)^{1/2} \max(N, q)^{1/2} \sqrt{Y_MY_N}}
    \le 
    \frac{s\sqrt{r} C}{rs \sqrt{Y_M Y_N}} = 
    \frac{C}{\sqrt{rY_MY_N}},
\]
and so \cref{eq:kloosterman-mnc} recovers \cref{eq:kloosterman-mnc-explicit} up to minor rearrangements.
\end{proof}

\begin{corollary}[Kloosterman bounds with averaging over $q, m, n, c$] \label{cor:kloosterman-averaging-qmnc}
Let $Q, M, N \ge 1/2$, $C, Z \gg 1$, $Y_N > 0$, $\eps > 0$, and $\omega \in \R/\Z$. For each $q \sim Q$, let $(q, N, Z, (a_{n,q})_{n \sim N}, A_{N,q}, Y_N)$ satisfy \cref{ass:large-sieve}, $w_q \in \C$, $\mb_q$ be a cusp of $\Gamma_0(q)$, and $\Phi_q : (0, \infty)^3 \to \C$ be a smooth function, with $\Phi_q(x, y, z)$ supported in $x, y, z \asymp 1$, and $\partial_x^j \partial_y^k \partial_z^\ell \Phi_q(x, y, z) \ll_{j,k,\ell,\eps} Z^{(j+k)\eps}$ for $j, k, \ell \ge 0$. Then with the choice of scaling matrices in \cref{eq:scaling-choices} and a consistent choice of the $\pm$ sign, one has
\begin{equation} \label{eq:kloosterman-qmnc}
\begin{aligned}
    \sum_{q \sim Q} w_q
    \sum_{m \sim M} e(m\omega)
    \sum_{n \sim N} a_{n,q}
    \sum_{c \in \mC_{\infty\mb_q}} \Phi_q\left(\frac{m}{M}, \frac{n}{N}, \frac{c}{C}\right) S_{\infty\mb_q}(m, \pm n; c)
    \ll_\eps
    (QMNCZ)^{O(\eps)}
    \\
    \times 
    \left(1 + T\right)^{2\theta_{\max}}
    \frac{\sqrt{QM} \|w_q A_{N,q}\|_2 C^2}{C + \sqrt{MN}} 
    \left(1 + \frac{MN}{C^2} + \frac{M}{Q} \right)^{1/2}
    \left(1 + \frac{MN}{C^2} + \frac{N}{Q}\right)^{1/2},
\end{aligned}
\end{equation}
for
\[
    T = \frac{T_0}{\sqrt{Y_N}}, 
    \qquad\qquad 
    T_0 := \frac{C}{\max\left(M, Q \right) \max\left(N, Q\right)^{1/2}}
    \le
    \frac{C}{Q^{3/2}}.
\]
In particular, let $R, S \ge 1/2$; for every $r \sim R, s \sim S$ with $(r, s) = 1$, let $w_{r,s} \in \C$, $\Phi_{r,s}$ be as above, and $(rs, N, Z, (a_{n,r,s})_{n \sim N}, A_{N,r,s}, Y_N)$ satisfy \cref{ass:large-sieve}. Then one has 
\begin{equation} \label{eq:kloosterman-rsmnc}
\begin{aligned}
    \sum_{\substack{r \sim R \\ s \sim S \\ (r, s) = 1}} w_{r,s} \sum_{m \sim M} e(m\omega)
    \sum_{n \sim N} a_{n,r,s}
    \sum_{(c, r) = 1} \Phi_{r,s}\left(\frac{m}{M}, \frac{n}{N}, \frac{c}{C}\right) S(m\bar{r}, \pm n; sc)
    \ll_\eps
    (RSMNCZ)^{O(\eps)}
    \\
    \times 
    \left(1 + \frac{C}{R \sqrt{S Y_N}}\right)^{2\theta_{\max}} 
    \sqrt{RSM} \|w_{r,s} A_{N,r,s}\|_2
    \\
    \times
    \frac{\left(S\sqrt{R}C + \sqrt{MN} + \sqrt{SM} C\right)\left(S\sqrt{R}C + \sqrt{MN} + \sqrt{SN} C\right)}{S\sqrt{R}C + \sqrt{MN}}.
\end{aligned}
\end{equation}
\end{corollary}

\begin{remark}
The norms $\|w_q A_{N,q}\|_2$ and $\|w_{r,s} A_{N,r,s}\|_2$ refer to sequences indexed by $q \sim Q$, respectively $r \sim R$, $s \sim S$ (but not $n \sim N$).
In practice, it is often helpful to follow \cref{eq:kloosterman-rsmnc} with the bound
\begin{equation} \label{eq:Kloosterman-follow-up}
\begin{aligned} 
    &\frac{\left(S\sqrt{R}C + \sqrt{MN} + \sqrt{SM} C\right)\left(S\sqrt{R}C + \sqrt{MN} + \sqrt{SN} C\right)}{S\sqrt{R}C + \sqrt{MN}}
    \\
    &\ll 
    S\sqrt{R}C + \sqrt{MN} + \sqrt{SM}C + \sqrt{SN}C + \frac{\sqrt{SM} C \sqrt{SN} C}{S\sqrt{R} C}
    \\
    &\ll 
    \left(\frac{C^2}{R} (M + RS)(N + RS) + MN \right)^{1/2}.
\end{aligned}
\end{equation}
\end{remark}

\begin{remark}
\cref{cor:kloosterman-averaging-qmnc} should be compared with \cite[Theorem 11]{deshouillers1982kloosterman}, the relevant saving being $Y_N^{\theta_{\max}}$. One can state a similar result, to be compared with \cite[Theorem 10]{deshouillers1982kloosterman}, using a general sequence $(b_m)_{m \sim M}$ instead of $b_m = e(m\omega)$; one would need to replace a factor of $\sqrt{M}$ with $\|b_m\|_2$, and adjust the value of $T_0$ using \cite[Theorem 6]{deshouillers1982kloosterman} (or rather, its optimization in \cite{lichtman2023primes}) instead of \cite[Theorem 7]{deshouillers1982kloosterman}.
\end{remark}

\begin{proof}[Proof of \cref{cor:kloosterman-averaging-qmnc}]
We proceed as in the proof of \cref{cor:kloosterman-averaging-mnc}, swapping the sum over $q$ with the integral to bound the sum $\mS$ in the left-hand side of \cref{eq:kloosterman-mnc} by
\[
    \mS \ll_\eps Z^{O(\eps)} C \iint_{\R^2} \frac{S(\zeta, \xi)\, d\zeta\, d\xi}{(1 + \zeta^4)(1 + \xi^4)},
\]
where
\[
\begin{aligned}
    &\mS(\zeta, \xi) 
    \\
    &:= \sum_{q \sim Q} |w_q| \left\vert \sum_{m \sim M}e\left(m\left(\omega + \frac{\zeta}{M}\right)\right) \sum_{n \sim N} a_{n,q}\, e\left(n\frac{\xi}{N}\right) 
    \sum_{c \in \mC_{\infty\mb_q}}\,
    \frac{\mS_{\infty\mb_q}(m, \pm n; c)}{c} 
    \varphi_{\zeta,\xi,q}\left(\frac{4\pi \sqrt{mn}}{c}\right) \right\vert,
\end{aligned}
\]
and $\varphi_{\zeta,\xi,q}(z)$ are smooth functions supported in $z \asymp X^{-1}$, satisfying $\varphi_{\zeta,\xi,q}^{(\ell)} \ll_\ell X^\ell$ for $X := \frac{C}{\sqrt{MN}}$. After applying the Kuznetsov formula, we bound the contribution of the regular spectrum to $\mS(\zeta, \xi)$ pointwise in $q$, as in the previous proofs (leading only to an extra factor of $\|w_q A_{N,q}\|_1 \le \|w_q A_{N,q}\|_2 \sqrt{Q}$ instead of $A_N$). As in \cref{eq:exceptional-contribution}, the contribution of the exceptional spectrum is
\[
\begin{aligned}
    &\mS_{\mM,\text{exc}}(\zeta, \xi) 
    \ll 
    \\
    &\frac{1}{1+X^{-1}}
    \sum_{q \sim Q} |w_q|
    \sum_{\lambda_j < 1/4} \frac{1 + X^{2\theta_j(q)}}{\cosh(\pi \kappa_j)} 
    \left\vert \sum_{m \sim M} e\left(m \left(\omega + \frac{\zeta}{M}\right) \right) \bar{\rho_{j\infty_q}(m)} \right\vert \left\vert \sum_{n \sim N} a_{n,q}\, e\left(n\frac{\xi}{N}\right)
    \rho_{j\mb_q}(n)\right\vert.
\end{aligned}
\]
We then apply Cauchy--Schwarz in the double sum over $q$ and $j$, splitting $X = X_0 \sqrt{X_1 X_2}$ for $X_2(\xi)$ as in \cref{eq:X12}; but this time we choose
\begin{equation} \label{eq:final-X1-choice}
    X_1 := \max\left(M, \frac{Q^2}{M}\right),
\end{equation}
corresponding to the allowable range in \cref{thm:large-sieve-level-avg}. Keeping $|w_q|$ only in the second sum, this yields
\[
\begin{aligned}
    \mS_{\mM,\text{exc}}(\zeta, \xi) 
    \ll 
    \sqrt{\frac{(1 + X_0)^{2\theta_{\max}}}{1+X^{-1}} \mS_M(\zeta, \xi)\, \mS_N(\zeta, \xi)},
\end{aligned}
\]
where 
\[
\begin{aligned}
    \mS_M(\zeta, \xi) &:= \sum_{q \sim Q}
    \sum_{\lambda_j < 1/4} \frac{X_1^{2\theta_j(q)}}{\cosh(\pi \kappa_j)} 
    \left\vert \sum_{m \sim M} e\left(m \left(\omega + \frac{\zeta}{M}\right) \right) \bar{\rho_{j\infty_q}(m)} \right\vert^2,
    \\
    \mS_N(\zeta, \xi) &:= \sum_{q \sim Q} |w_q|^2
    \sum_{\lambda_j < 1/4} \frac{X_2^{2\theta_j(q)}}{\cosh(\pi \kappa_j)} \left\vert \sum_{n \sim N} a_{n,q}\, e\left(n\frac{\xi}{N}\right) 
    \rho_{j\mb_q}(n)\right\vert^2.
\end{aligned}
\]
The treatment of $\mS_N$ remains the same as before, pointwise in $q$, leading to an extra factor of $\|w_q A_{N,q}\|_2^2$ instead of $A_N^2$. For $\mS_M$, we apply \cref{thm:large-sieve-level-avg} (which allowed the choice of $X_1$ from \cref{eq:final-X1-choice}), leading to an extra factor of $\sqrt{Q}$. Overall, instead of \cref{eq:S-final-bound-2}, we obtain
\[
\begin{aligned}
    \mS \ll_\eps (QMNCZ)^{O(\eps)}
    &\left(1 + \frac{X}{\sqrt{X_1 X_2(0)}} \right)^{2\theta_{\max}}
    \frac{C}{1+X^{-1}} 
    \\
    &\times 
    \left(1 + X^{-2} + \frac{M}{Q}\right)^{1/2}
    \left(1 + X^{-2} + \frac{N}{Q}\right)^{1/2} \sqrt{QM} \|w_q A_{N,q}\|_2,
\end{aligned}
\]
and plugging in the values of $X, X_1, X_2$ yields \cref{eq:kloosterman-qmnc}.

To prove \cref{eq:kloosterman-rsmnc}, let $Q := RS$. By the divisor bound, the left-hand side is at most
\[
    x^{o(1)} \sum_{Q < q \le 4Q}
    \max_{\substack{r \sim R \\ s \sim S \\ (r, s) = 1 \\ rs = q}} |w_{r,s}| 
    \left\vert 
    \sum_{m \sim M} e(m\omega)
    \sum_{n \sim N} a_{n,r,s}
    \sum_{(c, r) = 1}
    \Phi_{r,s} \left(\frac{m}{M}, \frac{n}{N}, \frac{c}{C}\right)
    S(m\bar{r}, \pm n; sc)
    \right\vert,
\]
where we interpret any empty maximum as $0$. For each $q$, let $r = r(q), s = s(q)$ attain the maximum (if there are no such $r, s$, pick $w_q := 0$ and disregard the rest of this paragraph). Then let $w_q := w_{r,s}$, $a_{n,q} := a_{n,r,s}$, $\Phi_q(x, y, z) := \Phi_{r,s}(x, y, z\, (S/s) \sqrt{R/r})$, and $\mb_q := 1/s$, with the scaling matrix in \cref{eq:scaling-choices}.

Due to \cref{lem:explicit-kloosterman}, after the change of variables $c \gets c/(s\sqrt{r})$, this leaves us with the sum
\[
    x^{o(1)} \sum_{Q < q \le 4Q}
    |w_q| \left\vert
    \sum_{m \sim M} e(m\omega)
    \sum_{n \sim N} a_{n,q}
    \sum_{c \in \mC_{\infty \mb_q}}
    \Phi_q \left(\frac{m}{M}, \frac{n}{N}, \frac{c}{S\sqrt{R}C}\right)
    S_{\infty \mb_q}(m, \pm n; c)
    \right\vert.
\]
Incorporating $1$-bounded coefficients into $(w_q)$ to remove absolute values, the desired bound now follows from \cref{eq:kloosterman-qmnc}. We note that the $T$ parameter becomes
\[
    T \ll \frac{S\sqrt{R} C}{Q^{3/2} \sqrt{Y_N}} \asymp \frac{C}{R\sqrt{S Y_N}},
\]
as in \cref{eq:kloosterman-qmnc}.
\end{proof}

As a direct consequence of \cref{cor:kloosterman-averaging-qmnc} and standard techniques, one can also deduce a result for sums of incomplete Kloosterman sums, improving \cite[Theorem 12]{deshouillers1982kloosterman}.

\begin{corollary}[Incomplete Kloosterman bounds with averaging over $r, s, n, c, d$] \label{cor:kloosterman-incomplete}
Let $R, S, N \ge 1/2$, $C, D, Z \gg 1$, $Y_N > 0$, and $\eps > 0$. For each $r \sim R, s \sim S$ with $\gcd(r, s) = 1$, let the tuple $(rs, N, Z, (a_{n,r,s})_{n \sim N}, A_{N,r,s}, Y_N)$ satisfy \cref{ass:large-sieve}, $w_{r,s} \in \C$, and $\Phi_{r,s} : (0, \infty)^3 \to \C$ be a smooth function, with $\Phi_{r,s}(x, y, z)$ supported in $x, y, z \asymp 1$, and $\partial_x^j \partial_y^k \partial_z^\ell \Phi_q(x, y, z) \ll_{j,k,\ell,\eps} Z^{j\eps}$ for $j, k, \ell \ge 0$. Then with a consistent choice of the $\pm$ sign, one has
\begin{equation} \label{eq:kloosterman-incomplete}
    \sum_{\substack{r \sim R \\ s \sim S \\ (r, s) = 1}} w_{r,s}
    \sum_{n \sim N} a_{n,r,s}
    \sum_{\substack{c, d \\ (rd, sc) = 1}} \Phi_{r,s}\left(\frac{n}{N}, \frac{d}{D}, \frac{c}{C}\right) e\left(\pm n \frac{\bar{rd}}{sc}\right)
    \ll_\eps 
    (RSNCDZ)^{O(\eps)}\, \|w_{r,s} A_{N,r,s}\|_2 \msI,
\end{equation}
where
\[
    \msI^2 := D^2 NR
    +
    \left(1 + \frac{C^2}{R^2 S Y_N}\right)^{2\theta_{\max}} 
    CS(C + DR)(RS + N).
\]
\end{corollary}

\begin{proof}[Proof of \cref{cor:kloosterman-incomplete}]
This follows from \cref{cor:kloosterman-averaging-qmnc} (specifically, \cref{eq:kloosterman-rsmnc}) by completing Kloosterman sums, passing from the $d$-variable to a variable $m$ of size $\ll_\eps (CDS)^\eps CS/D$; this is completely analogous to how \cite[Theorem 12]{deshouillers1982kloosterman} follows from \cite[Theorem 11]{deshouillers1982kloosterman} in \cite[\S 9.2]{deshouillers1982kloosterman}. We note that \cite[Theorem 12]{deshouillers1982kloosterman} has a minor error (replacing $D^2 N R$ with $D^2 N R S^{-1}$), which has been corrected in \cite{bombieri2019some}.
\end{proof}

\section{The greatest prime factor of \texorpdfstring{$n^2+1$}{n\^2+1}} \label{sec:greatest-prime-factor}

Here we use our new inputs from \cref{subsec:multilinear-kloosterman} in the computations of Merikoski \cite{merikoski2023largest} and de la Bret\`eche--Drappeau \cite{de2020niveau}, in order to prove \cref{thm:appl-greatest-prime}. We begin with a brief informal sketch. 

\subsection{Sketch of the argument} \label{subsec:informal-overview-greatest}

We will ultimately prove a lower bound of the shape
\[
    \sum_{n \sim x} \sum_{\substack{p \text{ prime} \\ p \mid n^2+1 \\ p > x^{1.3}}} \log p
    > \eps\, x \log x,
\]
which implies that for some (in fact, for many) $n \sim x$, we must have $P^+(n^2+1) > x^{1.3}$. As in previous works \cite{merikoski2023largest,de2020niveau,deshouillers1982greatest,hooley1967greatest},  we use an idea of Chebyshev to estimate the full sum
\[
    \sum_{n \sim x} \sum_{\substack{p \text{ prime} \\ p \mid n^2+1}} \log p
    \approx 
    \sum_{n \sim x}
    \sum_{d \mid n^2+1} \Lambda(d)
    =
    \sum_{n \sim x} 
    \log(n^2+1)
    =
    2x\log x + O(x),
\]
where $\Lambda$ is the von Mangoldt function. It then remains to upper bound
\[
    \sum_{n \sim x} \sum_{\substack{p \text{ prime} \\ p \mid n^2+1 \\ p \le x^{1.3}}} \log p
    =
    \sum_{\substack{p \text{ prime} \\ p \le x^{1.3}}} \log p\,
    \sum_{n \sim x} \one_{n^2 \equiv -1 \pmod{p}}
    \stackrel{?}{<} (2-\eps)\,x \log x.
\]
Following Merikoski \cite{merikoski2023largest}, we use repeated applications of Buchstab's identity inside the Harman sieve method, to reduce estimating the above sum over primes to
bounding ``Type I'' and ``Type II'' sums of the form
\[
    \sum_{d \le D} \lambda_d \sum_{\substack{q \sim Q \\ q \equiv 0 \pmod{d}}} \left(
    \sum_{n \sim x} \one_{n^2 \equiv -1 \pmod{q}}
    -
    \frac{x}{q} \sum_{\nu \pmod{q}} \one_{\nu^2 \equiv -1 \pmod{q}}\right),
\]
respectively
\[
    \sum_{q_1 \sim Q_1} \lambda_{q_1} \sum_{q_2 \sim Q_2} \mu_{q_2} \left(
    \sum_{n \sim x} \one_{n^2 \equiv -1 \pmod{q_1 q_2}}
    -
    \frac{x}{q_1 q_2} \sum_{\nu \pmod{q_1q_2}} \one_{\nu^2 \equiv -1 \pmod{q_1q_2}}\right),
\]
for various ranges of $D, Q, Q_i$ with $Q_1 Q_2 = Q \le x^{1.3}$, aiming to win over the trivial bound of $x$. 

Let us sketch our improvement of the Type I information. As in \cite{de2020niveau,merikoski2023largest}, we separate $n$ into residue classes $\nu \pmod{q}$ and apply a truncated version of Poisson summation to the sum over $n$; the principal frequency $h = 0$ cancels with the subtracted main term, leading to a sum like
\[
    \sum_{d \le D} \lambda_d \sum_{\substack{q \sim Q \\ q \equiv 0 \pmod{d}}} \sum_{\substack{\nu \pmod{q} \\ \nu^2 \equiv -1\ (q)}}\ 
    \sum_{h \sim \frac{Q}{x}} e\left(\frac{h\nu}{q}\right).
\]
We can then parametrize the solutions to $\nu^2 \equiv -1 \pmod{q}$ by the Gauss correspondence (see \cite[Lemma 2]{deshouillers1982greatest}), which gives $q = r^2 + s^2$ and $\tfrac{\nu}{q} \approx \tfrac{\bar{r}}{s} \pmod{1}$. In the critical range $r, s \sim \sqrt{Q}$, this leaves us with the sum
\[
    \sum_{d \le D} \lambda_d \sum_{\substack{r, s \sim \sqrt{Q}\\ r^2 \equiv -s^2 \pmod{d}}} 
    \sum_{h \sim \frac{Q}{x}} e\left(\frac{h\bar{r}}{s}\right).
\]
Separating $r \bar{s}$ into residue classes $\varrho \mod d$ and Fourier-completing the sum over $r$ yields a dual variable $k$ of size $D$, and inner sums of the shape
\[
    \sum_{k \sim D} \sum_{h \sim \frac{Q}{x}} e\left(\frac{k\varrho}{d}\right) \sum_{s \sim \sqrt{Q}} S(h, k\bar{d}; s),
\]
where $\varrho$ is a solution to $\varrho^2 \equiv -1 \pmod{d}$; compare these to \cref{eq:sum-of-kloosterman}. Since the sequence $(e(\tfrac{k\varrho}{d}))_{k \sim D}$ depends on the level $d$ and its length is as large as the level, the large sieve inequality for general sequences from \cref{thm:large-sieve-general} cannot produce any savings in the $X^\theta$-aspect for this sequence. However, our large sieve inequality for exponential phases from \cref{thm:large-sieve-expo-phases} will save a factor of $D^{\theta/2}$, which follows through to the sieve computations. This improves the results of de la Bret\`eche--Drappeau \cite[\S 8]{de2020niveau}, based in turn on Duke--Friedlander--Iwaniec \cite{duke1995equidistribution}, and is nearly enough to remove the dependency on Selberg's eigenvalue conjecture in the relevant Type I ranges, as illustrated in \cref{fig:arithm-info} (left).

Following Merikoski \cite{merikoski2023largest}, the Type II sums can be treated similarly using an additional Cauchy--Schwarz step. Once again, this leads to trilinear forms of Kloosterman sums as in \cref{eq:sum-of-kloosterman}, where both sequences $(a_m)$ and $(b_n)$ depend on the level. Using our large sieve inequalities, we can leverage the fact that $(a_m)$ happen to be exponential phases as in \cref{thm:large-sieve-expo-phases}, while $(b_n)$ have the shape in \cref{thm:large-sieve-dispersion-coeffs} (if the sum over $h$ is kept inside the Cauchy--Schwarz step). This ultimately produces three admissible Type II ranges, all gathered in \cref{prop:type-II} and illustrated in \cref{fig:arithm-info} (right). 

By carefully plugging in these
Type I and II estimates into Merikoski’s Harman sieve computations, which require the numerical
calculation of multidimensional integrals, we deduce \cref{thm:appl-greatest-prime}.

\subsection{Arithmetic information} \label{subsec:arithm-info} 
We aim to improve the dependency on the $\theta$ parameter in the arithmetic information from \cite[Propositions 1 and 2]{merikoski2023largest}; to do so, we first improve a lemma of de la Bret\`eche--Drappeau \cite[Lemme 8.3]{de2020niveau}. 

\begin{lemma}[De la Bret\`eche--Drappeau-style exponential sums] \label{lem:improvement-bd}
Let $\eps > 0$, $M \gg 1$, and $\theta := \tfrac{7}{64}$.
\begin{itemize}
    \item[$(i)$.] Let $q, h \in \Z$ and $1 \le |h| \ll q$. Given a smooth function $f : (0, \infty) \to \C$ supported in $v \asymp 1$, with $f^{(j)} \ll_j 1$ for $j \ge 0$, one has
\begin{equation} \label{eq:bd-1}
    \sum_{(m, q) = 1} f\left(\frac{m}{M}\right) 
    \sum_{\nu^2 \equiv -1 \pmod{mq}}
    e\left(\frac{h\nu}{mq}\right) 
    \ll_\eps  
    (qhM)^\eps \left(
    |h| + \sqrt{qM}\left(1 + (q, h)^{\theta} q^{- 3\theta/2} M^{\theta} \right) \right).
\end{equation}
    \item[$(ii)$.] Let $Q \ge 1/2$, $1/2 \le H \ll QM$, and $t \in \R/\Z$. Given smooth functions $(f_q(v))_{q \sim Q}$ supported in $v \asymp 1$, with $f_q^{(j)} \ll_j 1$ for $j \ge 0$, one has
\begin{equation} \label{eq:bd-2}
\begin{aligned}
    \frac{1}{Q} \sum_{q \sim Q}
    &\left\vert \frac{1}{H} \sum_{h \sim H} e(th) 
    \sum_{(m, q) = 1} f_q\left(\frac{m}{M}\right) 
    \sum_{\nu^2 \equiv -1 \pmod{mq}}
    e\left(\frac{h\nu}{mq}\right) 
    \right\vert
    \\
    &\qquad \qquad \ll_\eps  
    (QHM)^\eps \left(
    H + \sqrt{M} \left(1 + H^{-2\theta} Q^{\theta/2} M^{\theta}\right) + \sqrt{\frac{QM}{H}}\left(1 + Q^{-3\theta/2} M^{\theta} \right)\right).
\end{aligned}
\end{equation}
\end{itemize}
\end{lemma}

\begin{proof}
This is a refinement of the first and third bounds in \cite[Lemme 8.3]{de2020niveau}, winning factors of about $q^{\theta/2}$ via our \cref{cor:kloosterman-averaging-nc,cor:kloosterman-averaging-qmnc}. We only mention what changes from the proof in \cite[\S 8.1]{de2020niveau}, working in the particular case $d = r = 1$, $D = -1$. We note that for $D = -1$, the relevant cusps $\ma$ from \cite[\S 8.1]{de2020niveau} are equivalent to $0/1$, and thus have $\mu(\ma) = q^{-1}$ (which is also why Merikoski's bounds in \cite[\S 3.8]{merikoski2023largest} only require such cusps too). 

For part $(i)$, we consider the sums of Kloosterman sums from \cite[(8.30)]{de2020niveau}, given (with notation to be explained below) by
\[
    V_N = V_N(q, h) := \sum_{N/2 \le |n| \le 2N} \sum_{\gamma \in \mC_{\infty\ma}} S_{\infty\ma}(h, n; \gamma)\, G_N(\gamma, n).
\]
Here, the $n$-variable came from a completion of Kloosterman sums, and was localized to a dyadic range of size $N \ll q^{1+\eta} M^\eta$ (where $\eta > 0$ is a small parameter), while $G_N(\gamma, n)$ is a smooth function normalized such that
\[
    \Phi(x, y) := q\, G_N\left(yq\sqrt{M}, xN\right)
\]
satisfies the assumptions of \cref{cor:kloosterman-averaging-nc} with $Z = qM$ and $\eps \asymp \eta$. Also, $\ma$ is a cusp of $\Gamma_0(q)$, and the scaling matrix $\sigma_\ma$ used implicitly in the Kloosterman sum $S_{\infty\ma}(h, n; \gamma)$ hides an exponential phase of the form $e(n\alpha_q)$; the value of $\alpha_q$ is arbitrary for our purposes. 

We can now apply \cref{cor:kloosterman-averaging-nc} (equivalently, we can bound $\mathscr{M}_N^{\textnormal{exc}}$ in \cite[(8.40)]{de2020niveau} using \cref{thm:large-sieve-expo-phases}), using $a_n = e(n\alpha_q)$, $Y_N = A_N = \sqrt{N}$ (corresponding to \cref{eq:sequence-choice-1}), $C = q\sqrt{M}$, and $m = |h|$. This yields
\[
    V_N \ll_{\eta} (qhM)^{O(\eta)} \left(1 + \frac{q\sqrt{M}}{q^{3/2} (q, h)^{-1/2} N^{1/4}}\right)^{2\theta} \sqrt{NM},
\]
where we used that $q^\eta C = q^{1+\eta}\sqrt{M} \gg \sqrt{hN}$, that $\sqrt{(q, h)|h|} \le |h| \le q$, and that $N \ll q^{1+\eta} M^\eta$ (in particular, the $1$-term is dominant in the last two parentheses from \cref{eq:kloosterman}, up to factors of $(qhM)^{o(1)}$). 

This bound is increasing in $N$, so using $N \ll q^{1+\eta} M^\eta$ once again, we get
\[
    V_N \ll_\eta (qhM)^{O(\eta)}  \sqrt{qM} \left(1 + (q, h)^{\theta} q^{-3\theta/2} M^{\theta}\right),
\]
which gives the second term claimed in the upper bound from \cref{eq:bd-1}.

Part $(ii)$ follows similarly using \cref{cor:kloosterman-averaging-qmnc} (or equivalently, by bounding $\mathscr{M}_N^{\textnormal{exc}}$ in \cite[\S 8.1.12]{de2020niveau} using \cref{thm:large-sieve-expo-phases} once again). Indeed, with the similar choices $a_{n,q} = e(n\alpha_q)$, $Y_N = A_{N,q} = \sqrt{N}$, $Z = QM$, and $C = Q\sqrt{M}$, our bound \cref{eq:kloosterman-qmnc} yields 
\[
    \frac{1}{QH} \sum_{q \sim Q} \left\vert \sum_{h \sim H} e(th)\, V_N(q, h)\right\vert
    \ll_\eta 
    (QHM)^{O(\eta)}
    \left(1 + \frac{Q\sqrt{M}}{\max(Q, H)\, Q^{1/2} N^{1/4}}\right)^{2\theta} \sqrt{\frac{NM}{H}}
    \left(1 + \frac{H}{Q}\right)^{1/2}.
\]
Again, this bound is increasing in $N$, so plugging in $N \ll Q^{1+\eta} M^\eta$ gives a right-hand side of
\[
\begin{aligned}
    &\ll_\eta 
    (QHM)^{O(\eta)} \sqrt{\frac{M}{H}} \max(Q, H)^{1/2} \left(1 + \max(Q, H)^{-2\theta} Q^{\theta/2} M^{\theta}\right)
    \\
    &\ll
    (QHM)^{O(\eta)} \sqrt{\frac{M}{H}} \left(H^{1/2} + Q^{1/2} + \left(H^{(1/2)-2\theta} + Q^{(1/2)-2\theta}\right) Q^{\theta/2} M^{\theta}\right),
\end{aligned}
\]
which gives all but the first term in the upper bound from \cref{eq:bd-2}.
As in \cite{de2020niveau}, the first terms of $|h|$ and $H$ from our bounds in \cref{eq:bd-1,eq:bd-2} could be improved via partial summation, but we omit this optimization too since it will not be relevant for our computations.
\end{proof}

\begin{notation}[Set-up for arithmetic information] \label{not:set-up-info}
Let $x \ge 1$, $\alpha \in [1, 3/2)$, and
\[
    P := x^\alpha.
\] 
Let $\Phi, \Psi$ be smooth functions supported in $[1, 4]$, satisfying $\Phi \ge 0$ and $\Phi^{(j)}, \Psi^{(j)} \ll_j 1$ for $j \ge 0$ (in \cite[\S 2.1]{merikoski2023largest}, Merikoski uses $b(t) = \Phi(t/x)$ and $\Psi(t) \gets \Psi(t/P)$). For $q \in \Z_+$, define
\[
    |\mA_q| := \sum_{n^2 \equiv -1 \pmod{q}} \Phi\left(\frac{n}{x}\right),
    \qquad\qquad 
    X := \int \Phi\left(\frac{t}{x}\right)\, dt
    = 
    x \int \Phi,
\]
\[
    \rho(q) := \# \left\{ \nu \in \Z/q\Z : \nu^2 \equiv -1 \pmod{q} \right\}.
\]
\end{notation}

We will estimate the difference 
\[
    |\mA_q| - X\frac{\rho(q)}{q}
\] 
in ``Type I'' and ``Type II'' sums with $q \asymp P$. The Type I sums average over moduli in arithmetic progressions, say $q \equiv 0 \pmod{d}$ and $d \le D$, with arbitrary divisor-bounded coefficients $\lambda_d$; the Type II sums average over moduli with a conveniently-sized factor, say $q = mn$ with $n \sim N$ (and $m \asymp P/N$), with divisor-bounded coefficients $a_m, b_n$. One can also view the Type I sums as special Type II sums where $a_m = 1$, except that Type II estimates typically require a lower bound on $N$.

The strength of the resulting Type I and Type II information is given by the ranges of parameters $D$ and $N$ (in terms of $x$ and $P$) for which we can obtain power-savings over the trivial bound---i.e., for which the sums over $|\mA_q|$ have an asymptotic formula. \cref{fig:arithm-info} illustrates the (previous unconditional, new unconditional, and conditional) admissible choices of $\log_x D$ and $\log_x N$ in terms of $\alpha = \log_x P$; both graphs continue downwards, the second region being lower-bounded by the function $\alpha - 1$. The previous unconditional and the conditional ranges are due to Merikoski \cite{merikoski2023largest} and de la Bret\`eche--Drappeau \cite{de2020niveau}; our improvements are \cref{prop:type-I,prop:type-II}. 

\begin{figure}[ht]
{\centering
\includegraphics[scale=0.3]{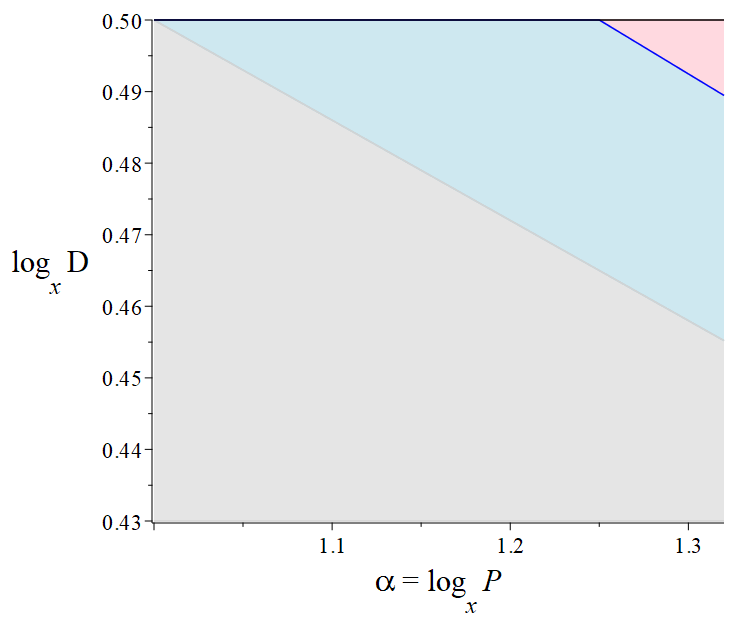}
\hspace{0.3cm}
\includegraphics[scale=0.3]{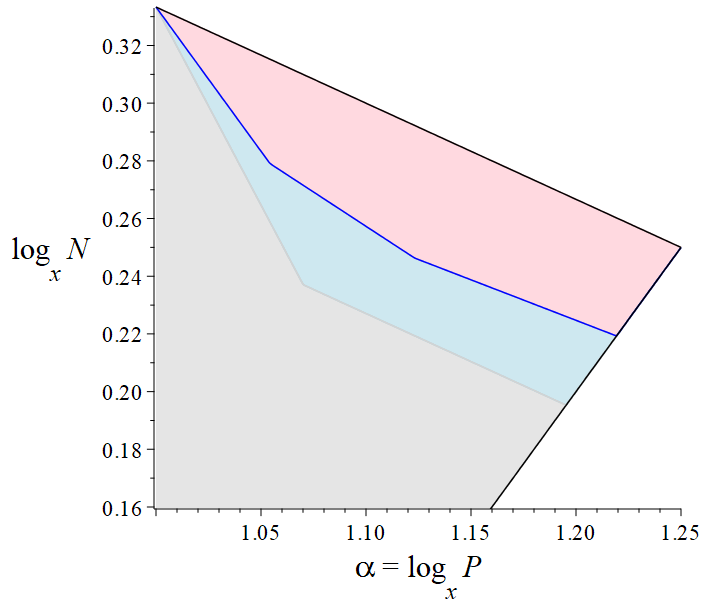}
}
\caption{Type I (left) and Type II (right) ranges. Previous results in gray; our improvements in blue; conditional ranges in red (assuming Selberg's eigenvalue conjecture).}
\label{fig:arithm-info}
\end{figure}
%\FloatBarrier

\begin{proposition}[Type I estimate] \label{prop:type-I}
For any sufficiently small $\eps > 0$ there exists $\delta > 0$ such that the following holds. With \cref{not:set-up-info}, $1 \le \alpha \le 1.4$, $\theta := \tfrac{7}{64}$, and $D \ge 1$, one has 
\begin{equation} \label{eq:type-I}
    \sum_{d \le D} \lambda_d \sum_{q \equiv 0 \pmod{d}} \left(|\mA_q| - X\frac{\rho(q)}{q}\right) \Psi\left(\frac{q}{P}\right) \log q 
    \ll_\eps x^{1-\delta},
\end{equation}
for any divisor-bounded coefficients $(\lambda_d)$, provided that
\[
    D \ll_\eps x^{-\eps} \min \left(x^{1/2}, x^{(1-2\theta\alpha)/(2-5\theta)}\right).
\]
\end{proposition}

\begin{proof}
This is a refinement of Merikoski's \cite[Prop.\,1]{merikoski2023largest} (which explicitated the computations in de la Bret\`eche--Drappeau's \cite[\S 8.4]{de2020niveau}), using our \cref{lem:improvement-bd}.$(ii)$ instead of \cite[(8.7)]{de2020niveau}. Indeed, in the first display on \cite[p.\,1620]{de2020niveau}, by applying \cref{eq:bd-2} for $H \gets P X^{-1+\delta}$ (for $\delta = \delta(\eps)$ to be chosen shortly), $Q \gets D \le x^{1/2}$ and $M \gets P/D$, we instead obtain the bound
\[
\begin{aligned}
    &R_H(x, P, D)
    \\
    &\ll_{\delta}
    x^{1+O(\delta)} P^{-1} D H\left(
    H + \sqrt{\frac{P}{D}} \left(1 + H^{-2\theta} D^{\theta/2} \left(\frac{P}{D}\right)^{\theta}\right) + \sqrt{\frac{P}{H}}\left(1 + D^{-3\theta/2} \left(\frac{P}{D}\right)^{\theta} \right)\right)
    \\
    &= 
    x^{O(\delta)} \left( \frac{PD}{x} + \sqrt{PD} \left(1 + x^{2\theta} P^{-\theta} D^{-\theta/2}\right) + \sqrt{x}D \left(1 + D^{-5\theta/2} P^{\theta}\right) \right).
\end{aligned}
\]
Here, $R_H(x, P, D)$ resulted from our Type I sum after putting $d$ in dyadic ranges, expanding and Fourier-completing $|\mA_q|$; see \cite[\S 8.4]{de2020niveau} and then \cite[\S 4, 5]{deshouillers1982greatest}. Overall, this bound is acceptable in \cref{eq:type-I} (i.e., $\ll_\eps x^{1-\delta}$) provided that for an absolute constant $K$, one has
\[
\begin{aligned}
    D 
    &\ll_\eps x^{-K\delta} \min \left( x^2 P^{-1}, x^{1/2}, x^{2(1-2\theta)/(1-\theta)} P^{-(1-2\theta)/(1-\theta)}, x^{1/(2-5\theta)} P^{-2\theta/(2-5\theta)} \right)
    \\
    &=
    x^{-\eps} \min \left(x^{2-\alpha}, x^{1/2}, x^{(2-\alpha)(1-2\theta)/(1-\theta)}, x^{(1-2\theta\alpha)/(2-5\theta)}\right),
\end{aligned}
\]
where we picked $\delta := \eps/K$ and substituted $P = x^\alpha$. A quick numerical verification shows that for $1 \le \alpha \le 1.4$ and $\theta = \tfrac{7}{64}$, the first and the third term do not contribute to the minimum.
\end{proof}

\begin{proposition}[Type II estimate] \label{prop:type-II}
For any sufficiently small $\eps > 0$ there exists $\delta > 0$ such that the following holds. With \cref{not:set-up-info}, $\theta := \tfrac{7}{64}$, and $MN = P$ with $M, N \ge 1$, one has
\begin{equation}\label{eq:type-II}
    \sum_{\substack{m \sim M \\ n \sim N}} a_m b_n \left(|\mA_{mn}| - X \frac{\rho(mn)}{mn}\right) \Psi\left(\frac{mn}{P}\right) \log (mn) 
    \ll_\eps 
    x^{1-\delta},
\end{equation}
for any divisor-bounded coefficients $(a_m)$ and $(b_n)$, provided that one of the following holds: 
\begin{itemize} 
    \item[$(i)$.] $(b_n)$ is supported on square-free integers, and
    \begin{equation} \label{eq:type-II-ranges-1}
        x^{\alpha-1+\eps} \ll_\eps N \ll_\eps x^{-\eps}
        \max \left(x^{(2-(1+2\theta)\alpha)/(3-4\theta)}, x^{(2-\alpha)(1-2\theta)/(3-2\theta)}\right);
    \end{equation}
    \item[$(ii)$.] $(b_n)$ is supported on primes, and
    \begin{equation} \label{eq:type-II-ranges-2}
        x^{\alpha-1+\eps} \ll_\eps N \ll_\eps
        x^{(4-3\alpha)/3 - \eps}.
    \end{equation}
\end{itemize}
\end{proposition} 

\begin{remark}
The upper range in \cref{prop:type-II}.$(ii)$, which completely removes the dependency on Selberg's eigenvalue conjecture, wins over that in \cref{prop:type-I}.$(i)$ only for $\alpha < 136/129 \approx 1.054$. As in \cite{merikoski2023largest}, assuming Selberg's eigenvalue conjecture, the full admissible range in part $(i)$ is $N \ll_\eps x^{(2-\alpha)/3}$, which includes the range in part $(ii)$.
\end{remark}

\begin{proof}[Proof of \cref{prop:type-II}.$(ii)$, assuming $(i)$]
This is a refinement of Merikoski's \cite[Prop.\,4.(ii)]{merikoski2023largest}, using our \cref{lem:improvement-bd}.$(i)$ instead of de la Bret\`eche--Drappeau's bound \cite[(8.5)]{de2020niveau}.

We briefly recall that in \cite[\S 3]{merikoski2023largest}, Merikoski expanded and Fourier-completed $|\mA_{mn}|$ (resulting in a sum over $1 \le |h| \le H := P x^{-1+\delta}$), removed the smooth cross-conditions in $h, m, n$, and inserted the condition $(m, n) = 1$ to reach Type II sums $\Sigma(M, N)$. Then they applied Cauchy--Schwarz with the sum over $n$ inside, to obtain $\Sigma(M, N) \ll M^{1/2}\, \Xi(M, N)^{1/2}$, and trivially bounded the `diagonal' contribution of $n_1 = n_2$ using the condition $N \gg_\eps x^{2(\alpha-1)+\eps}$. To estimate the remaining sum $\Xi_0(M, N)$ from the second-to-last display in \cite[\S 3.10]{merikoski2023largest}, we apply our bound \cref{eq:bd-1} with $q \gets n_1 n_2$ and $h \gets h(n_1 - n_2)$; this gives the refined bound
\[
\begin{aligned}
    \Xi_0(M, N) 
    &\ll_\delta 
    x^{O(\delta)} \sum_{\substack{n_1, n_2 \sim N \\ (n_1, n_2) = 1}} \frac{1}{H} \sum_{1 \le |h| \le H} \left(HN + \sqrt{MN^2} \left(1 + (n_1n_2, h(n_1-n_2))^{\theta} N^{-3\theta} M^{\theta} \right)\right)
    \\
    &\ll_\delta 
    x^{O(\delta)} N^2\left(HN + M^{1/2} N + M^{(1+2\theta)/2} N^{1-3\theta} \right).
\end{aligned}
\]
This results in a contribution to $\Sigma(M, N)$ of
\[
\begin{aligned}
    &\ll_\delta x^{O(\delta)} M^{1/2} N \left(H^{1/2} N^{1/2} + M^{1/4} N^{1/2} + M^{(1+2\theta)/4} N^{(1-3\theta)/2} \right)
    \\
    &\ll x^{O(\delta)} \left(x^{-1/2} PN + P^{3/4} N^{3/4}  + P^{(3+2\theta)/4} N^{(3-8\theta)/4} \right),
\end{aligned}
\]
which is acceptable (i.e., $\ll_\eps x^{1-\delta}$) provided that for a large enough absolute constant $K$,
\[
    N \ll_\eps x^{-K\delta}
    \min \left(x^{3/2} P^{-1}, x^{4/3} P^{-1}, x^{4/(3-8\theta)} P^{-(3+2\theta)/(3-8\theta)}\right).
\]
Trivially removing the first term, picking $\delta := \eps/K$, and substituting $P = x^\alpha$, this proves \cref{eq:type-II} in the range
\[
    x^{2(\alpha-1)+\eps} \ll_\eps N \ll_\eps x^{-\eps}
    \min \left(x^{(4-3\alpha)/3}, x^{(4-(3+2\theta)\alpha)/(3-8\theta)} \right),
\]
when $(b_n)$ is supported on primes. The remaining ranges to consider are
\begin{equation} \label{eq:remaining-range-1}
    x^{\alpha - 1 + \eps} \ll_\eps N \ll_\eps \min\left( x^{2(\alpha-1)+\eps}, x^{(4-3\alpha)/3 - \eps}\right)
\end{equation}
and
\begin{equation} \label{eq:remaining-range-2}
    \min \left( x^{\alpha-1+\eps}, x^{(4-(3+2\theta)\alpha)/(3-8\theta)-\eps} \right)
    \ll_\eps N 
    \ll_\eps 
    x^{(4-3\alpha)/3 - \eps},
\end{equation}
both of which are (barely) covered by \cref{prop:type-II}.$(i)$. Indeed, for \cref{eq:remaining-range-1}, a quick numerical verification shows that 
\[
    \min\left(2(\alpha-1), \frac{4-3\alpha}{3}\right)
    <
    \frac{2-(1+2\theta)\alpha}{3-4\theta}
\]
for $\theta = \tfrac{7}{64}$ and all $\alpha$, the smallest gap being $\approx 0.07$, at $a = 10/9$. In \cref{eq:remaining-range-2}, we have a nontrivial range only when
\[
    \frac{4-(3+2\theta)\alpha}{3-8\theta} \le \frac{4-3\alpha}{3}
    \qquad 
    \iff 
    \qquad 
    \alpha \ge \frac{16}{15} \ge 1.066,
\]
and for such $\alpha$, we have
\[
    \frac{4-3\alpha}{3} < \frac{2-(1+2\theta)\alpha}{3-4\theta}.
\]
Therefore, \cref{eq:type-II} holds in the full range from \cref{eq:type-II-ranges-2}.
\end{proof}

\begin{remark}
As in \cite[\S 3.10]{merikoski2023largest}, the bound for $\Xi_0(M, N)$ in the proof above does not leverage any cancellation over $h$. One can attempt to do this using \cref{cor:kloosterman-averaging-mnc} with $a_m = e(m\alpha_q)$ and $b_n$ as in \cref{eq:third-type-disp}, but the gain in the $H$-aspect would be smaller than the loss in the $\theta$-aspect in our computations. This is because \cref{prop:type-II}.$(ii)$ is only relevant for $\alpha$ close to $1$, i.e., for small values of $H$.
\end{remark}

\begin{proof}[Proof of \cref{prop:type-II}.$(i)$]
This is a refinement of Merikoski's \cite[Prop.\,4.(i)]{merikoski2023largest}, using \cref{cor:kloosterman-averaging-mnc} (plus \cref{thm:large-sieve-dispersion-coeffs}) instead of Deshouillers--Iwaniec's bound \cite[Theorem 9]{deshouillers1982kloosterman}. 

We very briefly recall the relevant parts of Merikoski's argument and the sizes of the parameters therein, pointing the reader to \cite[\S 3]{merikoski2023largest} for details. In \cite[\S 3.4]{merikoski2023largest}, one expanded and Fourier-completed $|\mA_{mn}|$, resulting in a sum over $1 \le |h| \le H$ with
\begin{equation} \label{eq:H-value}
    H := P x^{-1+\delta},
\end{equation}
as before. Then, one removed the smooth cross-conditions in $h, m, n$, and separated $k = (m, n)$ to reach the type-II sums $\Sigma_k(M, N)$ from the first display on \cite[p.\,1275]{merikoski2023largest}; we need to bound these by $\ll_\eps x^{1-\delta}/k$, for $\delta = \delta(\eps)$ to be chosen. 

In \cite[\S 3.5]{merikoski2023largest}, one applied Cauchy--Schwarz keeping the sums over $h, n$ inside, to obtain 
\begin{equation} \label{eq:sigma-k-cauchy}
    \Sigma_k(M, N) \ll \left(\frac{M}{k}\right)^{1/2} \Xi_k(M, N)^{1/2},
\end{equation}
and trivially bounded the contribution of $h_2 n_1 = h_1 n_2$ to $\Xi_k$, using the condition $N \gg_\eps x^{\alpha-1+\eps}$; then they separated $n_0 = (n_1, n_2)$ (and let $n_i \gets n_i/n_0$). We note that considering nontrivial values of the GCD-parameters $k$ and $n_0$ was not necessary in the proof of \cref{prop:type-II}.$(ii)$, since then $(b_n)$ was supported on primes; in a first pass the reader can pretend that $k = n_0 = 1$.

In \cite[\S 3.6]{merikoski2023largest}, one expanded the condition $(m, n_0n_1n_2) = 1$ by M\"obius inversion, resulting in a sum over $d \mid n_0 n_1 n_2$ (we switched notation from $\delta$ to $d$). Then, one applied Gauss' lemma (\cite[Lemma 9]{merikoski2023largest}), resulting in sums $\Psi_k(R, S)$ of incomplete Kloosterman sums, ranging over $r, s$ of sizes 
\begin{equation} \label{eq:RS-values}
    1 \ll R, S \ll \sqrt{\frac{P N}{k n_0}}.
\end{equation}
In \cite[\S 3.7]{merikoski2023largest}, one completed Kloosterman sums, resulting in a sum over $|t| \le T$ with
\begin{equation} \label{eq:T-value}
    T = x^{\delta} \frac{Sd N^2}{Rn_0},
\end{equation}
and trivially bounded the contribution of $t = 0$. This ultimately leads to the sums of Kloosterman sums $\tilde{\Psi}_k(R, S)$
from \cite[p.\,1279]{merikoski2023largest}, which have a relevant \emph{level} of
\begin{equation} \label{eq:rho-value}
    \varrho := dk^2 n_0 n_1 n_2 \asymp \frac{dN^2}{n_0}.
\end{equation}
Finally, in \cite[\S 3.8]{merikoski2023largest}, Merikoski used \cite[Theorem 9]{deshouillers1982greatest} to bound the trilinear sums of Kloosterman sums
\[
    \mK = \mK(d, n_0, n_1, n_2) := \max_{\alpha \pmod{\varrho}} \left\vert 
    \sum_{m \sim \MM} a_m \sum_{n \sim \NN} b_n \sum_{(c, \varrho) = 1} \Phi\left(\frac{m}{\MM}, \frac{n}{\NN}, \frac{c}{\CC}\right) S\left(m \bar{\varrho}, \pm n; c\right) \right\vert,
\]
where $\Phi$ is a smooth function as in \cref{cor:kloosterman-averaging-mnc} with $Z = 1$, $(c_h)$ are bounded coefficients,
\begin{equation} \label{eq:am-bn-jori}
    a_m := e\left(-m\frac{\alpha}{\varrho}\right),
    \qquad\qquad 
    b_n := \sum_{\substack{h_1 \sim H_1 \\ h_2 \sim H_2 \\ n = h_1 n_2 - h_2 n_1}} c_{h_1} \bar{c_{h_2}},
\end{equation}
both of which depend on the level $\varrho$, 
\begin{equation} \label{eq:jori-mn-bounds}
    \MM \ll T, \qquad\qquad 
    \NN \ll \frac{HN}{kn_0}, 
    \qquad\qquad 
    \CC \ll S,
\end{equation}
and $1/2 \le H_2 \le H_1 \le H$.
We will achieve better bounds for $\mK$ by leveraging the structure of the coefficients $(a_m)$ and $(b_n)$. To do so, we note that the coefficients $c_h$ (obtained by removing the cross-condition in $h, m, n$ on \cite[p.\,1274]{merikoski2023largest}) are smooth functions of $h$. In fact, expanding $|\mA_{mn}| - \frac{\rho(mn)}{mn}$ via \cref{lem:truncated-poisson} and fixing $j, u$ up to a logarithmic loss, we can use the coefficients
\[
    c_h := \Psi_j\left(\frac{|h|}{H_j}\right) e\left(-h\frac{ux}{P}\right),
\]
from \cref{eq:h-coeffs-poisson}, where $1 \le 2^j = H_j \le H = P x^{-1+\delta}$, $u \asymp 1$, and $\Psi_j : (\frac{1}{2}, 2) \to \C$ are compactly-supported smooth functions with bounded derivatives. In particular, through \cref{lem:truncated-poisson} we put $|h|$ in (smooth) dyadic ranges, and then separate into positive and negative values of $h$, all before applying Cauchy--Schwarz; so the resulting variables $h_1, h_2$ are of the same size. The coefficients $(b_n)$ from \cref{eq:am-bn-jori} become
\[
    b_n := \sum_{\substack{h_1, h_2 \in \Z \\ n = h_1 n_2 - h_2 n_1}} c_{h_1} \bar{c_{h_2}},
\]
which are in a suitable form to use \cref{thm:large-sieve-dispersion-coeffs} (see also \cref{eq:sequence-choice-2}), with $a = 1$, $H = H_j$, $\alpha_i = \pm ux/P \ll x^\delta H^{-1}$, and 
\begin{equation} \label{eq:L-value}
    L := \frac{N}{k n_0} \asymp n_1 \asymp n_2.
\end{equation}
In particular, since $\varrho \ge n_1 n_2 \asymp L^2$, the tuple $(\varrho, \NN, x, (b_n)_{n \sim \NN}, A_{\NN}, Y_{\NN})$ satisfies \cref{ass:large-sieve} with
\begin{equation} \label{eq:jori-savings-n}
    Y_{\NN} := \max\left(1, \frac{\NN H_j}{(H_j+L)L x^\delta} \right)
    \qquad\quad 
    \text{and}
    \qquad\quad
    A_{\NN} := \|b_n \one_{n \sim \NN}\|_2 + \sqrt{\NN}\sqrt{\frac{H_j}{L} + \frac{H_j^2}{L^2}},
\end{equation}
where we used that $T_{\NN/L}(\alpha_i) \ll T_H(\alpha_i) \ll 1 + H |\alpha_i| \ll x^\delta$. On the other hand, by \cref{thm:large-sieve-expo-phases} (see also \cref{eq:sequence-choice-1}), the tuple $(\varrho, \MM, x, (a_m)_{m \sim \MM}, A_{\MM}, Y_{\MM})$ satisfies \cref{ass:large-sieve}, with
\begin{equation} \label{eq:jori-savings-m}
    Y_{\MM} := \sqrt{\MM} 
    \qquad\quad 
    \text{and}
    \qquad\quad
    A_{\MM} := \sqrt{\MM}.
\end{equation}
By \cref{cor:kloosterman-averaging-mnc}, specifically \cref{eq:kloosterman-mnc-explicit}, it follows that
\[
    \mK \ll_\delta x^{O(\delta)} 
    \left(1 + \frac{\CC}{\sqrt{\varrho Y_{\MM} Y_{\NN}}} \right)^{2\theta}
    A_{\MM} A_{\NN}
    \frac{\left(\sqrt{\varrho}\CC + \sqrt{\MM\NN} + \sqrt{\MM} \CC\right)\left(\sqrt{\varrho}\CC + \sqrt{\MM \NN} + \sqrt{\NN} \CC\right)}{\sqrt{\varrho}\CC + \sqrt{\MM \NN}},
\]
and substituting \cref{eq:jori-savings-m,eq:jori-savings-n} gives
\begin{equation} \label{eq:jori-K-bound}
\begin{aligned}
    \mK \ll_\delta x^{O(\delta)}
    \left(1 + \frac{\CC}{\sqrt{\varrho} \MM^{1/4} \max\left(1, \sqrt{\frac{\NN H_j}{(H_j+L)L}}\right)} \right)^{2\theta}
    \sqrt{\MM} \left(\|b_n \one_{n \sim \NN}\|_2 + \sqrt{\NN} \sqrt{\frac{H_j}{L} + \frac{H_j^2}{L^2}}\right)
    \\
    \times
    \frac{\left(\sqrt{\varrho}\CC + \sqrt{\MM\NN} + \sqrt{\MM} \CC\right)\left(\sqrt{\varrho}\CC + \sqrt{\MM \NN} + \sqrt{\NN} \CC\right)}{\sqrt{\varrho}\CC + \sqrt{\MM \NN}}.
\end{aligned}
\end{equation}
Since $\NN \ll \varrho$ (which follows from $H \ll N$), the term on the second line of \cref{eq:jori-K-bound} is at most $\ll \sqrt{\rho} \CC + \sqrt{\MM} \CC + \sqrt{\MM \NN}$, as in \cite[p.\,1280]{merikoski2023largest}. The resulting bound is non-decreasing in $\MM, \CC$, so we can plug in their upper bounds from \cref{eq:jori-mn-bounds}, as well as \cref{eq:T-value,eq:rho-value} to obtain
\[
\begin{aligned}
    \sum_{d \mid n_0 n_1 n_2}
    \frac{1}{TH^2} \mK(d, n_0, n_1, n_2)
    \ll_\delta 
    x^{O(\delta)}
    \max_{d \ge 1}
    \frac{Rn_0}{SdN^2 H^2}
    \left(1 + \frac{S \min\left(1, \sqrt{\frac{(H_j+L)L}{\NN H_j}}\right)}{\sqrt{\frac{dN^2}{n_0}} \left(\frac{SdN^2}{Rn_0}\right)^{1/4} } \right)^{2\theta}
    \\ 
    \times
    \sqrt{\frac{SdN^2}{Rn_0}} 
    \left(\|b_n \one_{n \sim \NN}\|_2 + \sqrt{\NN} \sqrt{\frac{H_j}{L} + \frac{H_j^2}{L^2}}\right)
    \\
    \times 
    \left(\sqrt{\frac{dN^2}{n_0}} S + \sqrt{\frac{SdN^2}{Rn_0}} S + \sqrt{\frac{SdN^2}{Rn_0}\NN}\right),
\end{aligned}
\]
where none of the remaining variables have implicit dependencies on $d$. The right-hand side is seen to be non-increasing in $d$, so we can plug in $d = 1$ for an upper bound. Moreover, when summing over $n_1, n_2 \sim L = N/(kn_0)$, we have the same bound as in \cite[bottom of p.\,1280]{merikoski2023largest} (by \cite[Lemma 7]{merikoski2023largest}) for the contribution of $A_{\NN}$:
\[
    \sum_{n_1, n_2 \sim L} \left(\|b_n \one_{n \sim \NN}\|_2 + \sqrt{\NN} \sqrt{\frac{H_j}{L} + \frac{H_j^2}{L^2}} \right) 
    \ll 
    \sqrt{\NN} \max\left(H_jL, H_j^{1/2} L^{3/2}\right).
\]
The resulting bound for $\sum_{n_1, n_2 \sim L} \sum_{d \mid n_0 n_1 n_2} \frac{1}{TH^2} \mK(d, n_0, n_1, n_2)$ is non-decreasing in $\NN, H_j$, so we can plug in the upper bounds in $\NN \ll HL$ and $H_j \ll H$ and simplify the resulting expression to obtain
\[
\begin{aligned}
    \sum_{n_1, n_2 \sim L}
    \sum_{d \mid n_0 n_1 n_2}
    \frac{1}{TH^2} \mK(d, n_0, n_1, n_2)
    \ll_\delta 
    x^{O(\delta)}
    \left(1 + \frac{S^{3/4} R^{1/4} n_0^{3/4}}{N^{3/2}} \min\left(1, \frac{\sqrt{H+L}}{H}\right) \right)^{2\theta}
    \\
    \times
    \max\left(H^{1/2} L^{3/2}, L^2\right)
    \left(\frac{\sqrt{RS}}{H} + \frac{S}{H} + \sqrt{\frac{L}{H}}\right).
\end{aligned}
\]
Summing over $n_0$ and plugging in the bounds for $R, S, L$ from \cref{eq:RS-values,eq:L-value}, we get
\[
\begin{aligned}
    \Upsilon_k 
    &:= 
    \sum_{n_0 \ll N} \rho(n_0) \sum_{n_1, n_2 \sim L} \sum_{d \mid n_0 n_1 n_2} \frac{1}{TH^2} \mK(d, n_0, n_1, n_2)
    \\
    &\ll_\delta 
    x^{O(\delta)}
    \sum_{n_0 \ll N} 
    \left(1 + \frac{\sqrt{PN} n_0^{1/4}}{\sqrt{k} N^{3/2}} \min\left(1, \frac{\sqrt{H+\frac{N}{kn_0}}}{H}\right) \right)^{2\theta}
    \\
    &\times
    \max\left(H^{1/2} \left(\frac{N}{kn_0}\right)^{3/2}, \left(\frac{N}{kn_0}\right)^2\right)
    \left(\frac{\sqrt{PN}}{\sqrt{k n_0} H} + \sqrt{\frac{N}{kn_0 H}}\right).
\end{aligned}
\]
Using that $H \ll N$, this further yields
\[
\begin{aligned}
    \Upsilon_k 
    &\ll_\delta 
    \frac{x^{O(\delta)}}{k}
    \sum_{n_0 \ll N} 
    \frac{1}{n_0^{2-(\theta/2)}}
    \left(1 + \frac{\sqrt{PN}}{N^{3/2}} \min\left(1, \frac{\sqrt{N}}{H}\right) \right)^{2\theta}
    N^2
    \left(\frac{\sqrt{PN}}{H} + \sqrt{\frac{N}{H}}\right)
    \\
    &\ll_\delta
    \frac{x^{O(\delta)}}{k} 
    \left(1 + \min\left(\frac{\sqrt{P}}{N}, \frac{\sqrt{P}}{\sqrt{N} H}\right) \right)^{2\theta} 
    \left( 
    \frac{\sqrt{P} N^{5/2}}{H} + \frac{N^{5/2}}{\sqrt{H}}
    \right).
\end{aligned}
\]
Since we have $N \le \sqrt{x} \le \sqrt{P}$ and $\sqrt{N}H \le x^{1/4} P x^{-1+\delta} \le x^{1/4 + 3/4 -1 + \delta} \sqrt{P}$ for the ranges in \cref{prop:type-II}.$(i)$, we may ignore the $1$-term in the $\theta$-factor; plugging in \cref{eq:H-value}, we obtain
\[
    \Upsilon_k \ll_\delta 
    \frac{x^{1+O(\delta)} N^{5/2}}{k \sqrt{P}} 
    \min\left(\frac{\sqrt{P}}{N}, \frac{x}{\sqrt{NP}}\right)^{2\theta},
\]
which improves \cite[(3.7)]{merikoski2023largest}. In light of \cref{eq:sigma-k-cauchy} and $MN = P$, this gives a contribution to $\Sigma_k(M, N)$ of
\[
    \ll_\delta 
    \frac{x^{1/2+O(\delta)}}{k}
    P^{1/4} N^{3/4}
    \min\left(\frac{\sqrt{P}}{N}, \frac{x}{\sqrt{NP}}\right)^{\theta},
\]
which is acceptable (i.e., $\ll_\eps x^{1-\delta}/k$) provided that for a large enough absolute constant $K$,
\[
    N \ll_\eps x^{-K\delta} \max\left(x^{2/(3-4\theta)} P^{-(1+2\theta)/(3-4\theta)}, x^{2(1-2\theta)/(3-2\theta)} P^{-(1-2\theta)/(3-2\theta)} \right).
\]
Choosing $\delta := \eps/K$ and substituting $P = x^\alpha$ completes our proof.
\end{proof}

\subsection{Sieve computations} \label{subsec:harman-sieve}
To complete the proof of \cref{thm:appl-greatest-prime}, it remains to adapt the calculation in \cite[\S 2]{merikoski2023largest} with our Type I and Type II information.

\begin{notation}[Set-up for sieve computations] \label{not:set-up-harman}
Further to \cref{not:set-up-info}, we follow \cite[p.\,1257]{merikoski2023largest} and let $P_x := P^+\left(\prod_{x\le n \le 2x}(n^2+1)\right)$, then use a smooth dyadic partition of unity to split
\[
    S(x) := \sum_{\substack{x < p \le P_x \\ p \text{ prime}}} |\mA_p| \log p
\]
into a sum over $x \le P \le P_x$, $P = P_j = 2^j x$ of
\[
    S(x, P) := \sum_{p \text{ prime}} \Psi_j \left(\frac{p}{P}\right) |\mA_p| \log p,
\]
up to an error of $O(x)$. Here $\Psi_j$ are smooth functions supported on $[1, 4]$, with $\Psi_j^{(k)} \ll_k 1$ for all $k \ge 0$. Following \cite[p.\,1259]{merikoski2023largest}, given $z \ge 1$ and $u \in \Z_+$, we also let $P(z) := \prod_{\text{prime } p < z} p$ and
\[
    S(\mA(P)_u, z) := \sum_{(n, P(z)) = 1} |\mA_{un}| \Psi\left(\frac{un}{P}\right) \log(un),
\]
so that $S(x, P) = S(\mA(P), 2\sqrt{P})$ (where dropping the $u$ index means that $u = 1$). This has a corresponding main term of
\[
    S(\mB(P)_u, z) := X \sum_{(n, P(z)) = 1} \frac{\rho(un)}{un} \Psi\left(\frac{un}{P}\right) \log(un),
\]
sums of which can be computed via \cite[Lemma 1]{merikoski2023largest}. Finally, the linear sieve upper bound will require the solutions $F(s), f(s)$ to the delay-differential equation system from \cite[p.\,1263]{merikoski2023largest}, while the Harman sieve computations will require the Buchstab function $\omega(u)$, bounded as in \cite[(2.5)]{merikoski2023largest}.
\end{notation}

As in \cite[p.\,1257]{merikoski2023largest}, we aim to find the greatest $\bar{\omega}$ for which 
\begin{equation} \label{eq:harman-final-bound}
    \sum_{\substack{x \le P \le x^{\bar{\omega}} \\ P = 2^j x}} S(x, P) \le (1-\eps)\, X \log x.
\end{equation}
Since $S(x) = X \log x + O(x)$ (see \cite[(2.1)]{merikoski2023largest}), this will imply the lower bound $P_x \ge x^{\bar{\omega}}$.

\begin{lemma}[Linear sieve upper bound] \label{lem:linear-sieve}
For any $\eps > 0$ there exists $\delta > 0$ such that the following holds.  With $\theta := \tfrac{7}{64}$, \cref{not:set-up-info}, \cref{not:set-up-harman}, and $D := x^{-\eps} \min\left(x^{1/2}, x^{(1-2\theta \alpha)/(2-5\theta)}\right)$, one has
\[
    S(\mA(P), z) \le (1 + \delta)\, X \int \Psi\left(\frac{u}{P}\right) \frac{\alpha \log x}{e^\gamma \log z} F\left(\frac{\log D}{\log z}\right) \frac{du}{u},
\]
for any $x^\eps < z < D$, where $\gamma$ is the Euler--Mascheroni constant.
\end{lemma}

\begin{proof}
This is just \cite[Lemma 2]{merikoski2023largest} with the updated parameter $D$ from our Type I information (\cref{prop:type-I}).
\end{proof}

\begin{proposition}[Asymptotics for Harman sieve sums]
\label{prop:asymptotics-harman}
For any $\eps > 0$ there exists $\delta > 0$ such that the following hold.
With $\theta := \tfrac{7}{64}$, \cref{not:set-up-info} and \cref{not:set-up-harman},
let 
\[
    D := x^{-\eps} \min\left(x^{1/2}, x^{(1-2\theta \alpha)/(2-5\theta)}\right),
    \qquad\qquad 
    U := Dx^{1-\alpha-\eps} =: x^\xi,
\] 
and $(\lambda_u)$ be divisor-bounded coefficients. Also, let
\begin{equation} \label{eq:sigma-0}
    \sigma_0 := \max \left(\frac{2-(1+2\theta)\alpha}{3-4\theta}, \frac{(1-2\theta)(2-\alpha)}{3-2\theta}\right)
\end{equation}
be the exponent from \cref{prop:type-II}.$(i)$.
\begin{itemize}
    \item[$(i)$.] For $1 \le \alpha < 228/203 - O(\eps)$ and 
    \begin{equation} \label{eq:sigma-exponent}
        \sigma := \max \left(\frac{4-3\alpha}{3}, \sigma_0 \right) - \eps,
    \end{equation}
    one has
    \[
        \sum_{u \le U} \lambda_u \left( S\left(\mA(P)_u, x^\sigma\right) -
        S\left(\mB(P)_u, x^\sigma\right)\right)
        \ll_\eps x^{1-\delta}.
    \]
    \item[$(ii)$.] For $1 \le \alpha < 139/114 -O(\eps)$ and
    \begin{equation} \label{eq:gamma-exponent}
        \gamma := \sigma_0 - (\alpha-1) - 2\eps,
    \end{equation}
    one has
    \[
        \sum_{u \le U} \lambda_u \left( S\left(\mA(P)_u, x^\gamma\right) -
        S\left(\mB(P)_u, x^\gamma\right)\right)
        \ll_\eps x^{1-\delta}.
    \]
\end{itemize}

\end{proposition}
\begin{proof}
These are just \cite[Propositions 3 and 4]{merikoski2023largest}, adapted with our Type II information from \cref{prop:type-II}; the additional term of $(4-3\alpha)/3$ from \cref{eq:sigma-exponent} comes from \cref{prop:type-II}.$(ii)$. We note that the proof of \cite[Proposition 3]{merikoski2023largest} requires
\[
    2(\alpha - 1) < \sigma_0 - O(\eps) = \max \left(\frac{2-(1+2\theta)\alpha}{3-4\theta}, \frac{(1-2\theta)(2-\alpha)}{3-2\theta}\right) - O(\eps),
\]
which happens for $\alpha < 228/203 - O(\eps)$. Similarly, the proof of \cite[Proposition 4]{merikoski2023largest} requires
\[
    \alpha - 1 < \sigma_0 - O(\eps) = \max \left(\frac{2-(1+2\theta)\alpha}{3-4\theta}, \frac{(1-2\theta)(2-\alpha)}{3-2\theta}\right) - O(\eps),
\]
which happens for $\alpha < 139/114 - O(\eps)$.
\end{proof}

We are now ready to prove our \cref{thm:appl-greatest-prime}, in a very similar manner to \cite[\S 2.6]{merikoski2023largest}.

\begin{proof}[Proof of \cref{thm:appl-greatest-prime}]
We follow the Harman sieve computations in \cite[\S 2.4]{merikoski2023largest}, applying Buchstab's identity in the same ways (with adapted ranges corresponding to the values of $D, U, \sigma_0, \sigma, \gamma, \xi$ from \cref{lem:linear-sieve,prop:asymptotics-harman}). The five ranges relevant in the proof are now $\alpha < 25/24$, $25/24 \le \alpha < 228/203$, $228/203 \le \alpha < 7/6$, $7/6 \le \alpha < 139/114$, and $\alpha \ge 139/114$. Here, the values $228/203$ and $139/114$ come from \cref{prop:asymptotics-harman}, while $25/24$ and $7/6$ are the thresholds deciding the inequalities $\alpha < \xi + 2\sigma$, respectively $2(\alpha-1) < \xi$, up to $o(1)$ factors. Indeed, we recall that
\[
    \xi = \min\left(\frac{1}{2}, \frac{(1-2\theta \alpha)}{2-5\theta}\right) - (\alpha-1) - 2\eps,
\]
and only the first term in the minimum is relevant for the aforementioned inequalities. We thus obtain
\[
    \sum_{\substack{x \le P \le x^{139/114} \\ P = 2^j x}} S(x, P) \le \left(\frac{7}{6}-1 + G_1 + G_2 + G_3 + G_4 + G_5 - G_6 + o(1)\right) X \log x,
\]
where 
\[
\begin{aligned}
    G_1 &:= \int_1^{25/24} \alpha \left(\int_\sigma^{\alpha-2\sigma} \omega\left(\frac{\alpha}{\beta}-1\right) \frac{d\beta}{\beta^2}
    +
    \int_\xi^{\alpha/2} \omega\left(\frac{\alpha}{\beta}-1\right) \frac{d\beta}{\beta^2}\right) d\alpha < 0.02093,
    \\
    G_2 &:= \int_{25/24}^{228/203} \alpha \int_\sigma^{\alpha/2} \omega\left(\frac{\alpha}{\beta}-1\right) \frac{d\beta}{\beta^2} d\alpha < 0.10528,
    \\
    G_3 &:= \int_{228/203}^{7/6} \alpha \int_{\sigma_0}^{\alpha/2} \omega\left(\frac{\alpha}{\beta}-1\right) \frac{d\beta}{\beta^2} d\alpha < 0.07319,
    \\
    G_4 &:= \int f_4\left(\alpha, \vec{\beta}\right) \alpha\, \omega\left(\frac{\alpha-\beta_1-\beta_2-\beta_3}{\beta_3}\right) \frac{d\beta_1\, d\beta_2\, d\beta_3}{\beta_1 \beta_2 \beta_3^2} d\alpha < 0.00163,
    \\
    G_5 &:= 4 \int_{7/6}^{139/114} \alpha\, d\alpha < 0.25116,
    \\
    G_6 &:= \int_{7/6}^{139/114} \alpha \int_{\alpha-1}^{\sigma_0} \omega\left(\frac{\alpha}{\beta} - 1\right) \frac{d\beta}{\beta^2} d\alpha > 0.02789.
\end{aligned}
\]
Here, $f_4$ denotes the characteristic function of the set
\[
\begin{aligned}
    \Big \{\frac{228}{203} < \alpha < \frac{7}{6},\ &\gamma < \beta_3 < \beta_2 < \beta_1 < \alpha-1,
    \\
    &\beta_1 + \beta_2, \beta_1 + \beta_3, \beta_2 + \beta_3, \beta_1 + \beta_2 + \beta_3 \not\in [\alpha-1, \sigma_0] \Big \}.
\end{aligned}
\]
We computed the integrals $G_i$ (for $i \neq 5$) by directly adapting the ranges in Merikoski's Python 3.7 code files (see \cite[p.\,1268]{merikoski2023largest}). In the expression for $G_5$, we implicitly used the value $D = x^{1/2-\eps}$ since $\frac{1}{2} < \frac{(1-2\theta\alpha)}{2-5\theta}$ for $\alpha \le 139/114$, and the fact that $1 < 1/(2(\alpha-1)) \le 3$ for $7/6 \le \alpha \le 139/114$. Thus
\[
    \sum_{\substack{x \le P \le x^{139/114} \\ P = 2^j x}} S(x, P)
    < 
    0.59097\, X \log x.
\]
For the remaining range $\alpha \ge 139/114$, we apply \cref{lem:linear-sieve} to obtain (as in \cite[(2.8)]{merikoski2023largest})
\[
    \sum_{\substack{x^{139/114} \le P \le x^{\bar{\omega}} \\ P = 2^j x}} 
    S(x, P) 
    \le
    \left(4 \int_{139/114}^{1.25} \alpha\, d\alpha + 
    (4-10\theta) \int_{1.25}^{\bar{\omega}} \frac{\alpha}{1-2\theta\alpha} d\alpha \right) X \log x,
\]
where $\alpha = 1.25 = 5/4$ is the threshold at which the expression for $D$ changes (i.e., when $\frac{1}{2} = \frac{(1-2\theta\alpha)}{2-5\theta}$). We conclude that \cref{eq:harman-final-bound} holds (for small enough $\eps$) provided that
\[
    (4-10\theta) \int_{1.25}^{\bar{\omega}} \frac{\alpha}{1-2\theta\alpha} d\alpha 
    < 
    1 
    -
    0.59097
    -
    4 \int_{139/114}^{1.25} \alpha\, d\alpha,
\]
where the right-hand side is at least $0.257406$.
This inequality (barely) holds true when $\bar{\omega} = 1.30008$, which proves \cref{thm:appl-greatest-prime}.
\end{proof}

\bibliographystyle{plain}
\bibliography{main}

\end{document}

%% file: preamble.tex
\usepackage[english]{babel}
\usepackage[utf8]{inputenc}
\usepackage[T1]{fontenc}

\usepackage[margin=1in, letterpaper]{geometry}
\usepackage{amssymb,amsmath,amsfonts,amsthm}
\usepackage{bm,bbm,mathrsfs}
\usepackage[dvipsnames]{xcolor}
\usepackage{url}
\usepackage{enumitem}
\usepackage{mathtools}
\usepackage{placeins}
\usepackage{extarrows}
\usepackage{subfig}
\usepackage[font=small,labelfont=bf]{caption}
\usepackage{floatrow}
\usepackage{verbatim}
\usepackage{wrapfig}
\usepackage{tikz}
\usetikzlibrary{arrows,arrows.meta,cd,decorations.pathmorphing,backgrounds,positioning,fit,matrix}
\usepackage{xcolor,colortbl,graphics,graphicx}
\usepackage[colorlinks=true, allcolors=Blue, pdfencoding=auto]{hyperref}
\usepackage[noabbrev,capitalize]{cleveref}
\crefname{equation}{}{}

\newtheorem{theorem}{Theorem}%[section]
\newtheorem{proposition}[theorem]{Proposition}
\newtheorem{lemma}[theorem]{Lemma}
\newtheorem{corollary}[theorem]{Corollary}

\Crefname{problem}{Problem}{Problems}
\newtheorem*{question*}{Question}
\Crefname{question}{Question}{Questions}
\newtheorem{assumption}[theorem]{Assumption}
\newtheorem{knowntheorem}{Theorem}

\newtheorem{knowncorollary}[knowntheorem]{Corollary}

\newtheorem{knownproposition}[knowntheorem]{Proposition}

\newtheorem{knownlemma}[knowntheorem]{Lemma}

\Crefname{knownlemma}{Lemma}{Lemmas}
\Crefname{knowntheorem}{Theorem}{Theorems}

\newtheorem{notation}[theorem]{Notation}
\Crefname{notation}{Notation}{Notations}
\newtheorem*{notation*}{Notation}

\newtheorem*{example*}{Example}

\theoremstyle{remark}
\newtheorem*{remark}{Remark}

\numberwithin{equation}{section}
\allowdisplaybreaks

\newcommand{\Z}{\mathbb{Z}}
\newcommand{\R}{\mathbb{R}}
\newcommand{\Q}{\mathbb{Q}}
\newcommand{\C}{\mathbb{C}}

\renewcommand{\H}{\mathbb{H}}
\newcommand{\MM}{\mathcal{M}}
\newcommand{\NN}{\mathcal{N}}
\newcommand{\CC}{\mathcal{C}}

\newcommand{\msI}{\mathscr{I}}

\newcommand{\mA}{\mathcal{A}}
\newcommand{\mB}{\mathcal{B}}
\newcommand{\mC}{\mathcal{C}}
\newcommand{\mD}{\mathcal{D}}
\newcommand{\mE}{\mathcal{E}}

\newcommand{\mH}{\mathcal{H}}
\newcommand{\mI}{\mathcal{I}}
\newcommand{\mK}{\mathcal{K}}
\newcommand{\mM}{\mathcal{M}}

\newcommand{\mS}{\mathcal{S}}

\newcommand{\ma}{\mathfrak{a}}
\newcommand{\mb}{\mathfrak{b}}
\newcommand{\mc}{\mathfrak{c}}

\newcommand{\one}{\mathbbm{1}}
\renewcommand{\Re}{\mathrm{Re}}
\renewcommand{\Im}{\mathrm{Im}}
\newcommand{\GL}{\mathrm{GL}}
\newcommand{\SL}{\mathrm{SL}}
\newcommand{\PSL}{\mathrm{PSL}}
\newcommand{\Id}{\mathrm{Id}}
\newcommand{\la}{\left\langle}
\newcommand{\ra}{\right\rangle}
\newcommand{\eps}{\varepsilon}
\newcommand{\sgn}{\mathrm{sgn}}
\renewcommand{\bar}{\overline}
\renewcommand{\hat}{\widehat}
\renewcommand{\tilde}{\widetilde}
\renewcommand{\mod}[1]{\ \textnormal{mod } #1}
\renewcommand{\pmod}[1]{\ (\textnormal{mod } #1)}
\newcommand{\lcm}{\mathrm{lcm}}

\newcommand{\ch}{\cosh}
\newenvironment{psmall}
  {\left(\begin{smallmatrix}}
  {\end{smallmatrix}\right)}
\newcommand{\twoline}[2]{$\substack{\text{\normalsize \vphantom{y}#1} \\ \text{\normalsize \vphantom{y}#2}}$}
\definecolor{myBlue}{rgb}{0, 0, 0.6}